\documentclass{amsart}

\usepackage[left=1in,right=1in,top=1in,bottom=1in]{geometry}
\usepackage{tikz}
\usepackage{tikz-cd}
\usetikzlibrary{shapes,decorations,calc,arrows}
\usepackage[foot]{amsaddr}
\usepackage{amsthm}
\usepackage{amsfonts}
\usepackage{amsmath}
\usepackage{amssymb}
\usepackage{mathtools}
\usepackage{url}
\usepackage[hidelinks]{hyperref}
\usepackage{graphicx}
\usepackage{cite}
\usepackage{titlesec}
\usepackage{enumerate}
\usepackage{MnSymbol}
\usepackage{xcolor}

\usepackage{todonotes}

\usepackage{titletoc}
\dottedcontents{section}[1em]{}{1em}{1pc}
\dottedcontents{subsection}[5em]{}{2em}{1pc}

\titleformat{\section}[hang]
{\normalfont\Large\bfseries}
{\thesection.}{0.5em}{}

\titlespacing*{\section}{0pc}{2pc}{0.25pc}

\titleformat{\subsection}[hang]
{\normalfont\large\bfseries}
{\thesubsection}{0.5em}{}

\titlespacing{\subsection}{0pc}{1.5pc}{0.5pc}

\newcommand{\E}{\mathbb{E}}
\newcommand{\N}{\mathbb{N}}
\newcommand{\Z}{\mathbb{Z}}

\newcommand{\R}{\mathbb{R}}
\newcommand{\C}{\mathbb{C}}

\newcommand{\cA}{\mathcal{A}}

\newcommand{\mb}[1]{\mathbb{#1}}

\DeclareMathOperator{\id}{id}
\DeclareMathOperator{\Tr}{Tr}
\DeclareMathOperator{\tr}{tr}

\DeclareMathOperator{\sa}{sa}
\DeclareMathOperator{\op}{op}

\DeclareMathOperator{\Cov}{Cov}
\DeclareMathOperator{\Lip}{Lip}
\DeclareMathOperator{\diag}{diag}
\DeclareMathOperator{\Spec}{Spec}
\DeclareMathOperator{\free}{free}
\DeclareMathOperator{\Ad}{Ad}
\DeclareMathOperator{\supp}{supp}

\newcommand{\<}{\left\langle}
\renewcommand{\>}{\right\rangle}
\DeclarePairedDelimiter{\ip}{\langle}{\rangle}
\DeclarePairedDelimiter{\norm}{\lVert}{\rVert}

\newcommand{\vi}{\vec{\imath}}

\numberwithin{equation}{section}

\newtheorem{thm}{Theorem}[section]
\newtheorem{thmalpha}{Theorem}

\newtheorem{coralpha}[thmalpha]{Corollary}
\newtheorem{prop}[thm]{Proposition}
\newtheorem{lem}[thm]{Lemma}
\newtheorem{clm}[thm]{Claim}
\newtheorem*{lem*}{Lemma}
\newtheorem{cor}[thm]{Corollary}

\theoremstyle{definition}
\newtheorem{defi}[thm]{Definition}
\newtheorem{nota}[thm]{Notation}
\newtheorem{ex}[thm]{Example}

\newtheorem{rem}[thm]{Remark}
\newtheorem*{defi*}{Definition}

\begin{document}
	
	\title{Strong convergence to operator-valued semicirculars}
	
	\author{David Jekel$^\clubsuit$, Yoonkyeong Lee$^\diamondsuit$, Brent Nelson$^\heartsuit$, Jennifer Pi$^\spadesuit$}
	
	\address{$^\clubsuit$Department of Mathematical Sciences, University of Copenhagen \hfill \url{daj@math.ku.dk}}
	
	\address{$^\diamondsuit$Department of Mathematics, Michigan State University \hfill \url{leeyoo16@msu.edu}}
	
	\address{$^\diamondsuit$ \emph{Current address:} Department of Mathematics, Texas A \& M University}
	
	\address{$^\heartsuit$Department of Mathematics, Michigan State University \hfill \url{brent@math.msu.edu}}
	
	\address{$^\spadesuit$Mathematical Institute, University of Oxford \hfill \url{pi@maths.ox.ac.uk}}
	
	\begin{abstract}
		We establish a framework for weak and strong convergence of matrix models to operator-valued semicircular systems parametrized by operator-valued covariance matrices $\eta = (\eta_{i,j})_{i,j \in I}$.  Non-commutative polynomials are replaced by \emph{covariance polynomials} that can involve iterated applications of $\eta_{i,j}$, leading to the notion of \emph{covariance laws}.  We give sufficient conditions for weak and strong convergence of general Gaussian random matrices and deterministic matrices to a $B$-valued semicircular family and generators of the base algebra $B$.  In particular, we obtain operator-valued strong convergence for continuously weighted Gaussian Wigner matrices, such as Gaussian band matrices with a continuous cutoff, and we construct natural strongly convergent matrix models for interpolated free group factors.
	\end{abstract}
	
	\maketitle
	
	\section*{Introduction} 
	For many self-adjoint random matrix tuples $X^{(n)} = (X_i^{(n)})$, free probability allows one to describe the limit 
	% (in expectation and often almost surely)
	of $\tr_n[p(X^{(n)})]$ for non-commutative polynomials $p$, where $\tr_n$ is the normalized trace on $n \times n$ matrices.   If $(M,\tau)$ is a tracial von Neumann algebra and $X\subset M$, then the random matrix models $X^{(n)}$ are said to \emph{converge weakly to $X$} if $\lim_{n \to \infty} \tr_n(p(X^{(n)}) = \tau(p(X))$ for all $p$ (since $X^{(n)}$ is random, weak convergence can be said to occur in expectation, in probability, or almost surely). In some cases, one can even asymptotically describe the operator norm $\norm{p(X^{(n)})}$ (equivalently, the largest singular value).  We say that $X^{(n)}$ \emph{converges strongly} to $X$ if in addition to weak convergence, $\lim_{n \to \infty} \norm{p(X^{(n)})} \to \norm{p(X)}$ for all $p$ (again, one can specify almost surely, in probability, or in expectation).  Strong convergence was first studied by Haagerup and Thorbj{\o}rnsen \cite{HT2005} for applications in $\mathrm{C}^*$-algebras, and their paper has inspired many subsequent works.  Recently, a striking development due to Hayes \cite{HayesPT} showed that the Peterson--Thom conjecture for the free group von Neumann algebras would follow from a certain conjecture on strong convergence for tensor products of Gaussian matrices, which was then proved by several independent breakthroughs in random matrix theory \cite{BelCap2022,BorCol2023,dlSM2024,Parraud2024tensor,CGVvH2024strong2}.
	
	Our goal is to study weak and strong convergence in the setting of operator-valued free probability, which considers free independence with respect to a non-commutative conditional expectation onto a von Neumann subalgebra $B$, rather than a tracial state onto $\C$ (see \cite{Voiculescu1985,Voiculescu1995,Speicher1998}).  Operator-valued free probability also arises from studying block decompositions of random matrices as well as Gaussian band matrices \cite{Shlyakhtenko1998amalgamation,Shlyakhtenko1998amalgamation}.  Hence, it is natural to expect that, whereas the large-$n$ behavior of the Gaussian unitary ensemble (GUE) is described by a free semicircular family, the large-$n$ behavior of general Gaussian random matrices would be described by operator-valued semicircular families, under some mild hypotheses.
	
	An operator-valued semicircular family $(X_i)_{i \in I}$ (see \cite{Speicher1998,shlyakhtenko1999valued} and \S \ref{subsec: operator valued prelim}) is defined in terms of a specified tracial von Neumann algebra $(B,\tau)$ and a \emph{$B$-valued covariance matrix}, which is a family $\eta=(\eta_{i,j})_{i,j\in I}$ of normal linear maps on $B$ satisfying complete positivity and $\tau$-symmetry conditions (see Definition~\ref{def: operator valued covariance}). Consequently, our goal demands that we study matrix approximations for the operators  $(X_i)_{i\in I}$, the maps $\eta_{i,j}$, and the elements of $B$ jointly.  Although many works have considered the joint distribution of random matrices together with some deterministic matrices, joint finitary approximations for operators and completely positive maps have not been studied systematically.  Thus, our first task is to formulate what operator-valued weak and strong convergence should formally be.
	
	Toward this end, we introduce the notion of \emph{covariance polynomials} in self-adjoint operators $b = (b_\omega)_{\omega \in \Omega}$ and linear maps $(\eta_{i,j})_{i,j \in I}$. These are analogs of non-commutative polynomials which account for  iterated application of the linear maps $\eta_{i,j}$ on elements of the underlying algebra $B$; for example,
	\[
	f(\eta,b) = b_{\omega_1} \eta_{i_1,i_2}(b_{\omega_1} b_{\omega_2}) b_{\omega_3} b_{\omega_1} \eta_{i_2,i_2}(b_{\omega_1} \eta_{i_1,i_3}(b_{\omega_2})) b_{\omega_4}.
	\]
	For the general definition, see \S \ref{subsec:covariance_polynomials}.  These covariance polynomials are a generalization of trace polynomials in \cite{Cebron2013,DHK2013} and polynomials involving a conditional expectation in \cite[\S 2.2]{DGS2021}. They come up naturally in the evaluation of $B$-valued moments of a $B$-valued semicircular family (see Lemma~\ref{lem: semicircular partition formula}), and similarly from general Gaussian random matrices, as we will see below.  We formulate a corresponding notion of \emph{covariance laws} and a version of the GNS construction in \S \ref{subsec:covariance_laws}.  This results in the notions of \emph{weak and strong convergence for covariance laws}, where the ordinary polynomials are replaced by covariance polynomials.
	
	Our main results give sufficient conditions for families of deterministic matrices and Gaussian random matrices to converge weakly or strongly in covariance law to a $B$-valued semicircular family.  As explained in \S \ref{subsec: eta Gaussian ensemble}, for each $\tr_n$-symmetric $\mathbb{M}_n$-valued covariance matrix $\eta=(\eta_{i,j})_{i,j\in I}$, there is a family of mean-zero Gaussian random matrices $(X_i)_{i\in I}$ (unique in distribution) such that
	\[
	\mathbb{E}[X_i a X_j] = \eta_{i,j}(a), \text{ for } a \in \mathbb{M}_{n} \text{ and } i,j \in I,
	\]
	which we call the \emph{$\eta$-Gaussian family}.  
	It is perhaps more standard in the random matrix literature to specify a Gaussian ensemble in terms of the $n^2 |I| \times n^2 |I|$ matrix giving the covariances of the entries $X_i$.  This entrywise covariance has a natural operator-algebraic interpretation in terms of the \emph{Choi matrix} of the completely positive map $\eta$ (see \cite{Choi1975}) given by
	\[
	T = \sum_{i,j \in I} \sum_{s,t = 1}^{n} E_{i,j} \otimes \eta_{i,j}(E_{s,t}) \otimes E_{s,t},
	\]
	which is related to the covariances of the entries by
	\[
	T = \sum_{i,j \in I} \sum_{s,s',t,t' = 1}^{n} \mathbb{E}[(X_i)_{s,s'} \overline{(X_j)_{t,t'}}] \left( E_{i,j} \otimes E_{s,t} \otimes E_{s',t'} \right);
	\]
	see also Remark \ref{rem: entrywise expression}. 
	Bandeira, Boedihardjo, and van Handel \cite{BBvH2023} observed that the operator norm $\norm{T}$ is a useful parameter to ``measure non-commutativity'' of a general Gaussian random matrix, so that we should expect free probabilistic limiting behavior for sequences of $\eta^{(k)}$-Gaussian families $(X_i^{(k)})_{i\in I}$ when the corresponding Choi matrices satisfy $\norm{T^{(k)}} \to 0$ (or when $(\log n(k))^3 \norm{T^{(k)}} \to 0$ if we want strong convergence).  While the applications on weak and strong convergence to free semicirculars in \cite[\S 7]{BBvH2023} focused on the case where $\eta_{i,j}^{(k)}(1) = \delta_{i,j} 1$, they were based on much more general, sharp, non-asymptotic estimates that compare a Gaussian random matrix with a corresponding matrix of semicircular operators. Thus, we are able to use their results and techniques for general covariances $\eta$ while incorporating deterministic matrices.
	
	Our first main result asserts weak convergence of appropriate $\eta$-Gaussian matrices to operator-valued semicirculars. In the following, fix a tracial von Neumann algebra $(M,\tau)$ generated by a subalgebra $B\subset M$ and a $(B,\eta)$-semicircular family $X=(X_i)_{i\in I}$, where $\eta=(\eta_{i,j})_{i,j\in I}$ is a $\tau$-symmetric $B$-valued covariance matrix, and let $b=(b_\omega)_{\omega\in \Omega}$ be a generating tuple for $B$. For a sequence of integers $(n(k))_{k\in \N}\subset \N$, let $\eta^{(k)}=(\eta_{i,j}^{(k)})_{i,j\in I}$ be a $\tr_{n(k)}$-symmetric $\mb{M}_{n(k)}$-valued covariance matrix and for each finite $F\subseteq I$ let $T_F^{(k)}$ be the Choi matrix associated to $\eta^{(k)}_F=(\eta^{(k)}_{i,j})_{i,j\in F}$, let $X^{(k)}=(X_i^{(k)})_{i\in I}$ be an $\eta^{(k)}$-Gaussian family, and let $b^{(k)}=(b^{(k)}_\omega)_{\omega\in \Omega}\subset \mb{M}_{n(k)}$ satisfying $\sup_k \|\eta_{i,j}^{(k)}\| <\infty$ for each $i,j\in I$ and $\sup_k \| b_\omega^{(k)}\| <\infty$ for each $\omega\in \Omega$.

	\begin{thmalpha}[Operator-valued weak convergence; see Theorem~\ref{thm: weak convergence with proof}] \label{introthm: weak convergence}
		With the setup above, assume that $(\eta^{(k)},b^{(k)})$ converges weakly in covariance law to $(\eta,b)$ and assume that $\lim_{k \to \infty} \norm{T_F^{(k)}} = 0$ for all finite $F\subseteq I$.  Then $(\eta^{(k)},b^{(k)},X^{(k)})$ converges weakly in covariance law to $(\eta,b,X)$, both almost surely and in expectation.
	\end{thmalpha}

	Although Theorem \ref{introthm: weak convergence} takes for granted that $B$ and $\eta$ are already in hand, the existence of such a $B$ and $\eta$ is automatic provided that $\lim_{k \to \infty} \tr_{n(k)}(f(\eta^{(k)},b^{(k)}))$ exists for each covariance polynomial $f$ (see Proposition~\ref{prop: base algebra}), and this convergence can always be arranged by passing to a subsequence.  The proof of Theorem \ref{introthm: weak convergence} is a generalization of \cite[\S 7.1]{BBvH2023}, which uses $\norm{T_F^{(k)}}$ to estimate crossing terms in a Wick expansion (see \S \ref{subsec: Wick expansion} - \ref{subsec: weak convergence in expectation}).  We also work carefully with non-commutative $L^p$ norms to obtain almost sure convergence of the traces without knowing almost sure boundedness of the operator norms as $k \to \infty$ (see \S \ref{subsec: weak almost sure convergence}).
	
	Our second main theorem gives sufficient conditions for strong convergence of covariance laws.
	
	\begin{thmalpha}[Operator-valued strong convergence; see Theorem~\ref{thm: strong convergence general}] \label{introthm: strong convergence}
		With setup above, assume that $\lim_{k \to \infty} (\log n(k))^3 \norm{T_F^{(k)}} = 0$ for all finite $F\subseteq I$.  Assume that $(\eta^{(k)},b^{(k)})$ converge strongly in covariance law to $(\eta,b)$.  Then $(\eta^{(k)},b^{(k)},X^{(k)})$ converges strongly in covariance law to $(\eta,b,X)$, both almost surely and in expectation.
	\end{thmalpha}

	To prove Theorem \ref{introthm: strong convergence}, we first explain how the strong convergence of plain polynomials in $(b^{(k)},X^{(k)})$ follows from combining the results of Bandeira, Boedihardjo, and van Handel \cite{BBvH2023} with recent work of Gao and Kunnawalkam Elayavalli \cite{GKE2026Toeplitz}.  Indeed, because $(b^{(k)},\eta^{(k)})$ converges strongly to $(b,\eta)$, \cite[Corollary 1.4]{GKE2026Toeplitz} shows that $(b^{(k)}, X_{\free}^{(k)})$ converges strongly to $(b,X)$.  Then \cite{BBvH2023} allows to compare the the norm of a polynomial in $(b^{(k)},X^{(k)})$ with the norm of the same polynomial in $(b^{(k)},X_{\free}^{(k)})$, and hence deduce convergence of the norms of ordinary polynomials in $(b^{(k)},X^{(k)})$ (see \S \ref{subsec: norms of polynomials}).
	To finish the proof, we need to upgrade from plain non-commutative polynomials to covariance polynomials, which we accomplish by approximating a covariance polynomial by an ordinary non-commutative polynomial in free copies by an averaging trick (see \S \ref{subsec: averaging argument}).  
	This approximation also behaves well with respect to the random matrix models, in the sense the expected value of the polynomial over the extra independent copies yields the original covariance polynomial.

	We remark that in the first draft of this paper, we simply assumed strong convergence of $(b^{(k)}, S^{(k)})$ to $(b,S)$ as part of the hypothesis, but thanks to \cite{GKE2026Toeplitz} we can now state Theorem \ref{introthm: strong convergence} under the natural and more general hypotheses.  In fact, the results of \cite{GKE2026Toeplitz} also show convergence of $(\eta^{(k)}, b^{(k)},X_{\free}^{(k)})$ to $(\eta, b,X)$ immediately, and hence the main difficulty in upgrading to covariance polynomials comes from the random matrix side, i.e., in comparing the covariance laws of $(\eta^{(k)}, b^{(k)},X^{(k)})$ and $(\eta^{(k)}, b^{(k)},X_{\free}^{(k)})$. 
	It is not clear to us at present whether the Toeplitz algebra and gauge-invariant uniqueness mechanism used in \cite{GKE2026Toeplitz} can be used directly on the random matrix tuples, and it is also an interesting open question how to make the results of \cite{GKE2026Toeplitz} quantitative.

	In \S\ref{sec: weighted GUE}, we apply Theorems~\ref{introthm: weak convergence} and \ref{introthm: strong convergence} to the setting of continuously weighted Gaussian Wigner matrices.  This includes continuous versions of Gaussian band matrices which have been studied in many previous works \cite{Shl96,Shlyakhtenko1998band,BEYY2017band,Au2018band,Au2021band}, as well as Gaussian block matrices as in \cite{FOBS2008mimo} and \cite[\S 9]{MingoSpeicher}.  Fix a family of continuous functions $(h_{i,j})_{i,j\in I}$ on $[0,1]^2$ such that $h_{i,j}(s,t) = h_{j,i}(t,s)$ and $(h_{i,j}(s,t))_{i,j\in F} \geq0$ for all finite $F \subseteq I$.  Equip $L^\infty[0,1]$ with the trace $\tau = \int_{[0,1]}$ and let $\eta_{i,j}: L^\infty[0,1]\to L^\infty[0,1]$ be given by
	\[
	\eta_{i,j}(f)(s) = \int_{[0,1]} h_{i,j}(s,t) f(t)\,dt.
	\]
	Then $\eta=(\eta_{i,j})_{i,j\in I}$ is a $L^\infty[0,1]$-valued covariance matrix. Let $X$ be a $(L^\infty[0,1],\eta)$-semicircular family and let $b = (b_\omega)_{\omega \in \Omega}$ be generators of $L^\infty[0,1]$ taken from $C[0,1]$.   The random matrix approximations are given by a discretization of $h_{i,j}$ as follows. Let $E_n\colon \mb{M}_n\to \mb{D}_n$ be the $\tr_n$-preserving conditional expectation onto the subalgebra of diagonal matrices $\mb{D}_n$, and let us also identify $\mb{D}_n$ as subalgebra of $L^\infty[0,1]$ so that
	\[
	E^{(n)}(f):= \diag\left(n \int_0^{1/n} f,\cdots , n\int_{(n-1)/n}^1 f\right)
	\]
	gives the unique $\tau$-preserving conditional expectation onto $\mb{D}_n$. Then $\eta^{(n)}:=( E^{(n)}\circ \eta_{i,j}\circ E_n)_{i,j\in I}$ defines an $\mb{D}_n$-valued covariance matrix over $\mb{M}_n$. Finally, let $X^{(n)}$ be an $\eta^{(n)}$-Gaussian family and let $b^{(n)}= ( E^{(n)}(b_\omega) )_{\omega\in \Omega}$.
	
	\begin{thmalpha}[{see Theorem~\ref{thm: strong convergence weighted proof}}]\label{introthm: strong convergence weighted}
		For the continuously weighted Gaussian Wigner model described above, $(\eta^{(n)},b^{(n)},X^{(n)})$ converges weakly and strongly in covariance law to $(\eta,b,X)$, both almost surely and in expectation.
	\end{thmalpha}
	
	Note that the case of a Gaussian band matrix is obtained in Theorem \ref{introthm: strong convergence weighted} when $|I| = 1$ and $h(s,t) = f(s-t)$ for some nonnegative $f \in C[0,1]$ (see Example~\ref{ex: Guassian_band_matrices}); the weak convergence for such Gaussian band matrices was proved by Shlyakhtenko \cite[Theorem 4.1]{Shlyakhtenko1998band}.  Moreover, if $h_{i,j}(s,t) = \delta_{i=j}$, then we recover the case of independent GUE matrices indexed by $I$.  By choosing the weight appropriately, this construction also provides random matrix models for interpolated free group von Neumann algebras from \cite{Dykema1994,Rad94} (see Example~\ref{ex: interpolated_fgf}), and thus we have the following corollary.
	
	\begin{coralpha}[{See Example \ref{ex: interpolated_fgf}}]
		For each $t \in (1,\infty]$, the interpolated free group factor $L(\mathbb{F}_t)$ has generators $(b,X)$ which admit strongly convergent random matrix models.  In particular, $L(\mathbb{F}_t)$ has a weak-$*$ dense $\mathrm{C}^*$-subalgebra $\cA \subseteq L(\mathbb{F}_t)$ which is matricial field (MF), in the sense of Blackadar-Kirchberg \cite[Def 3.2.1]{BlackadarKirchberg1997}. 
	\end{coralpha}
	
	In \S \ref{subsec: shrinking band matrix}, we consider the example of periodic-band matrices with band width $n \varepsilon_n$ where $\varepsilon_n \to 0$ and $(\log n)^3 / n \varepsilon_n \to 0$.  For independent band matrices, we show operator-valued strong convergence to a $C(\R/\Z)$-valued semicircular family with $\eta_{i,j} = \delta_{i=j} \id$, which are a standard scalar-valued semicircular family in tensor position with $C(\R/\Z)$.  To check the strong convergence hypothesis for Theorem \ref{introthm: strong convergence}, we cannot use the same coupling argument as for Theorem \ref{introthm: strong convergence weighted}, but instead we take advantage of the fact that $C(\R/\Z)$ commutes with $\mathrm{C}^*(X)$ in the limit.
	
	Our work brings up several natural open problems for future research.  First, because the estimates of \cite{BBvH2023} have been extended from the Gaussian case to sums of independent matrices in \cite{BvH2024universality}, we expect that there a corresponding generalization of strong convergence to operator-valued semicirculars.  And second, give quantitative estimates for the convergence of norms and traces of covariance polynomials in Theorems \ref{introthm: weak convergence} and \ref{introthm: strong convergence}, after assuming appropriate quantitative versions of the hypotheses.
	
	\section*{Acknowledgments}
	
	We thank the Fields Institute for their hospitality during the Fall 2023 Thematic Program on Operator Algebras and Applications where this project was initiated.  We thank Charles Bordenave for suggesting to investigate band matrices with a shrinking width (\S \ref{subsec: shrinking band matrix}).  We also thank Ramon van Handel for comments on our draft.  We thank David Gao and Srivatsav Kunnawalkam Elayavalli for allowing us to include updated versions of our results using their work.  We thank the referee for careful reading and suggestions to improve the manuscript.
	
	D.~Jekel was partially supported by the Discovery grant ``Logic and $\mathrm{C}^*$-algebras'' from the Natural Sciences and Engineering Research Council (Canada); the Danish Independent Research Fund, grant 1026-00371B; and a Horizon Europe Marie Sk{\l}odowska Curie Action, FREEINFOGEOM, grant 101209517.  Y.~Lee and B.~Nelson were supported by National Science Foundation (US) grant DMS-2247047. J.~Pi was supported by the Engineering and Physical Sciences Research Council (UK), grant EP/X026647/1.
	
	Funded by the European Union.  Views and opinions expressed are those of the author(s) only and do not necessarily reflect 
	those of the European Union or the Research Executive Agency. Neither the European Union nor the granting authority can be held responsible for them.

	\tableofcontents

	\section{Preliminaries} \label{sec: preliminaries}
	
	\subsection{Operator-valued free probability} \label{subsec: operator valued prelim}
	
	We adopt the terminology \emph{tracial von Neumann algebra} to mean a pair $(M,\tau)$ consisting of a von Neumann algebra $M$ and a faithful normal tracial state $\tau\colon M\to \C$. Recall that if $B$ is a von Neumann subalgebra of $M$ then there is a unique faithful normal $\tau$-preserving conditional expectation $E_{B}: M \to B$. In this case, von Neumann subalgebras $B\subset M_i\subset M$, $i\in I$, are \emph{freely independent over $B$} with respect to $E_B$ if whenever $k \in \N$ and $i_1 \neq i_2 \neq \cdots \neq i_k$ and $x_j \in M_{i_j}$ with $E_{B}[x_j] = 0$, then $E_{B}[x_1 \cdots x_k] = 0$. By precomposing with $E_B$, we may extend any (faithful) normal linear map on $B$ to a (faithful) normal linear map on $M$. We will do this implicitly with any trace on $B$ and operator-valued covariance matrix on $B$ (see Definition~\ref{def: operator valued covariance} below).
	
	We next recall non-crossing partitions and $B$-valued cumulants, combinatorial tools that aid moment computations in $B$-valued free probability. For each $k\in \N$ we denote $[k]:=\{1,\ldots, k\}$, and we let $\mathcal{P}(k)$ and  $\mathcal{NC}(k)$  (resp. $\mathcal{P}_2(k)$ and $\mathcal{NC}_2(k)$) denote the set of partitions and non-crossing partitions (resp. pair partitions and non-crossing pair partitions) on $[k]$; for background, see \cite{Speicher1994}.  
	A non-crossing partition can be viewed as a diagram for the composition of some multilinear maps as follows \cite{Speicher1998}.  For $k \in \N$, let $\Lambda_k: M^k \to B$ be a $k$-linear map.  Then for each non-crossing partition $\pi \in \mathcal{NC}(k)$, there is a corresponding $k$-linear map $\Lambda_\pi$ defined recursively as follows:
	\begin{itemize}
		\item If $\pi$ has a single block $[k]$, then $\Lambda_\pi = \Lambda_k$.
		\item Suppose that $\pi \in \mathcal{NC}(k)$ and $\{s+1,\dots,s+\ell\}$ is a block of $\pi$, and let $\pi'$ be the partition of $[k - \ell]$ obtained by removing this block and relabeling the remaining indices in increasing order.  Then
		\[
		\Lambda_\pi[x_1,\dots,x_k] = \begin{cases} \Lambda_{\pi'}[x_1,\dots,x_s, \Lambda_\ell[x_{s+1},\dots,x_\ell] x_{\ell+1}, x_{\ell+2},\dots,x_k], & s + \ell < k, \\
			\Lambda_{\pi'}[x_1,\dots,x_s] \Lambda_\ell[x_{s+1},\dots,x_k], & s + \ell = k.
		\end{cases}
		\]
	\end{itemize}
	Given $B \subseteq M$ tracial von Neumann algebras, the \emph{$B$-valued cumulants} are the unique multilinear maps $K_k: M^k \to B$ satisfying
	\begin{equation} \label{eq: moment-cumulant formula}
		E_{B}[x_1 \cdots x_k] = \sum_{\pi \in \mathcal{NC}(k)} K_\pi[x_1,\dots,x_k] \text{ for } x_1, \dots, x_k \in M.
	\end{equation}
	These also can be defined by M{\"o}bius inversion (see \cite{Speicher1998}).  Moreover, the cumulants satisfy
	\[
	b_0 K_k[x_1 b_1, \dots, x_k b_k] = K_k[b_0 x_1, \dots, b_{k-1} x_k] b_k
	\]
	for $x_1$, \dots, $x_k \in M$ and $b_0$, \dots, $b_{k+1} \in B$.
	
	Operator-valued semicircular families (see \cite[\S 4.2-4.3]{Speicher1998}, \cite{shlyakhtenko1999valued}) are an analog of multivariate Gaussians for the setting of operator-valued free probability, which are specified in terms of operator-valued covariance matrices. The following definition should be compared with {\cite[Definition 2.1]{shlyakhtenko1999valued} and \cite[Theorem 4.3.1]{Speicher1998}.
		
		\begin{defi} \label{def: operator valued covariance}
			Let $(B,\tau)$ be a tracial von Neumann algebra and let $\eta_{i,j}:B\to B$ be normal linear maps for $i,j\in I$. We say that $\eta:=(\eta_{i,j})_{i,j\in I}$ is an \textbf{operator-valued covariance matrix} on $B$ (or a \textbf{$B$-valued covariance matrix}) if for every finite $F \subseteq I$, the map
			\begin{align*}
				\eta_F: B &\to B(\ell^2 I) \otimes B\\
				x &\mapsto \sum_{i,j \in F} e_{i,j} \otimes \eta_{i,j}(x)
			\end{align*}
			is completely positive, where $\{e_{i,j}\in B(\ell^2 I)\colon i,j\in I\}$ are the standard matrix units.  We say that $\eta$ is \textbf{$\tau$-symmetric} if
			\[
			\tau(\eta_{i,j}(x) y) = \tau(x \eta_{j,i}(y)) \qquad x,y \in B.
			\]
		\end{defi}
		
		\begin{rem}
			Note that since $\eta_F$ is completely positive for any finite $F \subseteq I$, in particular $\eta_{i, j}(x)$ is self-adjoint whenever $x$ is self-adjoint, hence
			\[
			\eta_{i,j}(x)^* = \eta_{j,i}(x^*).
			\]
		\end{rem}

		Note that in the case of an infinite index set $I$, our definition is more general than that of \cite[Definition 2.1]{shlyakhtenko1999valued}, which essentially demands one also take $F=I$ in the above definition. This more general definition is better suited to our main theorems, since they concern limits that only involve finitely many indices from $I$ at a time.
		
		\begin{defi} \label{def: operator-valued semicirculars}
			Let $(M,\tau)$ be a tracial von Neumann algebra with von Neumann subalgebra $B$.
			Given a $\tau$-symmetric operator-valued covariance matrix $\eta = (\eta_{i,j})_{i,j \in I}$ on $B$, a tuple $(X_i)_{i\in I} \subset M$ is called a \textbf{$(B,\eta)$-semicircular family}  if
			\begin{align*}
				K_2[X_i, bX_j] &= \eta_{i,j}(b) & \text{for } i,j \in I, b \in B, & \\
				K_k[X_{i_1} b_1,\dots,X_{i_k} b_k] &= 0, & \text{for } i_1, \dots, i_k \in I, b_1, \dots, b_k \in B, & \text{ for } k \neq 2,
			\end{align*}
			where $K_k$ are the cumulants associated to $E_B\colon M\to B$.
		\end{defi}
		
		Note that by \eqref{eq: moment-cumulant formula}, the joint $B$-valued moments of a $(B,\eta)$-semicircular family are uniquely determined by the cumulant conditions in Definition~\ref{def: operator-valued semicirculars}. A moment characterization of $(B,\eta)$-semicircular family can be found in Lemma~\ref{lem: semicircular partition formula} below. 
		
		The existence of operator-valued semicircular families can be shown as follows.  The case where the index set $I$ is finite fits into Shlyakhtenko's framework.  The $\mathrm{C}^*$-correspondence construction in \cite[Lemma 2.2]{shlyakhtenko1999valued} yields a $(B,\eta)$-semicircular family in a von Neumann algebra $\Phi(B,\eta)$.  In particular, \cite[Lemma 2.10]{shlyakhtenko1999valued} shows that there exists a normal conditional expectation $E_B:\Phi(B,\eta)\to B$, and \cite[Proposition 2.20]{shlyakhtenko1999valued} shows that $\tau \circ E_{B}$ is tracial when $\eta$ is $\tau$-symmetric, and in this case $E_{B}$ is faithful by \cite[Proposition 5.2]{shlyakhtenko1999valued}. Moreover, \cite[Definition 4.2.3]{Speicher1998} implies this family satisfies the cumulant conditions in Definition~\ref{def: operator-valued semicirculars} (see also \cite[Example 2.5]{Lee24}).  Now suppose that $I$ is infinite, and for $F \subseteq I$ finite, let $M_F$ be a tracial von Neumann algebra generated by a $(B,\eta|_F)$-semicircular family $(X_{F,i})_{i \in F}$.  For $F \subseteq F'$, the joint distribution of $(B,(X_{F',i})_{i \in F})$ is the same as that of $(B,(X_{F,i})_{i \in F})$, so we obtain a trace-preserving inclusion $M_F \to M_{F'}$.  Let $M$ be the inductive limit of the $M_F$'s over finite $F \subseteq I$, and let $X_i$ be the common image of $X_{F,i}$ in $M$.  Then $(X_i)_{i \in I}$ is a $(B,\eta)$-semicircular family.
		
		\subsection{{$\eta$}-Gaussian random matrix ensembles} \label{subsec: eta Gaussian ensemble}
		
		In order to give a matrix model for operator-valued semicirculars, it is natural to define Gaussian random matrix models associated to operator-valued covariance matrices where the base algebra $B$ is the $n \times n$ matrix algebra $\mathbb{M}_n$.
		
		\begin{prop} \label{eta-GUE exists}
			Let $\eta = (\eta_{i,j})_{i,j \in I}$ be a $\tr_n$-symmetric operator-valued covariance matrix on $\mb{M}_n$.  Then there is a jointly Gaussian family $\{X_i\colon i\in I\}$ of random $n \times n$ self-adjoint matrices with mean zero that satisfy
			\begin{equation} \label{eq: matrix covariance}
				\E[X_i a X_j] = \eta_{i,j}(a) \qquad a\in \mb{M}_n.
			\end{equation}
			The distribution of the family $\{X_i\colon i\in I\}$ is uniquely determined by $\eta$.
		\end{prop}
		
		\begin{proof}
			By the Kolmogorov extension theorem, it suffices to prove the claim when the index set $I$ is finite.  In general, to obtain existence of a jointly Gaussian family of real random variables, one only has to specify a positive semi-definite covariance matrix (here ``covariance matrix'' is meant in the usual sense in probability theory as opposed to ``operator-valued covariance matrix'').  Since $X:=(X_i)_{i\in I}$ should be a random variable in the real vector space $(\mathbb{M}_n)_{\sa}^{\oplus I}$, using the Hilbert--Schmidt inner product with respect to $\Tr_n$, the covariance matrix should be a positive semi-definite operator $T$ on  $(\mathbb{M}_n)_{\sa}^{\oplus I}$.  To translate \eqref{eq: matrix covariance} into a covariance matrix, recall that for a completely positive map $\eta: \mathbb{M}_n \to B(\ell^2 I) \otimes \mathbb{M}_n$, the \emph{Choi matrix} is the element
			\[
			T = \sum_{i,j \in I} \sum_{k,\ell=1}^n E_{i,j} \otimes \eta_{i,j}(E_{k,\ell}) \otimes E_{k,\ell} \in B(\ell^2 I) \otimes \mathbb{M}_n \otimes \mathbb{M}_n.
			\]
			Choi's theorem \cite[Theorem 2]{Choi1975} shows that $\eta$ is completely positive if and only if $T \geq 0$.  To view $T$ as a linear transformation of $\mathbb{M}_n^{\oplus I}$, recall that $B(\C^I \otimes \mathbb{M}_n)$ is naturally identified with $B(\ell^2 I) \otimes \mathbb{M}_n \otimes \mathbb{M}_n^{\op}$, where the $B(\ell^2 I)$ acts on $\C^I$, the first $\mathbb{M}_n$ acts on $\mathbb{M}_n$ by left multiplication, and the opposite algebra $\mathbb{M}_n^{\op}$ acts on $\mathbb{M}_n$ by right multiplication.  Moreover, $\mathbb{M}_n$ is isomorphic to $\mathbb{M}_n^{\op}$ by sending $E_{k,\ell}$ to $E_{\ell,k}^{\op}$, so we can write $T$ as
			\begin{equation} \label{eq: T definition}
				T = \sum_{i,j \in I} \sum_{k,\ell=1}^n E_{i,j} \otimes \eta_{i,j}(E_{k,\ell}) \otimes E_{\ell,k}^{\op} \in B(\ell^2 I) \otimes \mathbb{M}_n \otimes \mathbb{M}_n^{\op}.
			\end{equation}
			Thus, $T$ is now a nonnegative $\mathbb{C}$-linear transformation of $\C^I \otimes_{\C} \mathbb{M}_n$.  Since we actually want random variables in the real subspace $\R^I \otimes_{\R} (\mathbb{M}_n)_{\sa}$, it remains to show that $T$ leaves this real subspace invariant, or equivalently that $T[A^*] = (T[A])^*$ for $A \in \C^I \otimes_{\C} \mathbb{M}_n$.  Letting $T_{i,j} = \sum_{k,\ell} \eta_{i,j}(E_{k,\ell}) \otimes E_{\ell,k}^{\op}$, it suffices to check that $\Tr_n(T_{i,j}[A^*]B) = \Tr_n(A^* T_{j,i}[B])$, and by linearity it suffices to consider basis elements $E_{s,t}$, $E_{s',t'}$ for $s, t, s', t' \in [n]$.  Using the $\tr_n$-symmetry condition for $\eta$,
			\begin{align*}
				\Tr_n(T_{i,j}[E_{s,t}^*] E_{s',t'}) &= \sum_{k,\ell \in [n]} \Tr_n(\eta_{i,j}(E_{k,\ell}) E_{t,s} E_{\ell,k} E_{s',t'})\\
				&= \Tr_n(\eta_{i,j}(E_{s',s}) E_{t,t'}) \\
				&= \Tr_n(E_{s',s} \eta_{j,i}(E_{t,t'})) = \Tr_n(E_{s,s'}^* \eta_{j,i}(E_{t,t'})),
			\end{align*}
			which evaluates to $\Tr_n(E_{s,t}^* T_{j,i}[E_{s',t'}])$ by a symmetric sequence of manipulations.
			
			Finally, uniqueness holds because the joint distribution of a Gaussian family is determined by its covariance.
		\end{proof}
		
		\begin{defi} \label{def: eta-Gaussian}
			Let $\eta = (\eta_{i,j})_{i,j \in I}$ be a $\tr_n$-symmetric operator-valued covariance matrix.
			An \textbf{$\eta$-Gaussian family} is a tuple $(X_i)_{i\in I}$ of random $n \times n$ self-adjoint matrices whose entries are mean-zero Gaussian random variables satisfying \eqref{eq: matrix covariance}.
			We define the \textbf{associated Choi matrix of $\eta$} to be the matrix $T$ as defined in \eqref{eq: T definition}.
		\end{defi}
		
		\begin{rem} \label{rem: entrywise expression}
			Note that the Choi matrix is the same as the \emph{covariance matrix} of the entries in the usual sense, which is denoted $\Cov(X)$ in \cite{BBvH2023}.  Indeed, considering the case where $|I| = 1$ for simplicity,
			\[
			\Cov(X)_{s,t;s',t'} = \mathbb{E}[X_{s,t} \overline{X}_{s',t'}] = \mathbb{E}[\Tr_n(E_{s,t}^*X) \Tr_n(X E_{s',t'})] = \Tr_n[E_{s,t}^* T[E_{s',t'}]].
			\]
		\end{rem}

		\begin{rem}
			The $\eta$-Gaussian family also has the following invariance property:  Suppose that $B \subseteq \mb{M}_n$ is a unital $*$-subalgebra such that $\eta = \eta \circ E_{B} = E_{B} \circ \eta$.  Given a unitary matrix $u\in B' \cap \mb{M}_n$, for all $a\in \mb{M}_n$ one has
			\[
			\E[ (uX_i u^*)a(u X_j u^*)] = u \eta_{ij}(E_{B}[u^*a u]) u^*= \eta_{ij}(E_B[a]) = \eta_{i,j}(a),
			\]
			where we have used $E_B[u^* a u]= E_B[a]$. Hence, the $\eta$-Gaussian distribution is invariant under conjugation by unitaries which commute with $B$. 
		\end{rem}
		
		Besides the description of Gaussian matrix tuples in terms of the operator-valued covariance $\eta$ and its Choi matrix $T$, Gaussian matrices are often described as a matrix linear combination of classical Gaussians:
		\begin{equation} \label{eqn: Gaussian matrix setup BBvH}
			X_i = \sum_{k=1}^m g_k a_{k,i},
		\end{equation}
		where the $g_k$'s are independent standard normals (see \cite{BBvH2023}).  Taking $|I| = 1$ for simplicity, so $X= \sum_k g_k a_k$, then the operator-valued covariance can be expressed as
		\[
		\eta(b) = \mathbb{E}[XbX] = \sum_{k=1}^m a_k b a_k,
		\]
		which means precisely that the $a_{k,i}$'s are a family of Kraus operators for the completely positive map $\eta$.  Choi \cite[Theorem 1]{Choi1975} showed that every completely positive map between matrix algebras admits Kraus operators (which are not unique in general).  The ability to choose self-adjoint Kraus operators corresponds to the property that $\eta$ is $\tr_n$-symmetric (in fact, it is easy to verify this based on the fact that every jointly Gaussian variable in $(\mathbb{M}_n)_{\sa}$ can be expressed in the form \eqref{eqn: Gaussian matrix setup BBvH}).
		
		Hence, overall, we have three equivalent ways to specify a Gaussian random matrix model:
		\begin{itemize}
			\item By giving an operator-valued covariance matrix $\eta=(\eta_{i,j})_{i,j\in I}$ that should satisfy $\eta_{i,j}[a] = \mathbb{E}[X_iaX_j]$. 
			\item By giving an operator $T$ on $(\mathbb{M}_n)_{\sa}^{\oplus I}$ specifying covariances of the entries $\Cov((X_i)_{r,s}, (X_j)_{t,u})$ (which is the Choi matrix of $\eta$).
			\item By expressing the matrix model as $X_i = \sum_{k=1}^m g_k a_{k,i}$ where the $g_k$'s are independent normal random variables and the $a_{k,i}$'s are deterministic matrices (which are Kraus operators for $\eta$).
		\end{itemize}
		The translation between these different viewpoints on a Gaussian matrix also leads to a natural interpretation of random matrix results such as those of \cite{BBvH2023} in terms of operator-valued semicirculars.  Indeed, it is a common technique to compare a Gaussian random matrix $X = \sum_{t=1}^m g_t a_t$ as above with its free counterpart $X_{\operatorname{free}} = \sum_{t=1}^m a_t \otimes s_t \in \mathbb{M}_n \otimes M$ where the $s_t$'s are freely independent semicirculars in a tracial von Neumann algebra $M$.  The matrix $X_{\operatorname{free}}$ can in fact be reformulated as an $\mathbb{M}_n$-valued semicircular operator, and the same holds for several matrix models.
		
		\begin{prop} \label{prop: identifying X free}
			Let $X = (X_i)_{i \in I}$ be Gaussian random matrices expressed as
			\[
			X_i = \sum_{k=1}^m g_k a_{k,i}
			\]
			for self-adjoint matrices $a_{k,i}$ and independent standard normal random variables $g_k$, and let $\eta=(\eta_{i,j})_{i,j\in I}$ be the corresponding operator-valued covariance matrix $\eta_{i,j}[a] = \mathbb{E}[X_i a X_j]$.  Let $(S_k)_{k=1}^m$ be a free semicircular family in a tracial von Neumann algebra $(M,\tau)$, and let
			\[
			X_{\operatorname{free},i} = \sum_{k=1}^m a_{k,i} \otimes S_k \in \mathbb{M}_n \otimes M.
			\]
			Then $X_{\operatorname{free}}=\{X_{\operatorname{free},i}\colon i\in I\}$ is an $(\mathbb{M}_n,\eta)$-semicircular family.
		\end{prop}
		
		\begin{proof}
			Noting that the expectation from $\mb{M}_n\otimes M$ onto $\mathbb{M}_n$ is given by $\id \otimes \tau$, the free independence and vanishing first moments of $(S_k)_{k=1}^m$ imply all $\mathbb{M}_n$-valued cumulants vanish except those of order two (see \cite[Proposition 9.3.13]{MingoSpeicher} and \cite[Theorem 6.2]{NSS02}). For the order two cumulants we have:
			\[
			K_2[X_{\operatorname{free},i} b, X_{\operatorname{free},j}] = E_{\mathbb{M}_n}[X_{\operatorname{free},i} b X_{\operatorname{free},j}]= \sum_{k=1}^m a_{k,i} b a_{k,j} = \mathbb{E}[X_i b X_j].
			\]
			Thus $X_{\operatorname{free}}$ is an $(\mathbb{M}_n,\eta)$-semicircular family.
		\end{proof}
		
		\section{Covariance polynomials and laws} \label{sec: covariance laws}
		
		Suppose $(B,\tau)$ is a tracial von Neumann algebra  generated by a tuple ${b}=(b_\omega)_{\omega\in \Omega}$ and equipped with an operator-valued covariance matrix $\eta=(\eta_{i,j})_{i,j\in I}$. In order to conveniently encode $\eta$, we will consider the smallest $*$-subalgebra of $B$ that is invariant under the $\eta_{i,j}$'s and contains each $b_\omega$. Formally, this $*$-subalgebra is spanned by monomials in ${b}$ that are interspersed with (potentially iterated) applications of the $\eta_{i,j}$ (e.g. $b_{\omega_1}\eta_{i_1,j_1}(b_{\omega_2} \eta_{i_2,j_2}(b_{\omega_3}) b_{\omega_4})  b_{\omega_5}$). We call the elements of this $*$-subalgebra \emph{covariance polynomials} due to the relation $E_B(X_i b X_j) = \eta_{i,j}(b)$ when $\{X_i : i \in I \}$ is an $(B,\eta)$-semicircular family. We formalize this notion in Section~\ref{subsec:covariance_polynomials} below.  Covariance polynomials were previously used in \cite[\S 2.2]{DGS2021} in the case when there is only a single map $\eta$ which is a conditional expectation onto a subalgebra $D$, and they can be viewed as an operator-valued extension of the trace polynomials that often arise in free probability and random matrices \cite{Cebron2013,DHK2013}. 
		
		Like non-commutative polynomials, covariance polynomials admit a natural abstraction that can be evaluated in any concrete setting $(B,\eta)$ as above. We define a \emph{covariance law} to be a positive state on the abstract space of covariance polynomials, and the space of such laws is studied in Section~\ref{subsec:covariance_laws}. It is shown that, up to a boundedness condition, all such laws arise from a tracial von Neumann algebra equipped with an operator-valued covariance matrix (see Proposition~\ref{prop:concrete_covariance_laws}). This implies, for instance, that $B$, $\tau$, and $\eta$ from Theorem~\ref{introthm: weak convergence} are all guaranteed to exist by the assumptions on the finite-dimensional data (see Proposition~\ref{prop: base algebra}). We also note that pointwise convergence in the space of covariance laws provides a convenient formalism for the weak convergence of covariance polynomials.

		\subsection{Covariance polynomials}\label{subsec:covariance_polynomials}

		In order to define covariance polynomials, we adapt the formalism of \cite[Section 9.1]{MingoSpeicher}. 
		
		\begin{defi} \label{def: cov monomial}
			Let $A$ be a unital algebra equipped with a family of linear maps $\eta_{i,j}\colon A\to A$, $i,j\in I$. Fix $k \geq 1$. 
			We define $\eta_{\pi, \vec{\imath}}[a_1,\ldots,a_{2k-1}]$ for a partition $\pi \in \mathcal{NC}_2(2k)$ and a tuple of indices $\vec{\imath} = (i_1, \ldots, i_{2k})\in I^{2k}$ recursively as follows:
			\begin{itemize}
				% \item For $k = 2$, and $\pi = \{ \{1\}, \{2\} \}$, we have $\Lambda(\eta,\pi,\vec{\imath})[b] = b$.
				\item For $k = 1$ and $\pi = \{\{1,2\}\}$, we have $\eta_{\pi, i}[a_1] = \eta_{i_1,i_2}(a_1)$.
				% \item For $k > 2$, if $\{j \} \in \pi$, then $\Lambda(\eta,\pi,\vec{\imath})[b]$ is simply the multiplication of the two halves of $\pi$ separated by $j$:
				% \[
				% \Lambda(\eta, \, \pi \restriction_{[j-1]}, \, \vec{\imath} \restriction_{[j-1]} )[b_1,\dots,b_{j-2},b_{j-1}] \cdot b_j \cdot \Lambda(\eta, \, \pi \restriction_{\{j+1, \ldots, k\}}, \, \vec{\imath} \restriction_{\{j+1, \ldots, k\}} )[b_{j+1},\dots,b_{k}].
				% \]
				\item For $k > 1$, if $\{j,j+1\} \in \pi$, then $\eta_{\pi, i}[a_1,\ldots, a_{2k-1}]$ is 
				\[
				\eta_{\pi \setminus \{j,j+1\}, \vec{\imath}\setminus\{i_j, i_{j+1}\}}[a_1,\dots,a_{j-2},a_{j-1}\eta_{i_j,i_{j+1}}(a_j)a_{j+1},a_{j+2},\dots,a_{2k-1}],
				\]
				where $\pi \setminus \{j,j+1\}$ and $\vec{\imath}\setminus\{i_j, i_{j+1}\}$ are given by restricting $\pi$ and $\vec{\imath}$ to $[2k]\setminus\{ j,j+1\}$ and reindexing the latter to identify it with $[2(k-1)]$.
			\end{itemize}
			We also take the convention for $k=0$ that $\eta_{\pi,\vi}=1$ where $\pi$ and $\vi$ are the empty partition and tuple, respectively.
		\end{defi}

		The above is best illustrated via an example with pictures.
		
		\begin{ex}\label{ex:visualizing_cov_poly}
			Fix $a_1,a_2,a_3\in A$ and a tuple $\vec{\imath}=(i_1,i_2,i_3,i_4)\in I^4$. For $\pi_1=\{\{1,4\}, \{2,3\}\}$ and $\pi_2=\{\{1,2\},\{3,4\}\}$ in $\mathcal{NC}_2(4)$, we have
			\[
			\eta_{\pi_1,\vec{\imath}}[a_1,a_2,a_3]= \eta_{i_1,i_4}(a_1 \eta_{i_2,i_3}(a_2) a_3) \qquad \text{ and } \qquad \eta_{\pi_2,\vec{\imath}}[a_1,a_2,a_3]= \eta_{i_1,i_2}(a_1) a_2 \eta_{i_3,i_4}(a_3).
			\]
			These can be visualized as follows:
			
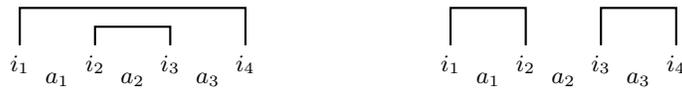
\begin{center}
				\begin{tikzpicture}[thick,font=\small]
					\path (0,0) node[] (a) {$i_1$}
					(0.5, -0.2) node[]  {$a_1$}
					(1,0) node[] (b) {$i_2$}
					(1.5, -0.2) node[]  {$a_2$}
					(2,0) node[] (c) {$i_3$}
					(2.5, -0.2) node[]  {$a_3$}
					(3,0) node[] (d) {$i_4$};
					\draw (a) -- +(0,0.75) -| (d);
					\draw (b) -- +(0,0.5) -| (c);
				\end{tikzpicture}
				\hspace*{2cm}
				\begin{tikzpicture}[thick,font=\small]
					\path (0,0) node[] (a) {$i_1$}
					(0.5, -0.2) node[]  {$a_1$}
					(1,0) node[] (b) {$i_2$}
					(1.5, -0.2) node[]  {$a_2$}
					(2,0) node[] (c) {$i_3$}
					(2.5, -0.2) node[]  {$a_3$}
					(3,0) node[] (d) {$i_4$};
					\draw (a) -- +(0,0.75) -| (b);
					\draw (d) -- +(0,0.75) -| (c);
				\end{tikzpicture}
			\end{center}
			That is, a pairing $\{j,k\}$ in a partition corresponds to an application of $\eta_{i_j,i_k}$ to the expression underneath it. Note that compositions can be achieved through nesting partitions (e.g. $\eta_{\pi_1,\vec{\imath}}[1, a, 1] = \eta_{i_1,i_4}(\eta_{i_2,i_3}(a))$ for $a\in A$), and products be achieved by adjacent pairings (e.g. $\eta_{\pi_2,\vi}[a_1,1,a_2] = \eta_{i_1,i_2}(a_1)\eta_{i_3,i_4}(a_2)$ for $a_1,a_2\in A$).
		\end{ex}
		
		The convenience of the above formalism is illustrated by the following moment characterization of a $(B,\eta)$-semicircular family.
		
		\begin{lem}\label{lem: semicircular partition formula}
			Let $(M,\tau)$ be a tracial von Neumann algebra and let $B\subset M$ be a von Neumann subalgebra equipped with a $\tau$-symmetric operator-valued covariance matrix $\eta=(\eta_{i,j})_{i,j\in I}$. Then $(x_i)_{i\in I}\subset M$ is a $(B,\eta)$-semicicular family if and only if for all $\vec{\imath}=(i_1,\ldots, i_\ell)\in I^\ell$ and $b_1,\ldots, b_{\ell-1}\in B$ one has
			\[
			E_B[x_{i_1}b_1 x_{i_2} \cdots b_{\ell-1} x_{i_\ell}] = \sum_{\pi\in \mathcal{NC}_2(\ell)} \eta_{\pi,\vec{\imath}}[b_1,\ldots, b_{\ell-1}].
			\]
		\end{lem}
		\begin{proof}
			$(\Rightarrow):$ Fix $\vec{\imath}\in I^\ell$ and $b_1,\ldots, b_{\ell-1}$. We can expand the conditional expectation on the left side of the above equation using the moment-cumulant formula (\ref{eq: moment-cumulant formula}), but Definition~\ref{def: operator-valued semicirculars} implies only non-crossing pair partitions contribute:
			\[
			E_B[x_{i_1}b_1 x_{i_2} \cdots b_{\ell-1} x_{i_\ell}] = \sum_{\pi\in \mathcal{NC}_2(\ell)} K_\pi[ x_{i_1} b_1,x_{i_2}b_2,\ldots, x_{i_{\ell-1}} b_{\ell-1}, x_{i_\ell} ].
			\]
			The claimed equality then follows by comparing the recursive formulas for $K_\pi$ and $\eta_{\pi,\vec{\imath}}$ and using that $K_2[x_{i_j} b_j, x_{i_k} b_k] = \eta_{i_j,i_k}(b_j) b_k$.\\
			
			\noindent $(\Leftarrow):$ Suppose that $x = (x_i)_{i \in I}$ satisfies the formula above.  Let $y = (y_i)_{i \in I}$ be a $(B,\eta)$-semicircular family.  By $(\Rightarrow)$, $y$ also satisfies the formula above, hence the $B$-valued moments of $x$ and $y$ agree, and so $x$ is a $(B,\eta)$-semicircular family. 
			%Since $\mathcal{NC}_2(1)=\emptyset$, the assumed formula implies $K_1[x_i] = E_B[x_i]=0$. Consequently,
			%    \[
			%        K_2[x_i,bx_j] = E_B[x_i b x_j] = \eta_{i,j}(b)
			%    \]
			%for all $i,j\in I$ and $b\in B$. We now proceed by induction on $\ell \geq 3$ to show that $K_\ell[x_{i_1} b_1,\ldots, x_{i_\ell} ]=0$ for all $\vec{\imath}=(i_1,\ldots, i_\ell)\in I^\ell$ and $b_1,\ldots, b_{\ell-1} \in B$. For the base case $\ell=3$, $\mathcal{NC}_2(3)=\emptyset$ in the assumed formula and (\ref{eq: moment-cumulant formula}) give
			%    \[
			%       0 = E_B[x_{i_1}b_1 x_{i_2} b_2 x_{i_3}] = \sum_{\pi\in \mathcal{NC}(3)} K_\pi[x_{i_1}b_1, x_{i_2} b_2, x_{i_3}] = K_3[x_{i_1}b_1, x_{i_2} b_2, x_{i_3}],
			%    \]
			%where all the other terms in the sum vanish by $K_1[x_i]=0$. Now suppose all cumulants of order $3\leq k \leq \ell-1$ vanish. Then using (\ref{eq: moment-cumulant formula}) again gives 
			%    \[
			%       \sum_{\pi\in \mathcal{NC}_2(\ell)} \eta_{\pi,\vec{\imath}}[b_1,\ldots, b_{\ell-1}] =  E_B[x_{i_1}b_1 x_{i_2} \cdots b_{\ell-1} x_{i_\ell}] = \sum_{\sigma\in \mathcal{NC}(\ell)} K_\sigma[ x_{i_1} b_1,x_{i_2}b_2,\ldots, x_{i_{\ell-1}} b_{\ell-1}, x_{i_\ell} ]
			%   \]
			%The only terms in the right-most expression that survive are $\sigma \in \mathcal{NC}_2(\ell)$ and $\sigma=\{ [\ell]\}$. The former terms sum to match the left-most expression, and so the latter term (which corresponds to the order $\ell$ cumulant) must be zero.
		\end{proof}

		\begin{defi}\label{defi:covariance_polynomial}
			Let $B$ be a von Neumann algebra equipped with an operator-valued covariance matrix $\eta=(\eta_{i,j})_{i,j\in I}$. Given a unital $*$-subalgebra $A\subset B$, a \textbf{covariance polynomial over} $(A,\eta)$ is an element of
			\[
			\text{span}\left\{ a_0 \eta_{\pi,\vec{\imath}}[a_1,\ldots, a_{2k-1}] a_{2k}\ \colon\ k\in \N,\ a_0,\ldots, a_{2k}\in A,\ \pi\in \mathcal{NC}_2(2k),\ \vec{\imath}\in I^{2k} \right\}.
			\]
			Given a tuple $b=(b_\omega)_{\omega\in \Omega}\subset B$, a \textbf{covariance polynomial over} $(b,\eta)$ is a covariance polynomial over $(\C\<b\>,\eta)$, and we denote the set of such elements by $\C_\eta\<b\>$. A \textbf{covariance monomial} over $(b,\eta)$ will mean an element of the form
			\[
			m_0(b)\eta_{\pi,\vec{\imath}}[m_1(b),\ldots, m_{2k-1}(b)] m_{2k}(b),
			\]
			where $m_0(b),\ldots, m_{2k}(b)$ are monic $*$-monomials over $b$ (possibly of degree zero).
		\end{defi}
		
		Note that a covariance monomial $m_0(b)\eta_{\pi,\vec{\imath}}[m_1(b),\ldots, m_{2k-1}(b)] m_{2k}(b)$ can be visualized as in Example~\ref{ex:visualizing_cov_poly}, where $m_0(b)$ and $m_{2k}(b)$ are placed outside of the diagram on the left and right, respectively. Also note that $\C\<b\>\subset \C_\eta \<b\>$, since by convention we can consider empty partitions and tuples.
		
		One can also consider formal covariance polynomials over a tuple of indeterminates $t=(t_\omega)_{\omega\in \Omega}$ and a family of formal linear maps $\Lambda=(\Lambda_{i,j})_{i,j\in I}$ as follows. For each $i,j\in I$, define $\Lambda_{i,j}[m(t)]$ for a monomial $m(t)\in \C\<t\>$ to be a new variable that is algebraically free from $\C\<t\>$, and then extend $\Lambda_{i,j}$ linearly to obtain a map on $\C\<t\>$. Set $\Lambda_{i,j}[m(t)]^*:=\Lambda_{j,i}[m(t)^*]$ so that these new variables are in fact $*$-algebraically free from $\C\<t\>$. Let us refer to the unital $*$-algebra generated by $\C\<t\>$ and the new variables $\{\Lambda_{i,j}[m(t)]\colon i,j\in I,\text{ monomials }m(t)\in \C\<t\>\}$ as \emph{depth $1$ covariance polynomials} (whereas $\C\<t\>$ are \emph{depth $0$ covariance polynomials}). A depth $1$ covariance \emph{monomial} is a then a product of the form
		\[
		m_1(t)\Lambda_{i_1,j_1}[m_1(t)] m_2(t) \Lambda_{i_2,j_2}[m_2(t)]\cdots \Lambda_{i_d,j_d}[m_d(t)] m_{d+1}(t).
		\]
		For each $i,j\in I$, we then define $\Lambda_{i,j}$ on such monomials as a yet another new formal variable (with appropriate adjoint) that is $*$-algebraically free from the depth $1$ covariance polynomials. Proceeding as above we obtain the \emph{depth 2 covariance polynomials}. Inductively, we obtain covariance polynomials of arbitrary depth, and their direct limit yields a unital $*$-algebra that we denote by $\C_{I}\<\Omega\>$. Note that $\C_I\<\Omega\>$ is just the unital $*$-algebra generated by a family of algebraically free variables indexed by $t$  and $\Lambda_{i,j}(f)$ for $i,j\in I$ and $f$ a depth $n$ covariance monomial with $n\in \N_0$. For each $i,j\in I$, we can define $\Lambda_{i,j}\colon \C_I\<\Omega\> \to \C_I\<\Omega\>$ by its restrictions to each depth $n$ (with values in depth $n+1$). Consequently, for $\pi\in \mathcal{NC}_2(2k)$, $\vec{\imath}\in I^{2k}$, and $p_1(t),\ldots, p_{2k-1}(t)\in \C\<t\>$ we can make sense of expressions of the form
		\[
		\Lambda_{\pi,\vi}[p_1(t),\ldots, p_{2k-1}(t)]
		\]
		as elements in $\C_I\<\Omega\>$ by expanding recursively as in Definition~\ref{def: cov monomial}. In fact, $\C_I\<\Omega\>$ is spanned by elements of the form
		\begin{align}\label{eqn:formal_cov_mon}
			m_0(t) \Lambda_{\pi,\vi}[m_1(t),\ldots,m_{2k-1}(t)] m_{2k}(t)
		\end{align} 
		for monomials $m_0(t),\ldots, m_{2k}(t)\in \C\<t\>$. We also define for each finite subset $F\subset I$ a linear map $\Lambda_F\colon \C_I\<\Omega\>\to \C_I\<\Omega\>\otimes \mathbb{M}_{|F|}$ by
		\[
		\Lambda_F(f) = (\Lambda_{i,j}(f) )_{i,j\in F} \qquad f\in \C_I\<\Omega\>.
		\]
		For $F=\emptyset$, we take the convention that $\Lambda_\emptyset$ is the identity map on $\C_I\<\Omega\>$.

		\begin{defi} \label{defi: formal covariance polynomials}
			Given index sets $\Omega$ and $I$, we call the elements of $\C_I\<\Omega\>$ \textbf{covariance polynomials}. We call an element of form in (\ref{eqn:formal_cov_mon}) a \textbf{covariance monomial}. Given a covariance polynomial
			\[
			f= \sum_{j=1}^N m_0^{(j)}(t) \Lambda_{\pi_j,\vi_j}[m_1^{(j)}(t),\ldots, m_{2k_j-1}^{(j)}(t)] m_{2k_j}^{(j)}(t),
			\]
			the \textbf{depth} of $f$ is the largest number of nested occurrences of $\Lambda$, while the \textbf{degree} of $f$ is the largest degree of the $*$-monomials $m_0^{(j)}(t) m_1^{(j)}(t)\cdots m_{2k_j}^{(j)}(t)$, $1\leq j\leq N$. 
		\end{defi}

		If we visualize a covariance monomial as in Example~\ref{ex:visualizing_cov_poly}, then its depth is simply the largest number of nested intervals. For instance, with $\pi=\pi_1,\pi_2$ as in that example, the depth of $m_0(t)\Lambda_{\pi,\vi}[m_1(t),m_2(t),m_3(t)]m_4(t)$ would be $2$ and $1$, respectively.
		
		Given a tuple $b=(b_\omega)_{\omega\in \Omega}$ in a von Neumann algebra $B$ equipped with an operator-valued covariance matrix $\eta=(\eta_{i,j})_{i,j\in I}$, there is a natural evaluation map:
		\begin{align*}
			\text{ev}_{(b,\eta)}\colon \C_I\<\Omega\>&\to \C_\eta\<b\>\\
			m_0(t) \Lambda_{\pi,\vi}[m_1(t),\ldots, m_{2k-1}(t)]m_{2k}(t) & \mapsto m_0(b) \eta_{\pi,\vi}[m_1(b),\ldots, m_{2k-1}(b)]m_{2k}(b).
		\end{align*}
		Note that this map is actually a $*$-homomorphism.

		\subsection{The space of covariance laws}\label{subsec:covariance_laws}
		
		We consider here an abstraction of our examples that will be useful later for expressing notions of weak convergence.

		\begin{defi}
			For index sets $\Omega$ and $I$, a \textbf{covariance law} over $\C_I\<\Omega\>$ is a linear functional $\mu\colon \C_I\<\Omega\>\to \C$ satisfying the following conditions:
			\begin{itemize}
				\item $\mu(1)=1$; \hfill (\textbf{unital})
				
				% \item $\mu(f^\sharp_.f)\geq 0$ for all $f\in \C_\bullet\<x_\omega\colon \omega\in \Omega\>$;
				
				\item $\mu(fg)=\mu(gf)$ for all $f,g\in \C_I\<\Omega\>$; \hfill (\textbf{tracial})
				
				\item $\mu(\Lambda_{i,j}(f) g) = \mu(f \Lambda_{j,i}(g))$ for all $i,j\in I$ and $f,g\in \C_I\<\Omega\>$; \hfill (\textbf{symmetric})
				
				\item for each finite subset $F\subset I$ and $n\in \N$ \hfill (\textbf{completely positive})
				\[
				\mu\otimes \tr_{|F|}\otimes \tr_n( (\Lambda_F\otimes I_n)(f^*f) g^*g)\geq 0 
				\]
				for $f\in \C_I\<\Omega\> \otimes \mb{M}_n$ and $g\in \C_I\<\Omega\>\otimes \mb{M}_{|F|}\otimes \mb{M}_n$ .
			\end{itemize}
			Consider tuples of positive numbers
			\[
			R:=(R_\omega)_{\omega\in \Omega} \subset (0,\infty) \qquad \text{ and } \qquad S:=(S_{i,j})_{i,j\in I} \subset (0,\infty).
			\]
			For a monic $*$-monomial $m(t) = t_{\omega_1}^{\epsilon_1}\cdots t_{\omega_d}^{\epsilon_d}$ with $\epsilon_1,\ldots, \epsilon_d\in \{1,*\}$, denote $R[m(t)] = R_{\omega_1}\cdots R_{\omega_d}$, and set $R[1]=1$.  
			We say $\mu$ is \textbf{$(R,S)$-exponentially bounded} if for all covariance monomials one has
			\[
			|\mu( m_0(t) \Lambda_{\pi,\vi}[m_1(t),\ldots, m_{2k-1}(t)] m_{2k}(t))| \leq \prod_{\{j,k\}\in \pi} S_{i_j, i_k} \prod_{j=0}^{2k} R[m_j(t)].
			\]
			% for $f=m_0(t)\Lambda_{\pi,\vi}[m_1(t),\ldots, m_{2k-1}(t)] m_{2k}(t)$ with $m_0(t),\ldots, m_{2k}(t)$ $*$-monomials, $F\subset I$ a finite subset, and $B\in \C_I\<\Omega\>\otimes \mathbb{M}_{|F|}$ one has
			%     \[
			%         \mu\otimes \tr_{|F|}( B^* \Lambda_F( f^* f) B) \leq S^{2p+1} R[m_0(t)\cdots m_{2k}(t)]^2 \mu\otimes \tr_{|F|}(B^*B),
			%     \]
			% where $p$ is the number of pairs in $\pi$. 
			We denote by $\mathbb{L}^\infty(\Omega,I)$ the set of covariance laws over $\C_I\<\Omega\>$ that are $(R,S)$-exponentially bounded for some functions $(R,S)$.
		\end{defi}
		
		Let $B$ be a von Neumann algebra equipped with a faithful normal tracial state $\tau$ and a $\tau$-symmetric operator-valued covariance matrix $\eta=(\eta_{i,j})_{i,j\in I}$.  For any tuple $b=(b_\omega)_{\omega\in \Omega} \subset B$, consider the map $\tau_{(b,\eta)}\colon \C_I\<\Omega\>\to \C$ defined by
		\[
		\tau_{(b,\eta)}(f):= \tau( f(b,\eta)) \qquad f\in \C_I\<\Omega\>,
		\]
		where we recall that $f(b,\eta) = \text{ev}_{(b,\eta)}(f)$ from the previous section. It is straightforward to verify that $\tau_{(b,\eta)}$ is a covariance law over $\C_I\<\Omega\>$. Moreover, if we set
		\[
		R_\omega:= \|b_\omega\|,\ \omega\in \Omega \qquad \text{ and } \qquad S_{i,j}:= \|\eta_{i,j}\|,\ i,j\in I,
		\]
		then $\tau_{(b,\eta)}$ is $\left( (R_\omega)_{\omega\in \Omega}, (S_{i,j})_{i,j\in I} \right)$-exponentially bounded. Thus $\tau_{(b,\eta)}\in \mathbb{L}^\infty(\Omega,I)$. The following proposition tells us that, in fact, every element of $\mathbb{L}^\infty(\Omega,I)$ arises in this way.

		\begin{prop}\label{prop:concrete_covariance_laws}
			Let $\Omega$ and $I$ be index sets. For every $\mu\in \mathbb{L}^\infty(\Omega,I)$ there exists a von Neumann algebra $B$, a faithful normal tracial state $\tau$ on $B$, a $\tau$-symmetric operator-valued covariance matrix $\eta= (\eta_{i,j})_{i,j\in I}$ on $B$, and a generating tuple $b=(b_\omega)_{\omega\in \Omega}\subset B$ so that $\tau_{(b,\eta)}= \mu$. Moreover, if $\mu$ is $(R,S)$-exponentially bounded, then $\|b_\omega\|\leq R_\omega$ for all $\omega\in \Omega$ and
			\[
			\|\eta_{i,j}\| \leq \left(S_{i,i}^2 + S_{i,j}^2 + S_{j,i}^2 + S_{j,j}^2 \right)^{\frac12},
			\]
			for all $i,j\in I$.
		\end{prop}
		\begin{proof}
			Assume that $\mu$ is $(R,S)$-exponentially bounded. Recall that $\C_I\<\Omega\>$ is the free $*$-algebra generated by indeterminates $t_\omega$ for $\omega\in \Omega$ and $\Lambda_{i,j}(f)$ for $i,j\in I$ and covariance monomials $f$, and by definition $\mu$ defines an exponentially bounded non-commutative law on this free $*$-algebra. (Recall that $\Lambda_\emptyset$ is the identity map by convention so that the positivity of $\mu$ follows from its complete positivity.) Thus if we let $J$ denote the indexing set for these generating indeterminates, then there exists a von Neumann algebra $B$ generated by a tuple $x=(x_j)_{j\in J}$ and admitting a faithful normal tracial state $\tau$ so that $\mu$ equals the law of $x$ with respect to $\tau$ (see \cite[Proposition 5.2.14, Theorem 5.2.24, Corollary 5.2.27]{AGZ2009}). The evaluation map $\text{ev}_x$ composed with the isomorphism $\C_I\<\Omega\>\cong \C\<J\>$ yields a unital $*$-homomorphism
			\[
			\pi\colon \C_I\<\Omega\> \to B
			\]
			satisfying
			\begin{align}\label{eqn:expo_bound_to_actual_bound}
				\| \pi\left( m_0(t) \Lambda_{\pi,\vi}[m_1(t),\ldots, m_{2k-1}(t)] m_{2k}(t)\right)\| \leq \prod_{\{j,k\}\in \pi} S_{i_j, i_k} \prod_{j=0}^{2k} R[m_j(t)].
			\end{align}
			Let us denote $b_\omega:=\pi(t_\omega)=x_\omega$ for each $\omega\in \Omega \subset J$.
			
			Now, for each $i,j\in I$ and $f\in \C_I\<\Omega\>$, define a linear map on the $*$-algebra $\pi(\C_I\<\Omega\>) = \C\<x\>$ by
			\[
			\eta_{i,j}(\pi(f)):=\pi( \Lambda_{i,j}(f)).
			\]
			This is well-defined since if $\pi(f)=0$, then using the symmetry of $\mu$ we have
			\[
			\| \eta_{i,j}(\pi(f)) \|_\tau^2 = \mu( \Lambda_{i,j}(f)^* \Lambda_{i,j}(f)) = \mu( \Lambda_{j,i}(\Lambda_{i,j}(f)^*) f) = \tau( \pi( \Lambda_{j,i}( \Lambda_{i,j}(f)^*)) \pi(f)  ) =0,
			\]
			and hence $\eta_{i,j}(\pi(f))=0$ by the faithfulness of $\tau$. We also define for each finite subset $F\subset I$ the map
			\[
			\eta_F( \pi(f)) :=( \eta_{i,j}( \pi(f) ) )_{i,j\in F}.
			\]
			By \cite[Proposition 7.3.4]{AP2017}, $L^2(B\otimes \mathbb{M}_{|F|}\otimes \mathbb{M}_n)_+$ is a self-dual cone for all $n\in \N$ with dense subset
			\[
			\{H^* H\colon H\in \pi(\C_I\<\Omega\>)\otimes \mathbb{M}_{|F|}\otimes \mathbb{M}_n\},
			\]
			and so it follows from the complete positivity of $\mu$ that $\eta_F\otimes I_n(\pi(f)^*\pi(f)) \geq 0$ for all $f\in \C_I\<\Omega\>\otimes \mathbb{M}_n$. In particular, $\eta_F(1)$ is a positive element of $B\otimes \mathbb{M}_{|F|}$.
			
			We next show $\|\eta_F( \pi(f) ) \|\leq \| \eta_F(1)\| \| \pi(f)\|$ for all finite subsets $F\subset I$ and $f\in \C_I\<\Omega\>$. Indeed, the $2$-positivity of $\eta_F$ on $\mathcal{A}$ implies
			\[
			0 \leq (\eta_F\otimes I_2)\left(\begin{array}{cc} 1 & \pi(f) \\ 0 & 0 \end{array} \right)^* \left(\begin{array}{cc} 1 & \pi(f) \\ 0 & 0 \end{array} \right) = \left(\begin{array}{cc} \eta_F(1) & \eta_F(\pi(f)) \\ \eta_F(\pi(f))^* & \eta_F(\pi(f)^*\pi(f)) \end{array} \right),
			\]
			from which it follows that $\eta_F(\pi(f))^*\eta_F(\pi(f)) \leq \|\eta_F(1)\|\eta_F\left(\pi(f)^*\pi(f)\right)$ (see, for example, \cite[Exercise 3.2]{Paulsen2002}). Let  $(g_n)_{n\in \N}\subset \C_I\<\Omega\>$ be a sequence so that $\pi(g_n)$ approximates $(\|\pi(f)\|^2 - \pi(f)^*\pi(f))^{1/2}$ in norm. For $H=( \pi(h_{i,j}) )_{i,j\in F}\in \pi(\C_I\<\Omega\>) \otimes \mb{M}_{|F|}$, the symmetry condition on $\mu$ gives
			\begin{align*}
				\< \eta_F( \|\pi(f)\|^2 - \pi(f)^* \pi(f)) H, H\>_{\tau\otimes \tr_{|F|}} &= \frac{1}{|F|} \sum_{i,j,k\in F} \tau\left( \pi(h_{j,i})^* \eta_{j,k}[\|\pi(f)\|^2 - \pi(f)^* \pi(f)] \pi(h_{k,i}) \right)\\
				&=\frac{1}{|F|} \sum_{i,j,k\in F} \tau\left(  [\|\pi(f)\|^2 - \pi(f)^* \pi(f)] \eta_{k,j}( \pi(h_{k,i} h_{j,i}^*) ) \right)\\
				&= \lim_{n\to\infty} \frac{1}{|F|} \sum_{i,j,k\in F} \tau\left(  \pi( g_n^*g_n ) \eta_{k,j}(\pi( h_{k,i} h_{j,i}^*)) \right)\\
				&= \lim_{n\to\infty} \< \eta_F( \pi( g_n^* g_n) ) H, H\>_{\tau\otimes \tr_{|F|}} \geq 0.
			\end{align*}
			Thus $\eta_F( \|\pi(f)\|^2 - \pi(f)^* \pi(f))\geq 0$, and combined with the previous inequality we have
			\[
			\eta_F(\pi(f))^* \eta_F(\pi(f)) \leq \|\eta_F(1)\| \eta_F( \|\pi(f)\|^2) = \|\eta_F(1)\| \eta_F(1) \|\pi(f)\|^2.
			\]
			Taking operator norms then gives the claimed inequality.
			
			The above implies that each $\eta_F$, and hence each $\eta_{i,j}$, admits a unique bounded extension to the norm closure of $\pi(\C_I\<\Omega\>)$. In order to show that the $\eta_{i,j}$ can be further extended to normal maps on $B$, we first show they admit $2$-norm bounded extensions to $L^2(B,\tau)$. Fix $i,j\in I$ and observe that $\eta_{i,j}(\pi(f))^* = \eta_{j,i}( \pi(f) ^*)$ for all $f\in \C_I\<\Omega\>$ since $\eta_{\{i,j\}}$ extends to a (completely) positive map on the norm closure of $\pi(\C_I\<\Omega\>)$. This along with the symmetry of $\mu$ gives
			\[
			\| \eta_{i,j}(\pi(f)) \|_\tau^2 = \tau\left( \eta_{j,i}(\pi(f)^*) \eta_{i,j}(\pi(f)) \right) = \tau\left( \pi(f)^* \eta_{i,j}^2( \pi(f) ) \right) \leq \| \pi(f)\|_\tau \| \eta_{i,j}^2( \pi(f) ) \|_\tau.
			\]
			Iterating this $d+1$ times we obtain
			\[
			\| \eta_{i,j}( \pi(f))\|_\tau^2 \leq \|\pi(f)\|_\tau^{1+\frac12 + \cdots + \frac{1}{2^d}} \| \eta_{i,j}^{2^{d+1}}( \pi(f) ) \|_\tau^{\frac{1}{2^{d}}}.
			\]
			Since $\| \eta_{i,j}^{2^{d+1}}( \pi(f) ) \|_\tau \leq \| \eta_{i,j}^{2^{d+1}}( \pi(f))\| \leq \|\eta_{i,j}\|^{2^{d+1}} \|\pi(f)\|$, it follows that
			\[
			\| \eta_{i,j}(\pi(f))\|_\tau^2 \leq \limsup_{d\to\infty} \|\pi(f)\|_\tau^2 \|\eta_{i,j}\|^{2} \|\pi(f)\|^{\frac{1}{2^d}} = \|\pi(f)\|_\tau^2 \|\eta_{i,j}\|^2.
			\]
			Consequently, each $\eta_{ij}$ extends to a bounded map on $L^2(B,\tau)$. For $a\in B$, we claim that $\eta_{i,j}(a)\in B$. Indeed, by the Kaplansky Density Theorem, let $(f_n)_{n\in \N} \subset \C_I\<\Omega\>$ be a sequence satisfying $\|\pi(f_n)\| \leq \|a\|$ and $\| \pi(f_n) - a \|_\tau \to 0$. Then $( \eta_{i,j}( \pi(f_n) ) )_{n\in \N}$ converges to $\eta_{i,j}(a)$ in $2$-norm, but since $\| \eta_{i,j}( \pi(f_n) ) \| \leq \| \eta_{i,j}(1)\| \| a\|$ for all $n\in \N$ it must be that $\eta_{i,j}(a)\in B$ (see \cite[Proposition III.5.3]{TakesakiI}). 
			To see that $\eta_{i,j}$ is normal, it suffices to show that it is weak operator topology continuous on the unit ball of $B$ (this is a standard consequence of the Krein-\v Smulian theorem). 
			Restricting to this norm ball further reduces the task to showing $a\mapsto \< \eta_{i,j}(a) \xi, \eta\>_\tau$ is weak operator topology continuous for $\xi,\eta$ varying the dense subspace $\pi(\C_I\<\Omega\>)$. But this is a consequence of the symmetry of $\mu$ since
			\[
			\< \eta_{i,j}(a) \pi(f), \pi(g)\>_\tau = \tau( \pi(g)^* \eta_{i, j}(a) \pi(f) ) = \tau( \eta_{j,i}( \pi(f)\pi(g)^*) a)
			\]
			is normal as a map on $a$ for all $f,g\in \C_I\<\Omega\>$. 
			
			Hence each $\eta_{i,j}\colon B\to B$ is normal and therefore $\eta:=(\eta_{i,j})_{i,j}$ forms a $\tau$-symmetric operator-valued covariance matrix on $B$ by inheriting the remaining properties from $\mu$. The equality $\pi(f) = \text{ev}_{(b,\eta)}(f)$ for $f\in \C_I\<\Omega\>$ can be verified by inducting on the depth of $f$, where the base case and induction step use our definitions of $b_\omega$ and $\eta_{i,j}$, respectively. Thus $\tau_{(b,\eta)} = \tau\circ \text{ev}_{(b,\eta)} = \tau\circ \pi = \mu$ by our initial construction of $(B,\tau)$.
			
			Finally, the bound $\|b_\omega\|\leq R_\omega$ follows from (\ref{eqn:expo_bound_to_actual_bound}) applied to the covariance monomial $t_\omega$ for each $\omega\in \Omega$. To see the bound for $\|\eta_{i,j}\|$, $i,j\in I$,  we note that
			\[
			\| \eta_{i,j}\| \leq \| \eta_{\{i,j\}}\| = \| \eta_{\{i,j\}}(1)\| \leq \left( \sum_{k,\ell\in \{i,j\}} \| \eta_{k,\ell}(1)\|^2 \right)^{\frac12} \leq \left( \sum_{k,\ell\in \{i,j\}} S_{k,\ell}^2 \right)^{\frac12},
			\]
			where the last inequality follows from (\ref{eqn:expo_bound_to_actual_bound}) applied to the covariance monomials $\Lambda_{k,\ell}(1)$.
		\end{proof}
		
		As an application of the previous proposition, we show that $B$, $\tau$, and $\eta$ in Theorem~\ref{introthm: weak convergence} are guaranteed to exist by the assumptions on the finite-dimensional data. We remark that the analogue of this result in the scalar-valued case is that the $n(k)\times n(k)$ identity matrices converge in distribution to the identity of the limiting von Neumann algebra.
		
		\begin{prop} \label{prop: base algebra}
			Fix index sets $\Omega, I$ and a sequence of positive integers $(n(k))_{k\in \N}\subset \N$. Further fix for each $k\in \N$ a tuple $b^{(k)}:=(b^{(k)}_\omega)_{\omega\in \Omega}\subset \mb{M}_{n(k)}$ and a $\tr_{n(k)}$-symmetric covariance matrix $\eta^{(k)}=(\eta^{(k)}_{i,j})_{i,j\in I}$ on $\mb{M}_{n(k)}$. Assume $\sup_k \|b_\omega^{(k)}\| <\infty$ for all $\omega\in \Omega$, $\sup_{k} \|\eta_{i,j}^{(k)}\| <\infty$ for all $i,j\in I$, and that
			\[
			\lim_{k\to\infty} \tr_{n(k)}\left[ f(b^{(k)},\eta^{(k)})\right]
			\]
			exists for all covariance polynomials $f\in \C_I\<\Omega\>$. Then there exists a von Neumann algebra $B$, a faithful normal tracial state $\tau$ on $B$, a $\tau$-symmetric operator-valued covariance matrix $\eta= (\eta_{i,j})_{i,j\in I}$ on $B$, and a generating tuple $b=(b_\omega)_{\omega\in \Omega}\subset B$ so that the above limit is given by $\tau(f(b,\eta))$ for all covariance polynomials $f$.
		\end{prop}
		\begin{proof}
			Define a linear functional $\mu$ on covariance polynomials by
			\[
			\mu(f):= \lim_{k\to\infty} \tr_{n(k)}\left[ f( b^{(k)},\eta^{(k)})\right].
			\]
			That $\mu$ is a covariance law follows immediately from the properties of the covariance matrices $\eta^{(k)}$. Moreover, the covariance law determined by $\tr_{n(k)}$ and $\eta^{(k)}$ is $( (\|b_\omega^{(k)}\|)_{\omega\in \Omega}, (\|\eta_{i,j}^{(k)}\|)_{i,j\in I})$-exponentially bounded. So letting
			\[
			R:=\left( \sup_{k\in \N}  \|b_\omega^{(k)}\| \right)_{\omega\in \Omega} \qquad \text{ and } \qquad S:= \left(\sup_{k\in \N} \| \eta_{i,j}^{(k)}(1)\| \right)_{i,j\in I},
			\]
			it follows that $\mu$ is $(R,S)$-exponentially bounded. Appealing to Proposition~\ref{prop:concrete_covariance_laws} then completes the proof.
		\end{proof}

		\section{Proof of weak convergence} \label{sec: weak convergence}
		
		The following notation will be used throughout this section. Let $\eta=(\eta_{i,j})_{i,j\in I}$ be a $\tr_n$-symmetric $\mb{M}_n$-valued covariance matrix and let $X = (X_i)_{i \in I}$ be an $\eta$-Gaussian family.  For each pair partition $\pi \in \mathcal{P}_2(\ell)$ and each $\vec{\imath} = (i_1,\dots,i_\ell)\in I^\ell$, define the $\ell$-tuple $X_{\pi, \vec{\imath}} = \left(X_{\pi,\vec{\imath};r}\right)_{r \in [\ell]}$ as follows.  Let $Y_{\{r,s\}} = \left(Y_{\{r,s\},i}\right)_{i \in I}$ be independent copies of $X$ indexed by the blocks $\{r,s\}\in \pi$.  For each block $\{r,s\}$, set $X_{\pi,\vec{\imath};r} = Y_{\{r,s\},i_r}$ and $X_{\pi,\vec{\imath};s} = Y_{\{r,s\},i_s}$. When $\eta$, $X$, and $X_{\pi,\vec{\imath}}$ are needed over varying matrix algebras $(\mb{M}_{n(k)})_{k\in \N}$, we will decorate them (and their entries) with a superscript: $\eta^{(k)}$, $X^{(k)}$, and $X^{(k)}_{\pi,\vec{\imath}}$.
		
		\subsection{Wick expansion} \label{subsec: Wick expansion}
		
		As usual in random matrix theory, we want to expand moments as a sum over pair partitions using the Wick formula. The following lemma is an immediate consequence of the Wick formula, and is an adaptation of \cite[Corollary 7.4]{BBvH2023} with deterministic matrices $B_j$ inserted between the $X$'s.
		
		\begin{lem} \label{lem: matrix Wick formula}
			Let $X$ and $X_{\pi,\vec{\imath}}$ be as above.  Let $b_1$, \dots, $b_{\ell-1}\in \mathbb{M}_n$ be deterministic matrices.  Then
			\[
			\mathbb{E} [X_{i_1} b_1 X_{i_2} \cdots b_{\ell-1} X_{i_\ell}] = \sum_{\pi \in \mathcal{P}_2(\ell)} \mathbb{E}[X_{\pi,\vec{\imath};1} b_1 X_{\pi,\vec{\imath};2} \cdots b_{\ell-1} X_{\pi,\vec{\imath};\ell}].
			\]
		\end{lem}
		
		\begin{proof}
			It suffices to prove that the left-hand side and right-hand side agree after we apply $\tr_n(\ \cdot\ b_\ell)$ to both sides for an arbitrary deterministic matrix $b_\ell$.  Since $X$ is jointly Gaussian, it can be written as $X_i = \sum_{s=1}^m g_s a_{i,s}$ where $a_{i,s}$ are deterministic matrices and $g_i$ are independent standard normal random variables.  Then by the Wick formula,
			\begin{align*}
				\mathbb{E} \circ \tr_n[X_{i_1} b_1 \cdots b_{\ell-1} X_{i_\ell} b_\ell] &= \sum_{s_1, \dots, s_\ell \in [m]} \mathbb{E}[g_{s_1} \cdots g_{s_\ell}] \tr_n[a_{i_1,s_1} b_1 \cdots a_{i_\ell,s_\ell} b_\ell] \\
				&= \sum_{\pi \in \mathcal{P}_2(\ell)} \sum_{s_1, \dots, s_\ell \in [m]} \prod_{\{u,v\} \in \pi} \Cov(g_{s_u} g_{s_v}) \tr_n[a_{i_1,s_1} b_1 \cdots a_{i_\ell,s_\ell} b_\ell].
			\end{align*}
			As for the right-hand side, we can expand $\mathbb{E}[X_{\pi,\vec{\imath};1} b_1 X_{\pi,\vec{\imath},2} \cdots b_{\ell-1} X_{\pi,\vec{\imath},\ell}]$ in the same manner as a sum over partitions $\pi' \in \mathcal{P}_2(\ell)$, and in this case only the partition $\pi' = \pi$ will contribute.  Hence, summing this expression over $\pi$ yields $\mathbb{E} \circ \tr_n[X_{i_1} b_1 \cdots b_{\ell-1} X_{i_\ell} b_\ell]$, as desired.
		\end{proof}
		
		The next lemma evaluates the terms corresponding to non-crossing partitions in terms of $\eta$, obtaining the expression analogous to the moments of an operator-valued semicircular operator (see Lemma~\ref{lem: semicircular partition formula}).
		
		\begin{lem} \label{lem: evaluation of NC terms}
			Let $X$ and $X_{\pi,\vec{\imath}}$ be as above.  Suppose that $\pi$ is a non-crossing pair partition.  Then for $b_1,\dots,b_{\ell-1} \in \mathbb{M}_{n}$, we have
			\[
			\mathbb{E}[X_{\pi,\vec{\imath};1} b_1 X_{\pi,\vec{\imath};2} \cdots b_{\ell-1} X_{\pi,\vec{\imath};\ell}] = \eta_{\pi,\vec{\imath}}[b_1,\dots,b_{\ell-1}]. 
			\]
		\end{lem}
		\begin{proof}
			We proceed by induction on the number of blocks in the partition.  The case where $\pi$ has one block is immediate from the defining equation (\ref{eq: matrix covariance}) in the definition of an $\eta$-Gaussian family.  Consider a non-crossing pair partition $\pi$ with more than one block.  Recall that an \emph{outer block} is a block that is not nested inside any other block.  Since $\pi$ is non-crossing, if it has more than one outer block, then we can write $\pi$ as the concatenation of two partitions $\pi_1 \in \mathcal{NC}_2(\ell_1)$ and $\pi_2 \in \mathcal{NC}_2(\ell_2)$ with $\ell_1 + \ell_2 = \ell$.  Correspondingly, let $\vec{\imath}_1$ and $\vec{\imath}_2$ be the restrictions of $\vec{\imath}$ to $\{1,\dots,\ell_1\}$ and $\{\ell_1+1,\dots,\ell\}$.  Then the variables $(X_{\pi,\vec{\imath};r})_{r \in \{1,\dots,\ell_1\}}$ and $(X_{\pi,\vec{\imath};r})_{r \in \{\ell_1+1,\dots,\ell\}}$ are independent.  Thus,
			\begin{align*}
				\mathbb{E}[X_{\pi,\vec{\imath};1} & b_1 X_{\pi,\vec{\imath};2} \cdots b_{\ell-1} X_{\pi,\vec{\imath};\ell}] \\
				&= \mathbb{E}[X_{\pi,\vec{\imath};1} b_1 X_{\pi,\vec{\imath};2} \cdots b_{\ell_1-1} X_{\pi,\vec{\imath};\ell_1}] b_{\ell_1} \mathbb{E}[X_{\pi,\vec{\imath};\ell_1+1} b_{\ell_1+1} X_{\pi,\vec{\imath};\ell_1+2} \cdots b_{\ell-1} X_{\pi,\vec{\imath};\ell}] \\
				&= \mathbb{E}[X_{\pi_1,\vec{\imath}_1;1} b_1 X_{\pi_1,\vec{\imath}_1;2} \cdots b_{\ell_1-1} X_{\pi_1,\vec{\imath}_1;\ell_1}] b_{\ell_1} \mathbb{E}[X_{\pi_2,\vec{\imath}_2;1} b_{\ell_1+1} X_{\pi_2,\vec{\imath}_2;\ell_1+2} \cdots b_{\ell-1} X_{\pi_2,\vec{\imath}_2;\ell}] \\
				&= \eta_{\pi_1,\vec{\imath}_1}[b_1,\dots,b_{\ell_1-1}] b_{\ell_1} \eta_{\pi_2,\vec{\imath}_2}[b_{\ell_1+1},\dots,b_\ell] \\
				&= \eta_{\pi,\vec{\imath}}[b_1,\dots,b_\ell],
			\end{align*}
			where we have applied the inductive hypothesis to $\pi_1$ and $\pi_2$ and the recursive definition of $\eta_{\pi,\vec{\imath}}$.
			
			On the other hand, if $\pi$ has only one outer block, then $\{1,\ell\} \in \pi$.  Let $\pi' = \{\{r,s\} \subseteq [\ell-2]: \{r+1,s+1\} \in \pi \}$ be the partition on $[\ell-2]$ obtained from the remaining blocks.  Correspondingly, let $\vec{\imath}'$ be the string of indices obtained by restricting to $\{2,\dots,\ell-1\}$.  Recall that $X_{\pi,\vec{\imath};1}$ and $X_{\pi,\vec{\imath};\ell}$ are independent of the other variables.  Let $\mathcal{F}$ be the $\sigma$-algebra generated by $X_{\pi,\vec{\imath};1}$ and $X_{\pi,\vec{\imath};\ell}$.  Hence,
			\begin{align*}
				\mathbb{E}[X_{\pi,\vec{\imath};1} b_1 X_{\pi,\vec{\imath};2} \cdots b_{\ell-1} X_{\pi,\vec{\imath};\ell}] &= \mathbb{E} [\mathbb{E}[X_{\pi,\vec{\imath};1} b_1 X_{\pi,\vec{\imath};2} \cdots b_{\ell-1} X_{\pi,\vec{\imath};\ell} \mid \mathcal{F}] ] \\
				&= \mathbb{E} [X_{\pi,\vec{\imath};1} b_1 \, \mathbb{E}[ X_{\pi,\vec{\imath};2} b_2 \cdots b_{\ell-2} X_{\pi,\vec{\imath};\ell-1} \mid \mathcal{F}] \, b_{\ell-1} X_{\pi,\vec{\imath};\ell} ] \\
				&= \mathbb{E} [X_{\pi,\vec{\imath};1} b_1 \, \mathbb{E}[ X_{\pi,\vec{\imath};2} b_2 \cdots b_{\ell-2} X_{\pi,\vec{\imath};\ell-1}] \, b_{\ell-1} X_{\pi,\vec{\imath};\ell} ] \\
				&= \eta_{i_1,i_\ell}[b_1 \, \eta_{\pi',\vec{\imath}}[b_2,\dots,b_{\ell-2}] \, b_{\ell-1}] \\
				&= \eta_{\pi,\vec{\imath}}[b_1,\dots,b_\ell].\qedhere
			\end{align*}
		\end{proof}
		
		Thus, the main task to obtain convergence is to bound the terms coming from partitions with crossings.  In the next subsection, we recall the ingredients from \cite{BBvH2023} to estimate this.
		
		\subsection{Bounds for crossing terms} \label{subsec: bounds for crossing}
		
		A key ingredient in the analysis of \cite{BBvH2023} is the \emph{matrix alignment parameter} first introduced by Tropp \cite[\S 3.2]{Tropp2018}.  When $X$ and $X'$ are Gaussian random matrices of mean zero, one defines
		\[
		w(X,X')^4 = \sup_{U,V,W \in \mathbb{U}_n} \norm{\mathbb{E}[XU\tilde{X}'VXW\tilde{X}'] },
		\]
		where $\tilde{X}'$ is a copy of $X'$ independent of $X$.  Note that we can just as well replace unitary matrices with any matrices of norm $\leq 1$ since these are the convex hull of the unitary group (see \cite[Proof of Lemma 4.5]{BBvH2023}).  The importance of $w(X,X')$ for is that it can be used to control the contributions of crossings in the expansion of moments of the Gaussian matrix \cite{Tropp2018}.
		
		The alignment parameter $w(X,X')$ is difficult to compute, so \cite{BBvH2023} give practical bounds in terms of more natural parameters.  The first parameter $\sigma(X)$ is given by $\sigma(X)^2 = \norm{\mathbb{E}X^2}$, and in our terminology:
		\[
		\sigma(X) = \norm*{\sum_{j=1}^m a_j^2 }^{1/2} = \norm*{\mathbb{E} X^2}^{1/2} = \norm*{\eta(1)}^{1/2},
		\]
		where $a_j$ are deterministic matrices such that $X = \sum_{j=1}^m g_j a_j$ for independent normal random variables $g_j$ as described in Section \ref{subsec: eta Gaussian ensemble}.
		The second main parameter of \cite{BBvH2023} is defined as $v(X) = \norm{\Cov(X)}^{1/2}$, where $\Cov(X) \in \mathbb{M}_{n^2}$ is the covariance matrix.  In the discussion following Definition~\ref{def: eta-Gaussian} we showed that the covariance matrix is simply the Choi matrix $T$ associated to the completely positive map, and so
		\[
		v(X) = \norm{\Cov(X)}^{1/2} = \norm{T}^{1/2}.
		\]
		
		\begin{prop}[{\cite[Proposition 4.6]{BBvH2023}}]\label{prop: matrix alignment estimate}
			For Gaussian matrices $X$ and $X'$ with mean zero, one has
			\[
			w(X,X')^4 \leq v(X) \sigma(X) v(X') \sigma(X').
			\]
		\end{prop}
		
		We adapt the matrix alignment parameter to several matrices as follows.
		
		\begin{defi}
			Let $\eta = (\eta_{i,j})_{i,j \in I}$ be a $\tr_n$-symmetric $\mb{M}_n$-valued covariance matrix.  Let $X = (X_i)_{i \in I}$ and $X' = (X_i)_{i \in I}$ be independent $\eta$-Gaussian families.  Let $(\mathbb{M}_n)_1$ denote the unit ball in operator norm in $\mathbb{M}_n$.  For $i, j, i', j'$, define
			\[
			w(\eta,i,j,i',j') = \sup_{U,V,W \in (\mathbb{M}_n)_1} \norm{ \mathbb{E}[X_iUX_{i'}'VX_jWX_{j'}']}^{1/4}
			\]
		\end{defi}
		
		\begin{lem} \label{lem: matrix alignment estimate 2}
			Let $\eta = (\eta_{i,j})_{i,j \in I}$ be a $\tr_n$-symmetric $\mb{M}_n$-valued covariance matrix, and for each finite $F\subseteq I$ let $T_F$ be the Choi matrix associated to $\eta_F=(\eta_{i,j})_{i,j\in F}$. Then for all $i,j,i',j'\in F$ we have
			\[
			w(\eta,i,j,i',j')^4 \leq 2 \norm{T_F} \norm{\eta_F(1)}.
			\]
		\end{lem}
		
		\begin{proof}
			Consider the random matrices $Y = X_i \oplus X_j$ and $Y' = X_{i'}' \oplus X_{j'}'$ in $\mathbb{M}_{2n} = \mathbb{M}_2 \otimes \mathbb{M}_n$.  For $U, V, W \in (\mathbb{M}_n)_1$, observe that
			\[
			\begin{bmatrix} X_iUX_{i'}'VX_jWX_{j'}' & 0 \\ 0 & 0 \end{bmatrix}
			=
			\begin{bmatrix} X_i & 0 \\ 0 & X_j \end{bmatrix}
			\begin{bmatrix} U & 0 \\ 0 & 0 \end{bmatrix}
			\begin{bmatrix} X_{i'}' & 0 \\ 0 & X_{j'}' \end{bmatrix}
			\begin{bmatrix} 0 & V \\ 0 & 0 \end{bmatrix}
			\begin{bmatrix} X_i & 0 \\ 0 & X_j \end{bmatrix}
			\begin{bmatrix} 0 & 0 \\ 0 & W \end{bmatrix}
			\begin{bmatrix} X_{i'}' & 0 \\ 0 & X_{j'}' \end{bmatrix}
			\begin{bmatrix} 0 & 0 \\ I & 0 \end{bmatrix},
			\]
			and therefore
			\[
			w(\eta,i,j,i',j')^4 \leq w(Y,Y')^4.
			\]
			By Proposition~\ref{prop: matrix alignment estimate}, this is bounded by $v(Y) \sigma(Y) v(Y') \sigma(Y')$.  We note that
			\[
			\sigma(Y)^2 = \norm*{ \mathbb{E}[(X_i \oplus X_j)^2] } = \norm*{\eta_{i,i}(1) \oplus \eta_{j,j}(1)} \leq \norm{\eta_F(1)},
			\]
			and similarly for $\sigma(Y')$, so that $\sigma(Y) \sigma(Y') \leq \norm{\eta_F(1)}$.  To estimate $v(Y)$,  using the decomposition of $\mathbb{M}_{2n}$ into $2 \times 2$ blocks, we see that
			\begin{align*}
				\Cov(Y) &= \sum_{s,t \in [2n]} \mathbb{E}[Y E_{s,t} Y] \otimes E_{t,s}^{\op} \in (\mathbb{M}_2 \otimes \mathbb{M}_n) \otimes (\mathbb{M}_2 \otimes \mathbb{M}_n)^{\op} \\
				&= \sum_{s,t \in [n]} (E_{1,1} \otimes \eta_{i,i}(E_{s,t})) \otimes (E_{1,1} \otimes E_{t,s})^{\op} \\
				& \quad +\sum_{s,t \in [n]} (E_{1,2} \otimes \eta_{i,j}(E_{s,t})) \otimes (E_{2,1} \otimes E_{t,s})^{\op} \\
				& \quad +\sum_{s,t \in [n]} (E_{2,1} \otimes \eta_{j,i}(E_{s,t})) \otimes (E_{1,2} \otimes E_{t,s})^{\op} \\
				& \quad +\sum_{s,t \in [n]} (E_{2,2} \otimes \eta_{j,j}(E_{s,t})) \otimes (E_{2,2} \otimes E_{t,s})^{\op} 
			\end{align*}
			Since the $2 \times 2$ matrix $E_{p,q}$ is always paired with $E_{q,p}^{\op}$, the norm of this is the same if we remove the second copy of $E_{q,p}^{\op}$.  This yields the operator $T_0$ corresponding to the completely positive map
			\[
			\eta_{\{i,j\}} =
			\begin{pmatrix}
				\eta_{i,i} & \eta_{i,j} \\ \eta_{j,i} & \eta_{j,j}
			\end{pmatrix}.
			\]
			In the case that $i \neq j$, $T_0$ is a corner of the operator $T_F$ given by compressing onto the coordinates $i$ and $j$ in $\mb{M}_{|F|}$, and hence $\norm{\Cov(Y)} \leq \norm{T}$.  In the case that $i = j$, $T_0$ is the $2 \times 2$ matrix of all ones tensored with the corner of $T_F$ corresponding to $\eta_{i,i}$, so that $\norm{\Cov(Y)} \leq 2 \norm{T_F}$.  Similarly, $\norm{\Cov(Y')} \leq 2\norm{T_F}$, so overall
			\[
			v(Y) v(Y') = \norm{\Cov(Y)}^{1/2} \norm{\Cov(Y')}^{1/2} \leq 2 \norm{T_F}.\qedhere
			\]
		\end{proof}
		
		Next we have the following bound for the terms corresponding to partitions $\pi$ with crossings, which is a direct adaptation of \cite[Lemma 7.5]{BBvH2023} incorporating the deterministic matrices.
		
		\begin{lem}[{Cf. \cite[Lemma 7.5]{BBvH2023}}] \label{lem: estimate for crossing terms}
			Let $X$ and $X_{\pi,\vec{\imath}}$ be as above and set $F:=\{i_1,\ldots, i_\ell\}$.  Suppose that the partition $\pi$ has a crossing.  Then for $b_1,\ldots, b_\ell \in \mathbb{M}_{n}$, we have
			\[
			\left| \mathbb{E} \circ \tr_n[X_{\pi,\vec{\imath};1} b_1 X_{\pi,\vec{\imath};2} \cdots b_{\ell-1} X_{\pi,\vec{\imath};\ell} b_\ell] \right| \leq 2 \norm{T_F} \norm{\eta_F(1)} \norm{b_1} \cdots  \norm{b_\ell} \max_{i\in F} \mathbb{E} \tr_n(|X_i|^{\ell-4}).
			\]
			%Furthermore, assuming that $(\log n(k))^3 \norm{T} \to 0$, there is a constant
		\end{lem}
		
		\begin{proof}
			First, we observe the following estimate for a monomial, which follows from the non-commutative H{\"o}lder's inequality applied to $L^\infty(\Omega) \otimes \mathbb{M}_{n}$:
			\begin{equation} \label{eq: NC Holder estimate}
				(\mathbb{E} \tr_{n} |X_{i_1} b_1 \cdots X_{i_\ell} b_\ell|^p)^{1/p} \leq \prod_{t=1}^\ell \norm{b_t}  (\mathbb{E} \tr_n |X_{i_t}|^{\ell p})^{1/\ell p} \leq \prod_{t=1}^\ell \norm{b_t} \max_{i\in F} (\mathbb{E} \tr_n |X_{i}|^{\ell p})^{1/p}.
			\end{equation}
			Now suppose that we have a crossing $r_1 < r_2 < s_1 < s_2$ with $\{r_1,s_1\} \in \pi$ and $\{r_2,s_2\} \in \pi$.  By cyclic symmetry of the trace and of partitions, assume without loss of generality that $r_1 = 1$.  Thus, we have
			\[
			X_{\pi,\vec{\imath};1} b_1 X_{\pi,\vec{\imath};2} \cdots b_{\ell-1} X_{\pi,\vec{\imath};\ell} b_\ell = X_{\pi,\vec{\imath},r_1} Y_1 X_{\pi,\vec{\imath},r_2} Y_2 X_{\pi,\vec{\imath};s_1} Y_3 X_{\pi,\vec{\imath};s_2} Y_4,
			\]
			where $Y_1$, \dots, $Y_4$ are the products of the intervening terms.  Let $\ell_j$ be the degree of $Y_j$ in the $X_{\pi,\vec{\imath}}$'s, so that $\ell_1 + \cdots + \ell_4 = \ell - 4$.  Let $p_j = (\ell - 4) / \ell_j$, so that $1/p_1 + \cdots + 1/p_4 = 1$. 
			Note that the $Y_j$'s are independent of $X_{\pi,\vec{\imath},r_1}$, $X_{\pi,\vec{\imath},r_2}$, $X_{\pi,\vec{\imath};s_1}$, $X_{\pi,\vec{\imath};s_2}$, and so by adapting \cite[Lemma 4.5]{BBvH2023} to this setting, we obtain
			\[
			|\mathbb{E} \tr_n[X_{\pi,\vec{\imath},r_1} Y_1 X_{\pi,\vec{\imath},r_2} Y_2 X_{\pi,\vec{\imath};s_1} Y_3 X_{\pi,\vec{\imath};s_2} Y_4] | \leq w(\eta,i_{r_1},i_{s_1},i_{r_2},i_{s_2})^4 \prod_{t=1}^4 (\mathbb{E} \tr_n |Y_t|^{p_t} )^{1/p_t}.
			\]
			Applying Lemma \ref{lem: matrix alignment estimate 2} to the $w$ term and applying \eqref{eq: NC Holder estimate} to the $Y_t$'s, we can bound this by
			\begin{align*}
				2 \norm{T_F} \norm{\eta_F(1)} \norm{b_1} \cdots \norm{b_\ell} \prod_{t=1}^4 \max_{i\in F} (\mathbb{E} \tr_n |X_{i}|^{\ell_t p_t})^{1/p_t} &= 2 \norm{T_F} \norm{\eta_F(1)} \norm{b_1} \cdots \norm{b_\ell} \prod_{t=1}^4 \max_{i\in I} (\mathbb{E} \tr_n |X_{i}|^{\ell-4})^{1/p_t} \\
				&= 2 \norm{T_F} \norm{\eta_F(1)} \norm{b_1} \cdots \norm{b_\ell} \max_{i\in F} (\mathbb{E} \tr_n |X_{i}|^{\ell-4}). \qedhere
			\end{align*}
		\end{proof}
		
		\subsection{Convergence in expectation} \label{subsec: weak convergence in expectation}
		
		With the estimates for the crossing terms in hand, we are now ready to prove convergence in expectation.  To bound the right-hand side from Lemma~\ref{lem: estimate for crossing terms}, we only need some a priori bound on $\max_{i\in F} \mathbb{E} \tr_n(|X_i^{(k)}|^{\ell-4})$.  Hence, we recall the following moment bound from \cite{BBvH2023}, translated into our notation.
		
		\begin{thm}[{\cite[Theorem 2.7]{BBvH2023}}] \label{thm: BBvH moment bound}
			Let $\eta: \mathbb{M}_n \to \mathbb{M}_n$ be completely positive, let $X$ be an $\eta$-Gaussian random matrix, and let $T$ be the Choi matrix of $\eta$.  Let $X_{\operatorname{free}}$ be an $(\mathbb{M}_n,\eta)$-semicircular.  Let $p \in \N$.  Then
			\[
			|(\mathbb{E} \tr_n(X^{2p}))^{1/2p} - \tau(X_{\operatorname{free}}^{2p})^{1/2p}| \leq 2 p^{3/4} (v(X) \sigma(X))^{1/2} = 2 p^{3/4} \norm{\eta(1)}^{1/4} \norm{T}^{1/4}.
			\]
		\end{thm}
		
		\begin{cor} \label{cor: moment bound}
			Let $\eta: \mathbb{M}_n \to \mathbb{M}_n$ be completely positive, let $X$ be an $\eta$-Gaussian random matrix, and let $T$ be the Choi matrix of $\eta$.  Then for $p \in \N$,
			\[
			(\mathbb{E} \tr_n(X^{2p}))^{1/2p} \leq 2 \norm{\eta(1)}^{1/2} + 2 p^{3/4} \norm{\eta(1)}^{1/4} \norm{T}^{1/4}. 
			\]
		\end{cor}
		
		\begin{proof}
			Since $X_{\operatorname{free}}$ is an $(\mathbb{M}_n,\eta)$-semicircular family by Proposition~\ref{prop: identifying X free}, it follows that $\norm{X_{\operatorname{free}}} \leq 2 \norm{\eta(1)}^{1/2}$ (see \cite[Equation (2.8)]{shlyakhtenko1999valued}).  Therefore,
			\[
			\tau(X_{\operatorname{free}}^{2p})^{1/2p} \leq \norm{X_{\operatorname{free}}} \leq 2 \norm{\eta(1)}^{1/2}.
			\]
			The asserted bound thus follows from Theorem \ref{thm: BBvH moment bound} and the triangle inequality.
		\end{proof}
		
		As an intermediate step to showing convergence in expectation for covariance polynomials, we first show show convergence for \emph{non-commutative polynomials} but in \emph{random} approximations for $B$.

		\begin{lem} \label{lem: weak convergence part 1}
			Let $X$ be a $(B,\eta)$-semicircular family for a tracial von Neumann algebra $(B,\tau)$ generated by $b=(b_\omega)_{\omega\in \Omega}$ and a $B$-valued covariance matrix $\eta=(\eta_{i,j})_{i,j\in I}$. For a sequence of integers $(n(k))_{k\in \N}$, let $\eta^{(k)}=(\eta_{i,j}^{(k)})_{i,j\in I}$ be a $\tr_{n(k)}$-symmetric $\mb{M}_{n(k)}$-valued covariance matrix satisfyng $\sup_k \|\eta_{i,j}^{(k)}\| <\infty$ for each $i,j\in I$. For each finite $F\subseteq I$, let $T^{(k)}_F$ be the Choi matrix associated to $\eta_F=(\eta_{i,j})_{i,j\in F}$, and assume
			\[
			\lim_{k \to \infty} \norm{T^{(k)}}_F = 0.
			\]
			Let $B^{(k)}=(B_\omega^{(k)})_{\omega \in \Omega}$ be random matrices in $\mathbb{M}_{n(k)}$ such that $\sup_k \mathbb{E} \norm{B_\omega^{(k)}}^m < \infty$ for each $\omega\in \Omega$ and $m \in \N$ and assume that
			\[
			\lim_{k\to\infty} \mathbb{E} \tr_{n(k)}[f(\eta^{(k)},B^{(k)})] = \tau(f(\eta,b))
			\]
			for all covariance polynomials $f\in \C_I\<\Omega\>$. If $X^{(k)}$ is an $\eta^{(k)}$-Gaussian family independent of $B^{(k)}$ and $X$ is a $(B,\eta)$-semicircular family, then for every non-commutative polynomial $p\in \C\<\Omega\sqcup I\>$, we have
			\[
			\lim_{k \to \infty} \mathbb{E} \tr_{n(k)}[p(B^{(k)}, X^{(k)})] = \tau(p(b,X)).
			\]
		\end{lem}
		
		\begin{proof}
			It suffices to consider a monomial of the form
			\[
			g_0(B^{(k)}) X_{i_1}^{(k)} g_1(B^{(k)}) \cdots X_{i_\ell}^{(k)} g_\ell(B^{(k)}),
			\]
			where the $g_t$'s are non-commutative polynomials.   By cyclic symmetry of the trace, we may combine $g_0(B^{(k)})$ with $g_\ell(B^{(k)})$. Set $\vec{\imath}:=(i_1,\ldots, i_\ell)$ and $F:=\{i_1,\ldots, i_\ell\}$.  By conditioning on the value of $B^{(k)}$ and applying Lemma~\ref{lem: matrix Wick formula},
			\[
			\mathbb{E} \tr_{n(k)}[X_{i_1}^{(k)} g_1(B^{(k)}) \cdots  X_{i_\ell}^{(k)} g_\ell(B^{(k)})] = \sum_{\pi \in \mathcal{P}_2(\ell)} \mathbb{E} \tr_n[X_{\pi,\vec{\imath};1}^{(k)} g_1(B^{(k)}) \cdots X_{\pi,\vec{\imath};\ell}^{(k)} g_\ell(B^{(k)})].
			\]
			By conditioning on the value of $B^{(k)}$ and applying Lemma \ref{lem: estimate for crossing terms}, if $\pi$ has a crossing, then
			\begin{multline*}
				|\mathbb{E} \tr_n[X_{\pi,\vec{\imath};1}^{(k)} g_1(B^{(k)}) \cdots X_{\pi,\vec{\imath};\ell}^{(k)} g_\ell(B^{(k)})]| \\
				\leq 2 \norm{T_F^{(k)}} \norm{\eta_F^{(k)}(1)} \mathbb{E}\left[ \norm{g_1(B^{(k)})} \cdots \norm{g_\ell(B^{(k)})} \right] \max_{i\in F} \mathbb{E} \tr_{n(k)}(|X_i^{(k)}|^{\ell-4}).
			\end{multline*}
			By Corollary \ref{cor: moment bound} and H{\"o}lder's inequality, $\max_i \mathbb{E} \tr_n(|X_i^{(k)}|^{\ell-4})$ is bounded as $k \to \infty$. Our norm assumptions then imply the overall contribution from each partition $\pi$ with a crossing vanishes.  Hence, we are left with the contributions from non-crossing partitions, which in turn can be evaluated using Lemma \ref{lem: evaluation of NC terms} to give
			\[
			\lim_{k \to \infty} \left| \mathbb{E} \tr_{n(k)}[X_{i_1}^{(k)} g_1(B^{(k)}) \cdots  X_{i_\ell}^{(k)} g_\ell(B^{(k)})] - \sum_{\pi \in \mathcal{NC}_2(\ell)} \mathbb{E} \tr_{n(k)}\left[\eta^{(k)}_{\pi,\vec{\imath}}[g_1(B^{(k)}),\dots,g_{\ell-1}(B^{(k)})] g_\ell(B^{(k)})\right] \right| = 0.
			\]
			Our assumption of convergence for covariance polynomials then gives
			\begin{align*}
				\lim_{k \to \infty} \tr_{n(k)}\left[\eta^{(k)}_{\pi,\vec{\imath}}[g_1(B^{(k)}),\dots,g_{\ell-1}(B^{(k)})] g_\ell(B^{(k)})\right] &= \tau(\eta_{\pi,\vec{\imath}}[g_1(b),\dots,g_{\ell-1}(b)] g_\ell(b))\\
				&= \tau(X_{i_1} g_1(b) \cdots X_{i_\ell} g_\ell(b)),
			\end{align*}
			where the last equality comes from Lemma~\ref{lem: semicircular partition formula}.  Hence, the desired conclusion follows from the triangle inequality.
		\end{proof}

		\begin{lem} \label{lem: weak convergence part 2}
			Consider the same setup and hypotheses as Lemma~\ref{lem: weak convergence part 1}.  Then for every covariance polynomial $f\in \C_I\<\Omega\sqcup I\>$, we have
			\[
			\lim_{k \to \infty} \mathbb{E} \tr_{n(k)}[f(\eta^{(k)}, B^{(k)}, X^{(k)})] = \tau[f(\eta,b,X)].
			\]
		\end{lem}
		
		\begin{proof}
			It suffices to consider the case where $f$ is a covariance monomial.  In particular, suppose
			\[
			f(\eta^{(k)}, B^{(k)}, X^{(k)}) = g_0(B^{(k)},X^{(k)}) \eta^{(k)}_{\pi,\vec{\imath}}[g_1(B^{(k)},X^{(k)}), \dots, g_{\ell-1}(B^{(k)},X^{(k)})] g_\ell(B^{(k)},X^{(k)}).
			\]
			By cyclic symmetry of the trace, absorb $g_0$ into $g_\ell$. Set $\vec{\imath}:=(i_1,\ldots, i_\ell)$ and let $X_{\pi,\vec{\imath}}^{(k)}$ be defined as at the beginning of this section, and assume that $X_{\pi,\vec{\imath}}^{(k)}$ is independent of $(B^{(k)},X^{(k)})$.  By conditioning and Lemma~\ref{lem: evaluation of NC terms}, we get
			\begin{align*}
				\mathbb{E} \tr_{n(k)} &\left[\eta^{(k)}_{\pi,\vec{\imath}}[g_1(B^{(k)},X^{(k)}), \dots, g_{\ell-1}(B^{(k)},X^{(k)})] g_\ell(B^{(k)},X^{(k)})\right] \\
				&= \mathbb{E} \tr_{n(k)} [X_{\pi,\vec{\imath};1}^{(k)} g_1(B^{(k)},X^{(k)}) \cdots  X_{\pi,\vec{\imath};\ell}^{(k)} g_\ell(B^{(k)},X^{(k)})].
			\end{align*}
			Note that $(X^{(k)},X_{\pi,\vec{\imath}}^{(k)})$ can be viewed as a single Gaussian family indexed by $I \sqcup [\ell]$.  Analogously, let $(X_{\pi,\vec{\imath},t})_{t=1}^\ell$ be given by copies of the appropriate $X_i$'s which are freely independent over $B$ for the distinct blocks of $\pi$, and jointly freely independent over $B$ from $X$.  By \cite[Example 2.6]{shlyakhtenko1999valued}, this is also a $(B,\eta)$-semicircular family indexed by $I \sqcup [\ell]$.  By \cite[Proposition 3.8]{shlyakhtenko2000cpentropy}, $X_{\pi,\vec{\imath}}$ is further an $( W^*(B,X), \eta)$-valued semicircular family.  So using Lemma~\ref{lem: semicircular partition formula} we have
			\[
			\tau(X_{\pi,\vec{\imath};1} g_1(b,X) \cdots  X_{\pi,\vec{\imath};\ell} g_\ell(b,X)) = \tau( \eta_{\pi,\vec{\imath}}[g_1(b,X),\dots,g_{\ell-1}(b,X)]g_\ell(b,X)).
			\]
			By Lemma \ref{lem: weak convergence part 1}, we have
			\[
			\lim_{k \to \infty} \mathbb{E} \tr_{n(k)} [X_{\pi,\vec{\imath};1}^{(k)} g_1(B^{(k)},X^{(k)}) \cdots  X_{\pi,\vec{\imath};\ell}^{(k)} g_\ell(B^{(k)},X^{(k)})] = \tau(X_{\pi,\vec{\imath};1} g_1(b,X) \cdots  X_{\pi,\vec{\imath};\ell} g_\ell(b,X)),
			\]
			so the proof is complete.
		\end{proof}
		
		\subsection{Almost sure convergence} \label{subsec: weak almost sure convergence}
		
		Next, we upgrade from convergence in expectation to almost sure convergence using standard concentration of measure arguments together with non-commutative H{\"o}lder inequalities.
		
		\begin{lem} \label{lem: Herbst}
			For a finite set $F$, let $\eta=(\eta_{i,j})_{i,j\in F}$ be a $\tr_n$-symmetric $\mb{M}_n$-valued covariance matrix and let $T$ be the associated Choi matrix. Then an $\eta$-Gaussian family $X$ satisfies Herbst's concentration inequality: for all $\delta>0$
			%\[
			%P( \re f(\mathbf{X}^{(k)}) - \mathbb{E} \re f(\mathbf{X}^{(k)}) \geq \delta) \leq e^{-n(k) \delta^2 / 2 \norm{T^{(k)}} \norm{f}_{\operatorname{Lip}}^2},
			%\]
			%and
			\[
			\mathbb{P}\left( \left|f(X) - \mathbb{E}[f(X)]\right| \geq \delta \right) \leq 4 \exp\left(\frac{ -n \delta^2}{2 \norm{T} \norm{f}_{\operatorname{Lip}}^2}\right),
			\]
			where $\norm{f}_{\operatorname{Lip}}$ is the Lipschitz norm of $f: L^2(\mb{M}_n, \tr_n)^F\to \C$.
		\end{lem}
		
		\begin{proof} 
			We can write $X = T^{1/2} Y$, where $Y$ is a Gaussian random variable in $\mb{M}_{n}^I$ with covariance matrix given by the identity in $B(\ell^2(I)) \otimes \mb{M}_{n} \otimes \mb{M}_{n}^{\op}$.  Recall that ${Y}$ satisfies the concentration inequality
			\begin{equation} \label{eq: standard Gaussian concentration}
				\mathbb{P}\left( |g(Y) - \mathbb{E}[g(Y)]| \geq \delta \right) \leq 4 \exp\left(\frac{-\delta^2}{ 2 \norm{g}_{\operatorname{Lip},\Tr_{n}}^2}\right),
			\end{equation}
			where $g:  L^2(\mb{M}_n, \Tr_n)^I \to \C$ is Lipschitz with norm $\norm{g}_{\Lip,\Tr_{n}}$ (see, for example, \cite[Corollary 4.1]{BL2000}, \cite[\S 2.3.2]{AGZ2009}).  Now suppose that $f: L^2(\mb{M}_n,\tr_n)^I \to \C$ is Lipschitz with $\norm{f}_{\Lip}$ defined with respect to the Hilbert-Schmidt norm for the \emph{normalized} trace.  Then $f(X) = f (T^{1/2} Y)$ and
			\[
			\norm{f \circ T^{1/2}}_{\Lip,\Tr_{n}} \leq \norm{f}_{\Lip,\Tr_{n}} \norm{T}^{1/2} = n^{-1/2} \norm{T}^{1/2} \norm{f}_{\Lip,\tr_{n}}.
			\]
			Hence, taking $g = f \circ T^{1/2}$ in \eqref{eq: standard Gaussian concentration}, we get the asserted estimate.
		\end{proof}
		
		Since covariance polynomials under consideration are not globally Lipschitz functions, we need to make an additional cut-off argument, as is also quite standard in random matrix theory.  Although one often does the cut-off by projecting onto the unit ball in operator norm, we will project instead onto the unit ball in some non-commutative $L^p$ norm, since in Theorem~\ref{introthm: weak convergence} we do not assume bounds on $\norm{T_F^{(k)}}$ strong enough to control the operator norms of the corresponding $\eta^{(k)}$-Gaussian family.  Hence, we first show the following auxiliary lemma.
		
		\begin{lem} \label{lem: as convergence of p norm}
			Consider the same setup as in Theorem~\ref{introthm: weak convergence}, so in particular ${X}^{(k)}$ is an $\eta^{(k)}$-Gaussian family and ${X}$ is a $(B,\eta)$-valued semicircular. For $p \in \N$ and $i\in I$, we have
			\[
			\lim_{k \to \infty} \tr_n( (X_i^{(k)})^{2p} ) = \tau((X_i)^{2p}) \text{ almost surely.}
			\]
		\end{lem}
		
		\begin{proof}
			Let $f(a) := \norm{a}_{2p} = \tr_n( (a)^{2p} )^{1/2p}$.  Recall that $\norm{a}_{2p} \leq n(k)^{\frac12-\frac{1}{2p}} \norm{a}_2$, so that $f$ is $n(k)^{\frac12-\frac{1}{2p}}$-Lipschitz as a map on $L^2(\mb{M}_{n(k)}, \tr_{n(k)})$. Letting $T_i^{(k)}$ be the Choi matrix associated to $\eta_{i,i}$, Lemma~\ref{lem: Herbst} then implies for $\delta > 0$ that
			\[
			\mathbb{P}\left( |f(X_i^{(k)}) - \mathbb{E} f(X_i^{(k)})| \geq \delta \right) \leq 4 e^{-n(k)^{1/p} \delta^2 / 2\norm{T_i^{(k)}}}.
			\]
			We first take $\delta = \norm{T_i^{(k)}}^{1/2}$.  Since $e^{-n(k)^{1/p}/2}$ is summable, the Borel--Cantelli lemma implies that almost surely for sufficiently large $k$, we have
			\[
			|f(X_i^{(k)}) - \mathbb{E} f(X_i^{(k)})| \leq \norm{T_i^{(k)}}^{1/2},
			\]
			and hence
			\[
			\lim_{k \to \infty} [f(X_i^{(k)}) - \mathbb{E} f(X_i^{(k)})] = 0 \quad\text{ almost surely.}
			\]
			Hence, it suffices to show that
			\[
			\lim_{k \to \infty} \mathbb{E} \norm{X_i^{(k)}}_{2p} = \norm{X_i}_{2p}.
			\]
			For the upper bound, note by H{\"o}lder's inequality,
			\[
			\limsup_{k \to \infty} (\mathbb{E} \norm{X_i^{(k)}}_{2p})^{2p} \leq \lim_{k \to \infty} \mathbb{E} \left( \norm{X_i^{(k)}}_{2p}^{2p} \right) = \tau[(X_i)^{2p}].
			\]
			For the lower bound, let $M_k = \mathbb{E}\norm{X_i^{(k)}}_{2p}$, and note that
			\begin{align*}
				\mathbb{E} \norm{X_i^{(k)}}_{2p}^{2p} &= \int_0^\infty 2p t^{2p-1} \mathbb{P}( \norm{X_i^{(k)}}_{2p} \geq t)\,dt \\
				&= \int_0^{M_k} 2p t^{2p-1} \mathbb{P}( \norm{X_i^{(k)}}_{2p} \geq t)\,dt +  \int_0^\infty 2p (M_k + \delta)^{2p-1} \mathbb{P}( \norm{X_i^{(k)}}_{2p} \geq M_k + \delta)\,d\delta \\
				&\leq M_k^{2p} + \int_0^\infty 2^{2p} p (M_k^{2p-1} + \delta^{2p-1}) e^{-n(k)^{1/p} \delta^2 / 2\norm{T^{(k)}}}\,d\delta \\
				&= M_k^{2p} + 2^{2p} M_k^{2p-1} \sqrt{2\pi} \norm{T^{(k)}}^{1/2} n(k)^{-1/2p} + 2^{2p} [(2p-2)(2p-4) \cdots 2] \norm{T^{(k)}}^{p-1} n(k)^{-(p-1)/p}.
			\end{align*}
			The terms other than $M_k^{2p}$ vanish as $k \to \infty$. 
			Hence,
			\[
			\norm{X_i}^{2p}_{2p} = \lim_{k \to \infty} (\mathbb{E} \norm{X_i^{(k)}}_{2p}^{2p}) \leq \liminf_{k \to \infty} M_k^{2p} = \liminf_{k \to \infty} \mathbb{E} \norm{X_i^{(k)}}^{2p}_{2p},
			\]
			which completes the proof.
		\end{proof}
		
		Next, we show that the trace of any covariance polynomial is $\norm{\,\cdot\,}_2$-Lipschitz on the product of non-commutative $L^p$-balls for sufficiently large $p$.  As in the case of ordinary trace polynomials, this will follow from the non-commutative H{\"o}lder's inequality.  We first recall the $L^p$ boundedness of $\eta$.
		
		\begin{lem} \label{lem: Lp boundedness}
			Let $(B,\tau)$ be a tracial von Neumann algebra equipped with a $\tau$-symmetric operator-valued covariance matrix $(\eta_{i,j})_{i,j\in I}$.  Then for all $x \in B$, $i,j\in I$, and $p \in [1,\infty]$ we have
			\[
			\norm{\eta_{i,j}(x)}_p \leq \norm{\eta_F(1)} \norm{x}_p
			\]
			for any finite subset $F\subseteq I$ containing $i,j$.
		\end{lem}
		
		\begin{proof}
			By non-commutative $L^p$ interpolation \cite{PX2003}, it suffices to prove the bound for $p = 1$ and $p = \infty$.  For $p = \infty$, note that $\norm{\eta_{i,j}(x)} \leq \norm{\eta_F(x)} \leq \norm{\eta_F(1)} \norm{x}$ by complete positivity.  For $p = 1$, let $x, y \in B$.  Using $\tau$-symmetry,
			\[
			|\tau(\eta_{i,j}(x) y)| = |\tau(x \eta_{j,i}(y))| \leq \norm{x}_1 \norm{\eta_{i,j}(y)} \leq \norm{\eta_F(1)} \norm{x}_1 \norm{y}_\infty.
			\]
			Thus,
			\[
			\norm{\eta_{i,j}(x)}_1 = \sup_{\norm{y}_\infty \leq 1} |\tau(\eta_{i,j}(x) y)| \leq \norm{\eta_F(1)} \norm{x}_1,
			\]
			which completes the $L^1$ case and hence the proof.
		\end{proof}
		
		\begin{lem} \label{lem: covariance polynomial Lipschitz}
			For a covariance monomial $f\in \C_I\<\Omega\sqcup I\>$, let $p \geq 2(\deg(f) - 1)$ and let $F\subseteq I$ and $G\subseteq \Omega$ be the finite subsets on which $f$ depends.  For $R > 0$ there is some constant $C$ depending only on $f$ and $R$, such that for any tuples $b=(b_\omega)_{\omega\in \Omega}$, $ x = (x_i)_{i\in I}$, and $x'=(x_i')_{i\in I}$ in a tracial von Neumann algebra $(M,\tau)$ equipped with an operator-valued covariance matrix $(\eta_{i,j})_{i,j\in I}$, whenever  
			\[
			\norm{\eta_F(1)} \leq R, \quad \max_{\omega\in G}\norm{b_\omega} \leq R, \quad \max_{i\in F}\norm{x_i}_p \leq R, \quad \text{ and } \quad \max_{i\in F}\norm{x_i'}_p \leq R,
			\]
			we have
			\[
			|\tau(f(\eta,b,x)) - \tau(f(\eta,b,x'))| \leq C \norm{{x} - {x}'}_2.
			\]
		\end{lem}
		
		\begin{proof}
			First, consider the case where
			\[
			f(b,x) = b_0 x_{i_1} b_1 \cdots x_{i_\ell} b_\ell.
			\]
			Suppose $p \geq 2(\ell - 1)$ and let $q$ be given by $1/q = 1/2 + (\ell-1) / p$, or $q = 2p / (p + 2(\ell-1))$.  Then
			\begin{align*}
				\norm{f(b,x') - f(b,x)}_q &\leq \sum_{t=1}^\ell \norm{ b_0 x_{i_1} b_1  \cdots x_{i_{t-1}} b_{t-1} (x_{i_t}' - x_{i_t})  b_t x_{i_{t+1}}' b_{t+1} \cdots x_{i_\ell}' b_\ell}_q \\
				&\leq \sum_{t=1}^\ell \norm{b_0} \norm{x_{i_1}}_p \norm{b_1}  \cdots \norm{x_{i_{t-1}}}_p \norm{b_{t-1}} \norm{x_{i_t}' - x_{i_t}}_2  \norm{b_t} \norm{x_{i_{t+1}}'}_p \norm{b_{t+1}} \cdots \norm{x_{i_\ell}'}_p \norm{b_\ell} \\
				&\leq \ell R^{2\ell-1} \norm{{x}' - {x}}_2.
			\end{align*}
			Now suppose that $f$ is a covariance monomial, so that
			\[
			f =  f_0\eta_{\pi,\vec{\imath}}[f_1,\dots,f_{2k-1}] f_{2k}
			\]
			where $f_j$ has degree $\ell_j$.  Let $1 / q_j = 1/2 + (\ell_j-1)/p$ and let $1 / r_j = \ell_j / p$.  Let
			\[
			1/q = 1/2 + (\ell_0+\ell_1 + \cdots + \ell_{2k} - 1)/p = 1/q_j + \sum_{i \neq j} 1/r_i.
			\]
			Note that $q_j$, $r_j$, and $q \geq 1$.  By switching out the terms $f_j(b,x)$ for $f_j(b,x')$ one by one as above, and by iteratively applying the non-commutative H{\"o}lder's inequality and Lemma~\ref{lem: Lp boundedness}, we have
			\begin{multline*}
				\norm{f(\eta,b,x') - f(\eta,b,x)}_q \\
				\leq \sum_{j=0}^{2k} C_j \norm{f_0(b,x)}_{r_0} \cdots \norm{f_{j-1}(b,x)}_{r_{j-1}} \norm{f_j(b,x') - f_j(b,x)}_{q_j} \norm{f_{j+1}(b,x')}_{r_{j+1}} \cdots \norm{f_{2k}(b,x')}_{r_{2k}} 
			\end{multline*}
			Then by applying H{\"o}lder's inequality to each of the monomials, $\norm{f_j(b,x)}_{r_j}$ is bounded by a constant.  Also, $\norm{f_j(b,x') - f_j(b,x)}_{q_j}$ is bounded by a constant times $\norm{{x} - {x}'}_2$.  Finally, note
			\[
			|\tau(f(\eta,b,x')) - \tau(f(\eta,b,x))| \leq \norm{f(\eta,b,x') - f(\eta,b,x)}_q.  \qedhere
			\]
		\end{proof}
		
		Finally, we complete the proof of almost sure convergence in Theorem~\ref{introthm: weak convergence}.
		
		\begin{thm}[{Theorem~\ref{introthm: weak convergence}}]\label{thm: weak convergence with proof}
			Let $(M,\tau)$ be a tracial von Neumann algebra generated by a subalgebra $B\subset M$ and a $(B,\eta)$-semicircular family $x=(x_i)_{i\in I}$, where $\eta=(\eta_{i,j})_{i,j\in I}$ is a $\tau$-symmetric $B$-valued covariance matrix, and let $b=(b_\omega)_{\omega\in \Omega}$ be a generating tuple for $B$. For a sequence of integers $(n(k))_{k\in \N}\subset \N$, let $\eta^{(k)}=(\eta_{i,j}^{(k)})_{i,j\in I}$ be a $\tr_{n(k)}$-symmetric $\mb{M}_{n(k)}$-valued covariance matrix, let $X^{(k)}=(X_i^{(k)})_{i\in I}$ be an $\eta^{(k)}$-Gaussian family, and let $b^{(k)}=(b^{(k)}_\omega)_{\omega\in \Omega}\subset \mb{M}_{n(k)}$. Assume that:
			\begin{itemize}
				\item $\displaystyle\sup_k \|\eta_{i,j}^{(k)}\|<\infty$ for each $i,j\in I$;
				\item $\displaystyle\sup_k \|b_\omega^{(k)}\| <\infty$ for each $\omega\in \Omega$;
				\item $\displaystyle\lim_{k\to\infty} \|T_F^{(k)}\|=0$ for each finite $F\subseteq I$, where $T_F^{(k)}$ is the Choi matrix associated to $\eta_F=(\eta_{i,j})_{i,j\in F}$;
				\item $(\eta^{(k)}, b^{(k)})$ converges weakly in covariance law to $(\eta,b)$.
			\end{itemize}
			Then $(\eta^{(k)}, b^{(k)}, X^{(k)})$ converges weakly in covariance law to $(\eta,b,x)$, both almost surely and in expectation.
		\end{thm}
		\begin{proof}
			Convergence in expectation follows by applying Lemma~\ref{lem: weak convergence part 2} to the deterministic matrices $b^{(k)}$, so it remains to prove almost sure convergence.  By linearity, it suffices to show for every covariance monomial $f\in \C_I\<\Omega\sqcup I\>$ that
			\[
			\lim_{k \to \infty} \tr_{n(k)}[f(\eta^{(k)},{b}^{(k)},{X}^{(k)})] = \tau[f(\eta,{b},x)].
			\]
			Let $F\subseteq I$ and $G\subseteq \Omega$ be the finite subsets that $f$ depends on.  Choose $p \in \N$ such that $2p \geq 2(\deg(f) - 1)$, and choose $R$ sufficiently large that
			\[
			\max_{i\in F} \norm{x_i}_{2p} < R, \qquad \sup_k \norm{\eta_F^{(k)}(1)} \leq R, \quad \text{ and } \quad \max_{\omega\in G} \sup_k  \norm{b_\omega^{(k)}} \leq R.
			\]
			For $a \in \mathbb{M}_{n(k)}$, let
			\[
			\Pi(a) = \operatorname{argmin}_{a' \in \mathbb{M}_{n(k)}: \norm{a'}_{2p} \leq R} \norm{a - a'}_2.
			\]
			This is well-defined because $\{a': \norm{a'}_{2p} \leq R\}$ is a closed convex set in the Hilbert space $L^2(\mathbb{M}_{n(k)}, \tr_{n(k)})$ and thus there is a unique closest point to $a$.  Moreover, as a general property of projections onto closed convex sets, we have $\norm{\Pi(a) - \Pi(a')}_2 \leq \norm{a - a'}_2$. 
			We will use the same notation $\Pi$ for the coordinate-wise application $\Pi$ to tuples of matrices such as ${X}_i^{(k)}$, so that $\Pi({X}_i^{(k)}) = (\Pi({X}_i^{(k)}))_{i \in I}$.  Let $C$ be a Lipschitz constant for $f$ given by Lemma~\ref{lem: covariance polynomial Lipschitz}.  Then because $\norm{\Pi(a)}_{2p} \leq R$ by construction and $\norm{\Pi(a) - \Pi(a')}_2 \leq \norm{a - a'}_2$, we have that for ${Y}, {Y}' \in (\mathbb{M}_{n(k)})_{\sa}^I$,
			\begin{equation} \label{eq: convergence Lipschitz estimate}
				|\tr_{n(k)}[f(\eta^{(k)},\Pi({Y}),b^{(k)})] - \tr_{n(k)}[f(\eta^{(k)},\Pi({Y}'),b^{(k)})]| \leq C \norm{{Y} - {Y}'}_2.
			\end{equation}
			Hence, the composition of $f$ with $\Pi$ in the middle variables is a globally $\norm{\,\cdot\,}_2$-Lipschitz function to which can we apply the Herbst concentration inequality.
			
			Now we estimate
			\begin{align}
				\tr_{n(k)}[f(\eta^{(k)}, {X}^{(k)},b^{(k)})]& - \tau[f(\eta,x,{b})] \nonumber\\
				=& \tr_{n(k)}[f(\eta^{(k)}, {X}^{(k)},b^{(k)})] - \tr_{n(k)}[f(\eta^{(k)}, \Pi({X}^{(k)}),b^{(k)})] 
				\label{eq: convergence term 1} \\
				&+ \tr_{n(k)}[f(\eta^{(k)}, \Pi({X}^{(k)}),b^{(k)})] - \mathbb{E} \tr_{n(k)}[f(\eta^{(k)}, \Pi({X}^{(k)}),b^{(k)})]
				\label{eq: convergence term 2} \\
				&+ \mathbb{E} \tr_{n(k)}[f(\eta^{(k)}, \Pi({X}^{(k)}),b^{(k)})] - \mathbb{E} \tr_{n(k)}[f(\eta^{(k)},{X}^{(k)},b^{(k)})] \label{eq: convergence term 3} \\
				&+ \mathbb{E} \tr_{n(k)}[f(\eta^{(k)},{X}^{(k)},b^{(k)})] - \tau[f(\eta,{X},{b})].
				\label{eq: convergence term 4}
			\end{align}
			We show that each of the four terms on the right-hand side vanishes almost surely as $k \to \infty$.
			
			By Lemma \ref{lem: as convergence of p norm} and our choice of $R$, we have almost surely for sufficiently large $k$, $\norm{{X}_i^{(k)}}_p < R$ and hence $\Pi({X}_i^{(k)}) = {X}_i^{(k)}$.  Thus, \eqref{eq: convergence term 1} vanishes for sufficiently large $k$.
			
			In light of \eqref{eq: convergence Lipschitz estimate}, \eqref{eq: convergence term 2} is the difference of a Lipschitz function and its expectation.  By the Herbst inequality (Lemma \ref{lem: Herbst}), we have
			\[
			\mathbb{P}\left(|\tr_{n(k)}[f(\eta^{(k)}, \Pi({X}^{(k)}),b^{(k)})] - \mathbb{E} \tr_{n(k)}[f(\eta^{(k)}, \Pi({X}^{(k)}),b^{(k)})]| \geq C\norm{T^{(k)}}^{1/2} \right) \leq 4 e^{-n(k)/2}.
			\]
			Since this is summable, the Borel--Cantelli lemma implies that almost surely for sufficiently large $k$,
			\[
			|\tr_{n(k)}[f(\eta^{(k)}, \Pi({X}^{(k)}),b^{(k)})] - \mathbb{E} \tr_{n(k)}[f(\eta^{(k)}, \Pi({X}^{(k)}),b^{(k)})]| \leq C\norm{T^{(k)}}^{1/2},
			\]
			and hence \eqref{eq: convergence term 2} vanishes as $k \to \infty$ almost surely.
			
			The third term \eqref{eq: convergence term 3} is minus the expectation of the first term \eqref{eq: convergence term 1}, so we want to apply the generalized dominated convergence theorem (see e.g.\ \cite[\S 11.4, Proposition 18]{Royden}) to \eqref{eq: convergence term 1}, which we already know vanishes almost surely as $k \to \infty$.  Let $\ell = \deg(f)$.  By H{\"o}lder-type estimates as in the proof of Lemma \ref{lem: covariance polynomial Lipschitz}, we can see that for some constant $C'$,
			\begin{align*}
				|\tr_{n(k)}[f(\eta^{(k)}, {X}^{(k)},b^{(k)})]| &\leq C' \left( \max_{i \in F} \norm{X_i^{(k)}}_{2p} \right)^\ell \\
				&\leq C' \max_{i \in F} \max(1, \norm{X_i^{(k)}}_{2p}) \\
				&\leq C' \left( 1 + \sum_{i \in F} \tr_{n(k)}[(X_i^{(k)})^{2p}] \right).
			\end{align*}
			Meanwhile, $\tr_{n(k)}[f(\eta^{(k)}, \Pi({X}^{(k)}),b^{(k)})]$ is bounded, and so after modifying $C'$ if necessary,
			\[
			\left| \tr_{n(k)}[f(\eta^{(k)}, {X}^{(k)},b^{(k)})] - \tr_{n(k)}[f(\eta^{(k)}, \Pi({X}^{(k)}),b^{(k)})] 
			\right| \leq C' \left( 1 + \sum_{i \in F} \tr_{n(k)}[(X_i^{(k)})^{2p}] \right).
			\]
			The function on the right-hand side converges almost surely to $C'(1 + \sum_{i \in F} \tau(X_i^{2p})$ by Lemma \ref{lem: as convergence of p norm}, and the expectations also converge to the same limit by Lemma \ref{lem: weak convergence part 2}.  Hence, the hypotheses of the generalized dominated convergence theorem are satisfied, and so
			\[
			\lim_{k \to \infty} \mathbb{E} \left| \tr_{n(k)}[f(\eta^{(k)}, {X}^{(k)},b^{(k)})] - \tr_{n(k)}[f(\eta^{(k)}, \Pi({X}^{(k)}),b^{(k)})] 
			\right| = 0.
			\]
			Hence, \eqref{eq: convergence term 3} vanishes as $k \to \infty$.
			
			Finally, \eqref{eq: convergence term 4} is deterministic and vanishes as $k \to \infty$ by Lemma \ref{lem: weak convergence part 2}.
		\end{proof}
		
		\section{Proof of strong convergence} \label{sec: strong convergence}
		
		In this section, we prove the general strong convergence result, formally stated in Theorem \ref{thm: strong convergence general}. 
		
		The proof proceeds in several steps:
		\begin{itemize}
			\item We use the results of \cite{BBvH2023} and \cite{GKE2026Toeplitz} to show convergence of the norms of ordinary non-commutative polynomials in ${X}^{(k)}$ and ${b}^{(k)}$.
			\item We show that a covariance polynomial in ${X}$ and ${b}$ can be approximated in operator norm by a non-commutative polynomial in ${X}$, ${b}$, and $B$-freely independent copies of ${X}$, using an averaging trick based on the free LLN.
			\item We then conclude by applying the first step to these polynomials in independent copies of the matrix models and respectively the semicirculars.
		\end{itemize}

		\subsection{Convergence of operator norms for polynomials} \label{subsec: norms of polynomials}
		
		As sketched above, we first deduce convergence of the operator norms of non-commutative polynomials in ${X}^{(k)}$ and ${b}^{(k)}$ from the results of \cite{BBvH2023}.  We remark that, as usual, the asymptotic lower bounds on $\norm{p(b^{(k)},X^{(k)})}$ are automatic from weak convergence.  The following lemma is well known.
		
		\begin{lem} \label{lem: strong convergence general lower bound}
			Let $X^{(k)} = (X_i^{(k)})_{i \in I}$ for $k \in \N$ and $(X_i)_{i \in I}$ be self-adjoint operators in tracial von Neumann algebras $(M^{(k)},\tau^{(k)})$ and $(M,\tau)$ respectively.  Suppose that the law of $X^{(k)}$ converges weakly to the law of $X$.  Then for every non-commutative polynomial, we have $\liminf_{k \to \infty} \norm{p(X^{(k)})} \geq \norm{p(X)}$.
		\end{lem}
		
		\begin{proof}
			Note that
			\begin{align*}
				\norm{p(X)} &= \lim_{m \to \infty} \tau([p(X)^*p(X)]^m)^{1/2m} \\
				&= \sup_{m \in \N} \tau([p(X)^*p(X)]^m)^{1/2m} \\
				&= \sup_{m \in \N} \lim_{k \to \infty} \tau^{(k)}([p(X^{(k)})^*p(X^{(k)})]^m)^{1/2m} \\
				&\leq \sup_{m \in \N} \liminf_{k \to \infty} \norm{[p(X^{(k)})^*p(X^{(k)})]^m}^{1/2m} \\
				&= \liminf_{k \to \infty} \norm{p(X^{(k)})}.  \qedhere
			\end{align*}
		\end{proof}
		
		Hence, our results on weak convergence (Theorem \ref{introthm: weak convergence}) will allow us to deduce the lower bounds on the operator norms of the random matrices automatically.  We therefore focus our attention on the upper bounds, and show the following.

		\begin{prop} \label{prop: strong convergence plain}
			Let ${X}^{(k)} = (X_i^{(k)})_{i \in I}$ be Gaussian random matrices with operator-valued covariance $\eta^{(k)}$ and associated Choi matrix $T^{(k)}$.  Assume that for each finite $F\subset I$
			\[
			\sup_k \norm{\eta_{F}^{(k)}} < \infty, \qquad \lim_{k \to \infty} (\log n(k))^3 \norm{T_F^{(k)}} = 0.
			\]
			Let $b^{(k)} = (b_\omega^{(k)})_{\omega \in \Omega}$ be a tuple of deterministic matrices with $\sup_k \norm{b_\omega^{(k)}} < \infty$. Let $X$ be a $(B,\eta)$-semicircular family, and let $b = (b_\omega)_{\omega \in \Omega}$.   Assume that $(\eta^{(k)},b^{(k)})$ converges strongly in covariance law to $(b,\eta)$.  Then for every non-commutative polynomial $p$, we have almost surely
			\[
			\limsup_{k \to \infty} \norm{p(b^{(k)},X^{(k)})} \leq \norm{p(b,X)}.
			\]
		\end{prop}
		
		The first ingredient we need is the following result of Gao and Kunnawalkam Elayavalli.
		
		\begin{thm}[{See \cite[Corollary 1.4, Example 2.3(3)]{GKE2026Toeplitz}}] \label{thm: GKE Toeplitz theorem}
			Let $B^{(k)}$ and $B$ be tracial $\mathrm{C}^*$-algebras and let $b^{(k)} = (b_\omega^{(k)})_{\omega \in \Omega}$ and $b = (b_\omega)_{\omega \in \Omega}$ be self-adjoint tuples in $B^{(k)}$ and $B$ respectively with $\sup_k \norm{b_\omega^{(k)}} < \infty$ for each $\omega \in \Omega$.  Let $\eta^{(k)} = (\eta_{i,j}^{(k)})_{i,j \in I}$ and $\eta = (\eta_{i,j})_{i,j \in I}$ be tracially symmetric operator-valued covariance matrices over $B^{(k)}$ and $B$ respectively with $\sup_k \norm{\eta_{i,i}^{(k)}(1)} < \infty$ for each $i \in I$.  Let $S^{(k)}$ be a $(B^{(k)},\eta^{(k)})$ semicircular family and let $S$ be a $(B,\eta)$ semicircular family.  Assume that $(\eta^{(k)}, b^{(k)})$ converges strongly to $(\eta, b)$.  
			Then $(\eta^{(k)}, b^{(k)},S^{(k)})$ converges strongly to $(\eta,b,S)$.
		\end{thm}
		
		There are slight differences from \cite{GKE2026Toeplitz} in the setup that deserve explanation:  In \cite[Example 2.3(3), Corollary 1.4]{GKE2026Toeplitz}, the authors assume that $B$ is generated by $(b_\omega)_{\omega \in \Omega}$ as a $\mathrm{C}^*$-algebra.  To apply this to our setting, let $\widehat{b} = (f(\eta,b))_{f \in \C_I\ip{\Omega}}$ be the tuple consisting of all covariance polynomials in $(\eta,b)$ (or a spanning set if desired).  Let $B_0$ be the $\mathrm{C}^*$-algebra generated by $\widehat{b}$, which is invariant under $\eta_{i,j}$, and let $\eta_0 = \eta|_{B_0}$.  Thus, $B_0$ is generated by $\widehat{b}$.  Define $\widehat{b}^{(k)}$ and $B_0^{(k)}$ and $\eta_0^{(k)}$ analogously.  We note that since covariance polynomials are closed under composition, $(\eta_0^{(k)},\widehat{b}^{(k)})$ converges strongly to $(\eta_0,\widehat{b})$.  Moreover, $S^{(k)}$ and $S$ are also $(B_0^{(k)},\eta_0^{(k)})$ and $(B_0,\eta_0)$ semicircular families, respectively.
		
		We then to apply \cite[Corollary 1.4]{GKE2026Toeplitz} to $(\eta_0^{(k)}, \widehat{b}^{(k)})$ and $(\eta_0, \widehat{b})$.  We note that the strong convergence assumption in \cite[Example 2.3(3)]{GKE2026Toeplitz} is weaker than ours;\footnote{The two definitions of strong convergence in covariance law can be shown to be equivalent when the $\mathrm{C}^*$-algebra $B_0$ is generated by the tuple $\widehat{b}$ under consideration.  But note that in choosing $B_0$ as a $\mathrm{C}^*$-algebra of $B$, we had to arrange that it is both $\eta$-invariant and generated by $\widehat{b}$ as a $\mathrm{C}^*$-algebra, which is exactly where we use more general covariance polynomials.} it only requires that $\norm{p(\widehat{b}^{(k)}) - \eta_{0,i,j}(q(\widehat{b}^{(k)})} \to \norm{p(\widehat{b}) - \eta_{0,i,j}(q(\widehat{b}))}$ for non-commutative polynomials $p$ and $q$. This is of course satisfied under our hypotheses, and so we can apply \cite[Corollary 1.4]{GKE2026Toeplitz} to obtain that $(\widehat{b}^{(k)},S^{(k)})$ converge strongly in law to $(\widehat{b},S)$, and $(b^{(k)},S^{(k)})$ converges strongly to $(b,S)$ as asserted in Theorem \ref{thm: GKE Toeplitz theorem} above.

		The second ingredient we need, as in \cite[\S 7.2]{BBvH2023}, is Haagerup and Thorbj{\o}rnsen's linearization trick \cite[Lemma 1 and p.\ 758-760]{HT2005}, \cite[Theorem 7.7]{BBvH2023}.  We remark that to obtain more quantitative estimates for convergence one would want to proceed using a more explicit linearization such as \cite[Algorithm 4.3 and Theorem 4.9]{HMS2018}.  However, this would take us too far afield as the hypotheses we have assumed on $\eta^{(k)}$ are not quantitative anyway (nor are the results of \cite{GKE2026Toeplitz}).
		
		\begin{thm}[\cite{HT2005}] \label{thm: HT linearization}
			Let $X^{(n)} = (X_i^{(n)})_{i \in [m]}$ and $X = (X_i)_{i \in [m]}$ be tuples of self-adjoint operators in a $\mathrm{C}^*$-algebra.  Then the following are equivalent:
			\begin{enumerate}
				\item For every non-commutative polynomial $p$, we have
				\[
				\limsup_{n \to \infty} \norm{p(X^{(n)})} \leq \norm{p(X)}.
				\]
				\item For every $d \in \N$ and $a_0$, \dots, $a_m \in (\mathbb{M}_d)_{\sa}$ and $\varepsilon > 0$, we have for sufficiently large $n$
				\[
				\Spec\left(a_0 \otimes 1 + \sum_{i=1}^m a_j \otimes X_j^{(n)}\right) \subseteq N_{\varepsilon}\left(\Spec\left(a_0 \otimes 1 + \sum_{i=1}^m a_j \otimes X_j\right)\right),
				\]
				where $N_{\varepsilon}$ denotes the $\varepsilon$-neighborhood of a set.
			\end{enumerate}
		\end{thm}
		
		We next recall the key spectral estimate of Bandeira, Boedihardjo, and van Handel for showing strong convergence.  Here note that we have used the bound $\sigma_*(X) \leq v(X)$ to remove the $\sigma_*(X)$ from the statement.
		We remind the reader that in the following, we have for $X = a_0 + \sum_{j=1}^m g_j a_j$ as in the next theorem, the parameter $\sigma(X)$ is defined as
		\[
		\sigma(X)^2 = \left\|\sum_{j=1}^m a_j^2 \right\|.
		\]
		
		\begin{thm}[{\cite[Theorem 2.1]{BBvH2023}}] \label{thm: BBvH spectrum bound}
			Let $a_0$, \dots, $a_m$ be self-adjoint matrices in $\mathbb{M}_n$.  Let $g_1$, \dots, $g_m$ be independent standard normal random variables, and let $S_1$, \dots, $S_m$ be a free semicircular family.  Let
			\[
			X = a_0 + \sum_{j=1}^m g_j a_j, \qquad X_{\operatorname{free}} = b_0 \otimes 1 + \sum_{j=1}^m a_j \otimes S_j.
			\]
			Then
			\[
			\mathbb{P} \left[ \Spec(X) \subseteq \Spec(X_{\operatorname{free}}) + C \{v(X)^{1/2} \sigma(X)^{1/2} (\log n)^{3/4} + v(X)t\}[-1,1]  \right] \geq 1 - e^{-t^2},
			\]
			where $C$ is universal constant.
		\end{thm}
		
		\begin{proof}[Proof of Proposition \ref{prop: strong convergence plain}]
			Since each non-commutative polynomial $p$ in $(b,X)$, only depends on finitely many of the $b_\omega$'s, it suffices to prove strong convergence over every finite $F \subseteq I$ and finitely many coordinates of $b$, say $b_1$, \dots, $b_m$.  To apply Theorem \ref{thm: HT linearization}, fix $d \geq 1$ and $a_0$, \dots, $a_m 
			\in (\mathbb{M}_d)_{\sa}$, 
			and consider
			\[
			Y^{(k)} = a_0 \otimes 1 + \sum_{j=1}^m a_j \otimes b_j^{(k)} + \sum_{i \in F} a_i \otimes X_i^{(k)}.
			\]
			Note that $a_0 \otimes 1 + \sum_{j=1}^m a_j \otimes b_j^{(k)}$ is deterministic.  
			Set 
			\[
			\overline{Y}^{(k)} = \sum_{i \in F} a_i \otimes X_i^{(k)}
			\]
			to be the random part of $Y^{(k)}$.
			Note that $\sum_{i \in F} a_i \otimes X_i^{(k)}$ is a Gaussian random matrix, and since it is a matrix linear combination of $X_i^{(k)}$, we have
			\begin{align*}
				\sigma(\overline{Y}^{(k)}) &= \norm{\mathbb{E} (\overline{Y}^{(k)})^2}^{1/2} \leq K_1 \norm{\eta^{(k)}(1)}^{1/2} \\
				v(\overline{Y}^{(k)}) &= \norm{\Cov(\overline{Y}^{(k)})}^{1/2} \leq K_2 \norm{T^{(k)}}^{1/2},
			\end{align*}
			where $K_j$ is a constant depending only on $d$ and $(a_i)_{i \in I}$.  By our assumptions on $\norm{\eta^{(k)}(1)}$ and $\norm{T^{(k)}}$, we have 
			\[
			\varepsilon^{(k)} := v(\overline{Y}^{(k)})^{1/2} \sigma(\overline{Y}^{(k)})^{1/2} (\log n(k)d)^{3/4} + v(\overline{Y}^{(k)}) (\log n(k)d)^{3/2} \to 0.
			\]
			By Theorem \ref{thm: BBvH spectrum bound}, we have
			\[
			\mathbb{P} \left[ \Spec(\overline{Y}^{(k)}) \subseteq \Spec(\overline{Y}_{\operatorname{free}}^{(k)}) + C \varepsilon^{(k)} \right] \geq 1 - e^{-(\log n(k)d)^3},
			\]
			where $\overline{Y}_{\operatorname{free}}^{(k)}$ is obtained by replacing the classical Gaussians by $(\mathbb{M}_{n(k)},\eta^{(k)})$-semicirculars.  Note that $\sum_k e^{-(\log n(k)d)^3} \leq \sum_n e^{-(\log nd)^3} < \infty$, so by the Borel--Cantelli lemma, the containment of the spectra stated above holds for sufficiently large $k$ almost surely.
			
			Now let $X = (X_i)_{i \in I}$ be a $(B,\eta)$-semicircular family, and let
			\[
			Y = a_0 \otimes 1 + \sum_{j=1}^m a_j \otimes b_j + \sum_{i \in F} a_i \otimes X_i.
			\]
			Also, set
			\[
			Y^{(k)}_{\free} := a_0 \otimes 1 + \sum_{j=1}^m a_j \otimes b_j^{(k)} + \overline{Y}^{(k)}_{\free}.
			\]
			By Theorem \ref{thm: GKE Toeplitz theorem}, we know that $(b^{(k)},X_{\free}^{(k)})$ strongly converges to $(b,X)$.  Therefore, using Theorem \ref{thm: HT linearization}, there is some sequence $\delta^{(k)}$ tending to zero such that
			\[
			\Spec \left(Y_{\free}^{(k)} - \overline{Y}^{(k)}_{\free} \right) \subseteq \Spec \left( Y - \sum_{i \in F} a_i \otimes X_i \right) + \delta^{(k)}[-1,1].
			\]
			Combining with our earlier spectral estimates on the centered variables $\overline{Y}_{\free}^{(k)}, \overline{Y}^{(k)}$, we have almost surely for sufficiently large $k$,
			\[
			\Spec(Y^{(k)}) \subseteq \Spec(Y) + (\delta^{(k)} + C \varepsilon^{(k)}) [-1,1].
			\]
			Finally, we apply Theorem \ref{thm: HT linearization} again to conclude that $\limsup_{k \to \infty} \norm{p(b^{(k)},X_{\free}^{(k)})} \leq \norm{p(b,X)}$ for non-commutative polynomials $p$.
		\end{proof}
		
		\subsection{Approximation of covariance polynomials by polynomials in free copies} \label{subsec: averaging argument}
		
		We need to upgrade the convergence of norms from ordinary polynomials to covariance polynomials.  To this end, we want to approximate any covariance polynomial in free semicirculars by an ordinary polynomial.  We accomplish this through an averaging trick using infinitely many freely independent copies of ${X}$.
		
		\begin{prop} \label{prop: approximating covariance polynomial}
			Let $f \in \C_I\<\Omega\sqcup I\>$ be a covariance polynomial.  Then for every $\varepsilon > 0$ and $R > 0$, there exists a polynomial $g$ only depending on $f$, $\varepsilon$, and $R$ such that the following hold:
			\begin{enumerate}[(1)]
				\item Let $B$ be a tracial von Neumann algebra and $\eta: B \to B(\ell^2 I) \otimes B$ an operator-valued covariance matrix with $\norm{\eta_F(1)} \leq R$ for finite $F \subseteq I$.  Let $(b_\omega)_{\omega \in \Omega} \subset B$ with $\norm{b_\omega} \leq R$.  Let $M$ be the von Neumann algebra generated by a $(B,\eta)$-semicircular family ${X} = (X_i)_{i \in I}$.  Let $N$ be the free product over $B$ of copies of $M$ indexed by $\N_0$ and let $\iota_t: M \to N$ be the inclusion of $t$-th copy.  Then
				\begin{equation} \label{eq: approximation difference}
					\norm{g(b,(\iota_t({X}))_{t \in \N_0}) - \iota_0 \circ f(\eta,b,X)} \leq \varepsilon.
				\end{equation}
				and
				\begin{equation} \label{eq: approximation expectation}
					E_{\iota_0(M)}[g(b,(\iota_t({X}))_{t \in \N_0})] = \iota_0 \circ f(\eta,b,X).
				\end{equation}
				\item Let $\eta^{(k)}$ be a $\tr_{n(k)}$-symmetric $\mathbb{M}_{n(k)}$-valued covariance matrix, and let $X^{(k)}$ be an $\eta^{(k)}$-Gaussian matrix.  For $t \in \N_0$, let ${X}^{(k,t)}$ be independent copies of ${X}^{(k)}$. Then for deterministic matrices $(b^{(k)}_\omega)_{\omega \in \Omega}$, we have
				\begin{equation} \label{eq: approximation expectation matrix version}
					\mathbb{E} \left[ g(b^{(k)},({X}^{(k,t)})_{t \in \N_0}) \mid {X}^{(k,0)} \right] = f(\eta^{(k)},b^{(k)},{X}^{(k,0)}).
				\end{equation}
			\end{enumerate}
		\end{prop}
		
		\begin{nota}
			For a tracial von Neumann algebra $M$ and self-adjoint $X = (X_i)_{i \in I}$ in $M$, define the map $\Ad_{X}: M \to B(\ell^2 I) \otimes M$ by
			\[
			(\Ad_{X})_{i,j}(Y) = X_i Y X_j.
			\]
		\end{nota}
		
		Note that $\Ad_{X}$ is an operator-valued covariance matrix over $M$.  Our goal is to show that we can simulate the application of $\eta$ by $(1/m) \sum_{t=1}^m \Ad_{\iota_t(X)}$ where $\iota_t({X})$ are freely independent copies of ${X}$.
		
		\begin{lem}
			Consider the same setup as Proposition \ref{prop: approximating covariance polynomial} (1).  Let $(Y_t)_{t=1}^{m}$ be elements of $M$, let $i(t)$, $j(t) \in I$, and let
			\[
			Z_t = (\Ad_{\iota_t({X})})_{i(t),j(t)}(\iota_0(Y_t)).
			\]
			Then $(Z_t)_{t=1}^m$ are freely independent over $B$.
		\end{lem}
		
		\begin{proof}
			Let $\tilde{Z}_t = (\Ad_{\iota_t({X})})_{i(t),j(t)}(\iota_{m+t}(Y_t))$.  Since the index sets $\{t,m+t\}$ are disjoint for $t = 1$, \dots, $m$, the variables $\tilde{Z}_t$ are freely independent over $B$.  Hence, it suffices to show that the $B$-valued moments of $(Z_t)_{t=1}^m$ agree with those of $(\tilde{Z}_t)_{t=1}^m$.
			
			Consider a $B$-valued monomial $Z_{t(1)} b_1 \cdots Z_{t(\ell)} b_\ell$.  By the moment-cumulant formula \eqref{eq: moment-cumulant formula}, 
			\begin{align*}
				E_{B}[Z_{t(1)} b_1 \cdots Z_{t(\ell)} b_\ell] &= E_{B}[\iota_{t(1)}(X_{i(1)}) \iota_0(Y_{t(1)}) \iota_{t(1)}(X_{j(1)})b_1 \cdots \iota_{t(\ell)}(X_{i(\ell)}) \iota_0(Y_{t(\ell)}) \iota_{t(\ell)}(X_{j(\ell)})b_\ell] \\
				&= \sum_{\pi \in \mathcal{NC}(3\ell)} K_\pi(\iota_{t(1)}(X_{i(1)}), \iota_0(Y_{t(1)}), \iota_{t(1)}(X_{j(1)})b_1, \dots, \iota_{t(\ell)}(X_{i(\ell)}), \iota_0(Y_{t(\ell)}), \iota_{t(\ell)}(X_{j(\ell)})b_\ell).
			\end{align*}
			Since the images of $\iota_t$ are freely independent, the only partitions $\pi$ that contribute are those where the index $t$ is constant on each block.  In other words, every block either connects variables $\iota_t(X_i)$ with the same value of $t$, or it connects variables $\iota_0(Y_i)$.  Furthermore, as $(\iota_t(X_i))_{i \in I}$ are jointly semicircular, the partition $\pi$ must connect the variables $\iota_t(X_i)$ in pairs.
			
			We claim that $\iota_0(Y_{t(r_1)})$ and $\iota_0(Y_{t(r_2)})$ cannot be connected by $\pi$ unless $t(r_1) = t(r_2)$.  Indeed, suppose that $\iota_0(Y_{t(r_1)})$ and $\iota_0(Y_{t(r_2)})$ are connected.  If $t(r_1) \neq t(r_2)$, then there is an odd number of semicircular variables from $\iota_{t(r_1)}(M)$ in between $\iota_0(Y_{t(r_2)})$ and $\iota_0(Y_{t(m)})$ in the monomial.  Since $\pi$ is non-crossing, these semicircular variables must be matched among themselves by the partition $\pi$; but they cannot be arranged into pairs since there is an odd number.
			
			Given that $\iota_0(Y_{t(r_1)})$ and $\iota_0(Y_{t(r_2)})$ cannot be connected by $\pi$ unless $t(r_1) = t(r_2)$, the resulting contributions will be unchanged if we substitute $\iota_{m+t}(Y_t)$ instead of $\iota_0(Y_t)$.  Hence, $(Z_t)_{t=1}^m$ have the same joint moments as $(\tilde{Z}_t)_{t=1}^m$ as desired.
		\end{proof}
		
		We recall the following bound on the norm of the sum of freely independent operators.  The case $B = \C$ is proved in \cite[Lemma 3.2]{Voiculescu1986} and the proof in the amalgamated setting is the same (see e.g. \cite[Proposition 3.19]{JekelLiu2020}).
		
		\begin{lem} \label{lem: norm of free sum}
			Let $B \subseteq M$ be tracial von Neumann algebras, and $Z_1$, \dots, $Z_m \in M$ be freely independent over $B$ with $E_{B}[Z_t] = 0$ for $t = 1$, \dots, $m$.  Then 
			\[
			\norm*{\sum_{t=1}^m Z_t } \leq \norm*{ \sum_{t=1}^m Z_t^* Z_t }^{1/2} + \norm*{ \sum_{t=1}^m Z_t Z_t^* }^{1/2} + \max_{t=1,\dots,m} \norm{Z_t}.
			\]
		\end{lem}
		
		\begin{lem} \label{lem: averaging trick}
			Consider the same setup as Proposition \ref{prop: approximating covariance polynomial} (1).  Let $m \in \N$ and $i, j \in I$.  Let $Y \in M$.  Then
			\begin{equation} \label{eq: averaging expectation}
				E_{\iota_0(M)} \left[ \frac{1}{m} \sum_{t=1}^m (\Ad_{\iota_t({X})})_{i,j}[\iota_0(Y)] \right] = \iota_0 \circ \eta_{i,j}(Y),
			\end{equation}
			and
			\begin{equation} \label{eq: averaging difference}
				\norm*{\frac{1}{m} \sum_{t=1}^m (\Ad_{\iota_t({X})})_{i,j}[\iota_0(Y)] - \iota_0 \circ \eta_{i,j}(Y)} \leq \frac{24R \norm{Y}}{\sqrt{m}}
			\end{equation}
		\end{lem}
		
		\begin{proof}
			Let
			\[
			Z_t = (\Ad_{\iota_t({X})})_{i,j}[\iota_0(Y)] - \iota_0 \circ \eta_{i,j}(Y).
			\]
			For $Y' \in M$, using free independence of $\iota_t({X})$ and $M$ and the free moment-cumulant formula,
			\[
			E_{B}[\iota_t(X_i) \iota_0(Y) \iota_t(X_j) Y'] = \iota_0 \circ \eta_{i,j} \circ E_{B}(Y) E_{B}[Y'] = E_{B}[\eta_{i,j} \circ E_{B}(Y) Y']
			\]
			Since this holds for all $Y'$, we have $E_{\iota_0(M)}[\iota_t(X_i) \iota_0(Y) \iota_t(X_j)] = \iota_0 \circ \eta_{i,j} \circ E_{B}(Y)$, and averaging this over $t$ yields \eqref{eq: averaging expectation}.
			
			The above computation also implies that $E_{B}[Z_t] = 0$. 
			Furthermore, note that 
			\[
			\norm{X_i} \leq \norm{L(\xi_i)} + \norm{L(\xi_i)^*} \leq 2 \sqrt{\norm{\eta_{i,i}(1)}} \leq 2R^{1/2},
			\]
			where $L(\xi_i)$ is the creation operator used to define $X_i$, as in the construction given by \cite[\S~2.4 and Equation~(2.6)]{shlyakhtenko1999valued}.
			Thus, $\norm{\iota_t(X_i) \iota_0(Y) \iota_t(X_j)} \leq 4R \norm{Y}$, and hence $\norm{Z_t} \leq 8R \norm{Y}$.  Thus, by Lemma \ref{lem: norm of free sum}, we have
			\[
			\norm*{\sum_{t=1}^{m} (\Ad_{\iota_t({X})})_{i,j}[\iota_0(Y)] - m \iota_0 \circ \eta_{i,j}(Y)} \leq \sqrt{m(8R \norm{Y})^2} + \sqrt{m(8R \norm{Y})^2} + 8R \norm{Y} = (2\sqrt{m} + 1) 8R \norm{Y}.
			\]
			We bound this (crudely) by $24 \sqrt{m} R \norm{Y}$ and then divide both sides by $m$ to obtain \eqref{eq: averaging difference}.
		\end{proof}
		
		\begin{proof}[Proof of Proposition \ref{prop: approximating covariance polynomial}]
			We proceed by induction on the number of occurrences of $\Lambda$ in $f$. (Recall here we write $\Lambda = (\Lambda_{i,j})_{i,j \in I}$ for any formal linear map as in Definition~\ref{defi: formal covariance polynomials}). The base case is when there are zero occurrences of $\Lambda$ and we take $g = f$.  For the induction step, consider an arbitrary covariance polynomial $f$.  By linearity, it suffices to consider the case where $f$ is a monomial, i.e.\ it is $\Lambda_{\pi,\vec{\imath}}$ applied to some $*$-polynomials.  Pick some innermost block in the pair partition $\pi$, and suppose it is associated to an occurrence of $\Lambda_{i,j}$ and that the term inside $\Lambda_{i,j}$ is a non-commutative polynomial $p$.  Thus, there is a covariance polynomial $\widehat{f}$ with fewer occurrences of $\Lambda$, such that for any covariance matrix $\eta$ and inputs $X$ and $b$,
			\[
			f(\eta,b,X) = \widehat{f}(\eta,b,X,\eta_{i,j}(p(b,X))),
			\]
			where $\widehat{f}$ is linear in the variable corresponding to $\eta_{i,j}(p(b,X))$.  Note that $\eta_{i,j}(p(b,X)) \in B$ is bounded by a constant only depending on $R$ and $p$.  By the induction hypothesis, there is a polynomial $\widehat{g}$ only depending on $\widehat{f}$, $\varepsilon$, and $R$ such that
			\[
			\norm*{ \iota_0 \circ \widehat{f}(\eta,b,X,\eta_{i,j}(p(b,X))) - \widehat{g}(b, \iota_t({X})_{t=0}^{t_0},\eta_{i,j}(p(b,\iota_0({X})))) } \leq \frac{\varepsilon}{2},
			\]
			and the relations \eqref{eq: approximation expectation} and \eqref{eq: approximation expectation matrix version} are satisfied; where we have used indices $t = 0$, \dots, $t_0$ rather than all $\N_0$ because the polynomial only depends on finitely many variables.  Let $m \in \N$ to be chosen later, and let
			\[
			g(b,\iota_t({X})_{t=0}^{t_0 + m})
			= \widehat{g}\left(b, \iota_t({X})_{t=0}^{t_0}, \frac{1}{m} \sum_{t=1}^{m} \iota_{t_0 +t}(X_i) p(b,\iota_0({X})) \iota_{t_0 +t}(X_j) \right).
			\]
			Since $\widehat{g}$ is linear in the variable corresponding to $\eta_{i,j}(p(b,X))$, there is a constant $C$ depending only on $\widehat{g}$ and $R$ such that
			\begin{multline*}
				\norm*{\widehat{g}(b,\iota_t({X})_{t=0}^{t_0},\eta_{i,j}(p(b,\iota_0({X})))) - 
					\widehat{g}\left(b, \iota_t({X})_{t=0}^{t_0}, \frac{1}{m} \sum_{t=1}^m \iota_{t_0 +t}(X_i) p(b,X) \iota_{t_0 +t}(X_j)\right) } \\
				\leq C \norm*{\frac{1}{m} \sum_{t=1}^m \iota_{t_0 +t}(X_i) p(b,\iota_0({X})) \iota_{t_0 +t}(X_j) - \eta_{i,j}(p(b,\iota_0({X})))}
			\end{multline*}
			By Lemma \ref{lem: averaging trick}, we have
			\[
			\norm*{\eta_{i,j}(p(b,\iota_0({X}))) - \frac{1}{m} \sum_{t=1}^m \iota_{t_0 +t}(X_i) p({X},{b}) \iota_{t_0 +t}(X_j) } \leq \frac{24R}{\sqrt{m}} \norm{p(b,X)},
			\]
			which is bounded by $1/\sqrt{m}$ times a constant that only depends on $p$ and $R$.  By choosing $m$ large enough (only depending on $\widehat{g}$, $p$, $R$ and $\varepsilon$), we can guarantee that
			\[
			\norm*{\widehat{g}(b,\iota_t({X})_{t=1}^{t_0},\eta_{i,j}(p(b,X))) - 
				\widehat{g}\left(b,\iota_t({X})_{t=0}^{t_0}, \frac{1}{m} \sum_{t=0}^m \iota_{t_0+t}(X_i) p(b,\iota_0({X})) \iota_{t_0+t}(X_j)\right) } \leq \frac{\varepsilon}{2},
			\]
			and therefore,
			\[
			\norm*{g(b,(\iota_t({X}))_{t=0}^{t_0+m}) - f(\eta,b,(\iota_t({X}))_{t=0}^{t_0})} \leq \varepsilon.
			\]
			Thus, \eqref{eq: approximation difference} is satisfied.
			
			To verify \eqref{eq: approximation expectation}, first observe that by Lemma \ref{lem: averaging trick}
			\[
			E_{\iota_0(M)} \left[ \frac{1}{m} \sum_{t=1}^m \iota_{t_0+t}(X_i) p(b,\iota_0({X})) \iota_{t_0+t}(X_j) \right] = \eta_{i,j}(p(b,X)).
			\]
			Since the $\iota_t(M)$'s are freely independent over $B$, we also have that
			\[
			E_{\iota_0(M) \vee \dots \vee \iota_{t_0}(M)} \left[ \frac{1}{m} \sum_{t=1}^m \iota_{t_0+t}(X_i) p(b,\iota_0({X})) \iota_{t_0+t}(X_j) \right] = \eta_{i,j}(p(b,X)).
			\]
			From the bimodule property of $E_{\iota_0(M) \vee \dots \vee \iota_{t_0}(M)}$ and the fact that the variable corresponding to $\eta_{i,j}(p(b,\iota_0(X)))$ only occurs once in $\widehat{g}$, we have that
			\begin{multline} \label{eq: conditional expectation manipulation}
				E_{\iota_0(M) \vee \dots \vee \iota_{t_0}(M)} \left[ \widehat{g}\left(b, \iota_t({X})_{t=1}^{t_0}, \frac{1}{m} \sum_{t=1}^m \iota_{t_0+t}(X_i) p(b,\iota_0(X)) \iota_{t_0+t}(X_j) \right) \right] \\
				= \widehat{g}\left(b, \iota_t({X})_{t=1}^{t_0}, E_{\iota_0(M) \vee \dots \vee \iota_{t_0}(M)} \left[ \frac{1}{m} \sum_{t=1}^m \iota_{t_0+t}(X_i) p(b,\iota_0({X})) \iota_{t_0+t}(X_j) \right] \right),
			\end{multline}
			and therefore,
			\begin{equation} \label{eq: conditional expectation manipulation 2}
				E_{\iota_0(M) \vee \dots \vee \iota_{t_0}(M)}[g(b,\iota_t({X})_{t=1}^{t_0+m})] = \widehat{g}(b,\iota_t({X})_{t=1}^{t_0},\eta_{i,j}(p(b,X))).
			\end{equation}
			By applying $E_{\iota_0(M)}$ to \eqref{eq: conditional expectation manipulation 2} and using the induction hypothesis \eqref{eq: approximation expectation} for $\widehat{f}$ and $\widehat{g}$, we obtain \eqref{eq: approximation expectation} for $f$ and $g$.  This completes the proof of (1).
			
			It remains to show \eqref{eq: approximation expectation matrix version} holds.  Let $\eta^{(k)}$ and ${X}^{(k,t)}$ be as in (2).  Using the independence of the ${X}^{(k,t)}$'s and the definition of the Gaussian matrix associated to $\eta^{(k)}$, we have
			\[
			\mathbb{E}\left[ \frac{1}{m} \sum_{t=1}^m X_i^{(k,t_0+t)} p(b^{(k)},X^{(k,0)}) X_j^{(k,t_0+t)} \Big \mid ({X}^{(k,t)})_{t=1}^{t_0} \right] = \eta_{i,j}^{(k)}(p(b^{(k)},X^{(k,0)}).
			\]
			Now since $\widehat{g}$ is linear in the variable corresponding to $\eta_{i,j}^{(k)}(p(b^{(k)},X^{(k,0)}))$ and since all the other terms in $\widehat{g}$ depend only on $({X}^{(k,t)})_{t=1}^{t_0}$, we have
			\[
			\mathbb{E} \left[g(b^{(k)},({X}^{(k,t)})_{t=1}^{t_0+m}) \Big\mid ({X}^{(k,t)})_{t=0}^{t_0}  \right] = \widehat{g}(b^{(k)},({X}^{(k,t)})_{t=1}^{t_0},\eta_{i,j}^{(n)}(p(b^{(k)},{X}^{(k,0)})).
			\]
			Now taking the conditional expectation given ${X}^{(k,0)}$ on both sides and applying the induction hypothesis \eqref{eq: approximation expectation matrix version} to $\widehat{g}$ and $\widehat{f}$, we obtain \eqref{eq: approximation expectation matrix version} for $g$ and $f$ as desired, which completes the proof of (2).
		\end{proof}
		
		\subsection{Convergence of operator norms for covariance polynomials}
		
		We are almost ready to complete the proof of Theorem \ref{introthm: strong convergence}.  As the proof requires an argument to exchange limit and expectation, similar to the one in the proof of weak convergence, we first record the following lemma.
		
		\begin{lem} \label{lem: dominated convergence argument}
			Let $X^{(k)}$ be an $\eta^{(k)}$-Gaussian family of random matrices and assume that $\sup_k \norm{\eta^{(k)}(1)} < \infty$ and $\sup_k (\log n(k))^3 \norm{T^{(k)}} < \infty$.  Let $f^{(k)}$ be a real-valued function and $L$ a constant such that
			\[
			\limsup_{k \to \infty} f^{(k)}({X}^{(k)}) \leq L \text{ almost surely,}
			\]
			and suppose $M$ and $r$ are constants such that
			$|f^{(k)}({X}^{(k)})| \leq M(1 + \max_i \norm{X_i^{(k)}}^r)$.
			Then
			\[
			\limsup_{k \to \infty} \mathbb{E} f^{(k)}({X^{(k)}}) \leq L.
			\]
		\end{lem}
		
		\begin{proof}
			From Theorem \ref{thm: BBvH spectrum bound}, we have
			\[
			\mathbb{P}\left[\norm{X^{(k)}} \leq \norm{X_{\operatorname{free}}^{(k)}} + C \{v(X^{(k)})^{1/2} \sigma(X^{(k)})^{1/2} (\log n(k))^{3/4} + v(X^{(k)})t\}  \right] \geq 1 - e^{-t^2}.
			\]
			Since $\norm{X_{\operatorname{free}}^{(k)}} \leq 2 \norm{\eta^{(k)}(1)}^{1/2}$ is bounded as $k \to \infty$, and similarly $(\log n(k))^3 v(X^{(k)})^2 = (\log n(k))^3 \norm{T^{(k)}}$ and $\sigma(X^{(k)}) = \norm{\eta^{(k)}(1)}^{1/2}$ are bounded, there is a constants $C_1$ and $C_2$ such that
			\[
			\mathbb{P}\left[\norm{X^{(k)}} \leq C_1 + C_2 t \right] \geq 1 - e^{-t^2}.
			\]
			Now let $Z^{(k)} = f^{(k)}({X}^{(k)})$.  Thus,
			\[
			\mathbb{P}\left[|Z^{(k)}| \geq M(1 + (C_1 + C_2 t)^r)\right] \leq e^{-t^2}.
			\]
			Thus also for $\ell \in \N$,
			\begin{align*}
				\mathbb{E} [|Z^{(k)}| \mathbf{1}_{|Z^{(k)}| \geq M(1 + (C_1 + C_2 2^\ell)^r}] &= \sum_{j=\ell}^\infty \mathbb{E} [Z^{(k)} \mathbf{1}_{M(1 + (C_1 + C_2 2^j)^r \leq |Z^{(k)}| \leq M(1 + (C_1 + C_2 2^{j+1})^r}] \\
				&\leq \sum_{j=\ell}^\infty M(1 + (C_1 + C_2 2^{j+1}))^r e^{-(2^j)^2}\\
				&\leq \sum_{j=\ell}^\infty (C_3 2^j)^r e^{-2^{j+1}}.
			\end{align*}
			Denote $K_\ell = M(1 + (C_1 + C_2 2^\ell)^r$ and $\varepsilon_\ell = \sum_{j=\ell}^\infty (C_3 2^j)^r e^{-2^{j+1}}$.  Because the series converges, we have $\varepsilon_\ell \to 0$, and the above equation says that
			\[
			\mathbb{E} [|Z^{(k)}| \mathbf{1}_{|Z^{(k)}| \geq K_\ell}] \leq \varepsilon_\ell.
			\]
			Now we observe that
			\[
			\limsup_{k \to \infty} Z^{(k)} \mathbf{1}_{|Z^{(k)}| < K_\ell} \leq L \text{ almost surely.}
			\]
			Moreover, $Z^{(k)} \mathbf{1}_{|Z^{(k)}|< K_\ell}$ is bounded by $K_\ell$ and so by applying Fatou's lemma to $K_\ell - Z^{(k)} \mathbf{1}_{|Z^{(k)}|< K_\ell}$ we have
			\[
			\limsup_{k \to \infty} \mathbb{E}[Z^{(k)} \mathbf{1}_{|Z^{(k)}| < K_\ell}] \leq L.
			\]
			Therefore,
			\[
			\limsup_{k \to \infty} \mathbb{E}[Z^{(k)}] \leq \limsup_{k \to \infty} \mathbb{E} [|Z^{(k)}| \mathbf{1}_{|Z^{(k)}| < K_\ell}]  + \limsup_{k \to \infty}  \mathbb{E} [|Z^{(k)}| \mathbf{1}_{|Z^{(k)}| \geq K_\ell}] \leq L + \varepsilon_\ell.
			\]
			Since $\varepsilon_\ell \to 0$ as $\ell \to \infty$, we have $\limsup_{k \to \infty} \mathbb{E}[Z^{(k)}] \leq L$ as desired.
		\end{proof}
		
		\begin{thm}[Theorem \ref{introthm: strong convergence}] \label{thm: strong convergence general}
			Let $(M,\tau)$ be a tracial von Neumann algebra generated by a subalgebra $B\subset M$ and a $(B,\eta)$-semicircular family $x=(x_i)_{i\in I}$, where $\eta=(\eta_{i,j})_{i,j\in I}$ is a $\tau$-symmetric $B$-valued covariance matrix, and let $b=(b_\omega)_{\omega\in \Omega}$ be a generating tuple for $B$. For a sequence of integers $(n(k))_{k\in \N}\subset \N$, let $\eta^{(k)}=(\eta_{i,j}^{(k)})_{i,j\in I}$ be a $\tr_{n(k)}$-symmetric $\mb{M}_{n(k)}$-valued covariance matrix, let $X^{(k)}=(X_i^{(k)})_{i\in I}$ be an $\eta^{(k)}$-Gaussian family, and let $b^{(k)}=(b^{(k)}_\omega)_{\omega\in \Omega}\subset \mb{M}_{n(k)}$. Assume that:
			\begin{itemize}
				\item $\displaystyle\sup_k \|\eta_{i,j}^{(k)}\|<\infty$ for each $i,j\in I$;
				\item $\displaystyle\sup_k \|b_\omega^{(k)}\| <\infty$ for each $\omega\in \Omega$;
				\item $\displaystyle\lim_{k\to\infty} (\log n(k))^3\|T_F^{(k)}\|=0$ for each finite $F\subseteq I$, where $T_F^{(k)}$ is the Choi matrix associated to $\eta_F=(\eta_{i,j})_{i,j\in F}$;
				\item $(\eta^{(k)}, b^{(k)})$ converges strongly in covariance law to $(\eta,b)$.
			\end{itemize}
			Then $(\eta^{(k)}, b^{(k)}, X^{(k)})$ converges strongly in covariance law to $(\eta,b,x)$, both almost surely and in expectation.
		\end{thm}
		
		\begin{proof}
			First, we observe that weak convergence holds.  Indeed, we assumed strong convergence of $(\eta^{(k)},b^{(k)})$ which includes weak convergence by definition.  Hence, Theorem \ref{introthm: weak convergence} implies that $(\eta^{(k)},b^{(k)},X^{(k)})$ converges weakly to $(\eta,b,X)$ both almost surely and in expectation.
			
			Next, we prove the easier direction of strong convergence, the asymptotic lower bounds on the operator norms of the matrix models, both almost surely and in expectation.  Weak convergence implies in particular that for each covariance polynomial $f$, we have that $Z^{(k)} = f(\eta^{(k)},b^{(k)},X^{(k)})^* f(\eta^{(k)},b^{(k)},X^{(k)})$ converges in non-commutative law to $Z = f(\eta,b,X)^* f(\eta,b,X)$ almost surely, since each monomial in $Z^{(k)}$ is a covariance polynomial in $(\eta^{(k)},b^{(k)},X^{(k)})$, and there are only countably many such monomials.  Therefore, by Lemma \ref{lem: strong convergence general lower bound}, we have almost surely
			\[
			\liminf_{k \to \infty} \norm{Z^{(k)}} \geq \norm{Z}
			\]
			or
			\[
			\liminf_{k \to \infty} \norm{f(\eta^{(k)},b^{(k)},X^{(k)})} \geq \norm{f(\eta,b,X)},
			\]
			which is the desired almost sure lower bound.  For the lower bound in expectation, note that by Fatou's lemma,
			\[
			\liminf_{k \to \infty} \mathbb{E} \norm{f(\eta^{(k)},b^{(k)},X^{(k)})} \geq \mathbb{E} \liminf_{k \to \infty} \norm{f(\eta^{(k)},b^{(k)},X^{(k)})} \geq \norm{f(\eta,b,X)}. 
			\]
			
			Finally, we prove the harder direction of strong convergence, giving upper bounds on the operator norms for the matrix models.  Since we only need to deal with one covariance polynomial at a time, we can assume without loss of generality that $I$ is finite.  Let $M$ be the von Neumann algebra generated by $B$ and $X$.  Let $N$ be the free product over $B$ of $\N_0$-many copies of $M$, let $\iota_t: M \to N$ be the corresponding inclusion of the $t$-th copy, and let $S = (\iota_t(X))_{t \in \N_0}$.  Similarly, let $M^{(k)}$ be the von Neumann algebra associated to ${X}_{\free}^{(k)}$, let $N^{(k)}$ be the free product of $\N_0$-many copies over $\mathbb{M}_{n(k)}$, let $\iota_t^{(k)}$ the associated inclusion, and let $S^{(k)} = (\iota_t(X_{\free}^{(k)}))_{t \in \N_0}$.  Let $({X}^{(k,t)})_{t \in \N_0}$ be independent copies of the random matrix tuple ${X}^{(k)}$.  We assume these copies are realized on a probability space which is the product of copies of the given probability space on which the $X^{(k)}$'s live.
			
			Let $f(\eta,b,X)$ be a covariance polynomial.  It only depends on finitely many coordinates $b_1$, \dots, $b_h$, so there is some $R > 0$ with
			\[
			\sup_k \norm{\eta^{(k)}(1)} \leq R, \qquad \sup_k \norm{b_j^{(k)}} \leq R \, \text{ for all } 1 \leq j \leq h.
			\]
			Fix $\varepsilon > 0$, and let
			\[
			g(b,(\iota_t({X}))_{t \in \mathbb{N}_0})
			\]
			be the associated polynomial given by Proposition \ref{prop: approximating covariance polynomial}.  Again $g$ only depends on finitely many values of $t$, say $t = 1$, \dots, $t_0$.  Since the covariance map $\tilde{\eta}^{(k)}$ associated to $({X}^{(k,t)})_{t=1}^{t_0}$ is a direct sum of copies of $\eta^{(k)}$, the Choi matrices $\tilde{T}^{(k)}$ also satisfy $(\log n(k))^3 \norm{\tilde{T}^{(k)}} \to 0$.  As $\tilde{\eta}^{(k)}$ is given by a direct sum, the associated $\mathbb{M}_{n(k)}$-valued semicircular families are given by freely independent copies of ${X}_{\free}^{(k)}$, that is, they can be identified with $({X}_{\free}^{(k,t)})_{t=1}^{t_0}$.  We also note that $(\tilde{\eta}^{(k)}, b^{(k)})$ converge strongly in covariance law to $(\tilde{\eta}, b)$; indeed, any covariance polynomial in $(\tilde{\eta}^{(k)}, b^{(k)})$ is equal to a covariance polynomial in $(\eta^{(k)}, b^{(k)})$ since each entry of $\tilde{\eta}^{(k)}$ is either zero or equal to an entry of $\eta^{(k)}$.  Hence, Proposition \ref{prop: strong convergence plain} can be applied to the independent copies.
			
			Therefore, we have
			\[
			\limsup_{k \to \infty} \norm{g(b^{(k)},(X^{(k,t)})_{t=1}^{t_0})} \leq \norm{g(b,S)} \text{ almost surely.}
			\]
			By Proposition \ref{prop: approximating covariance polynomial}, equation \eqref{eq: approximation expectation matrix version},
			\[
			f(\eta^{(k)},b^{(k)},X^{(k,0)}) = \mathbb{E} \left[ g(b^{(k)},(X^{(k,t)})_{t=1}^{t_0}) \Big \mid X^{(k,0)} \right],
			\]
			and hence
			\[
			\norm{ f(\eta^{(k)},b^{(k)},X^{(k,0)}) } \leq \mathbb{E} \left[ \norm{g(b^{(k)},(X^{(k,t)})_{t=1}^{t_0})} \Big \mid X^{(k,0)} \right].
			\]
			We need to show that the conditional expectation on the right-hand side is bounded asymptotically by $\norm{g(b,(\iota_t(X))_{t \in \mathbb{N}_0})}$, or in other words, we must take the limit inside the conditional expectation.  By our assumption, the probability space decomposes as a product $(\Omega_1,P_1) \times (\Omega_2,P_2)$, where $X^{(k,0)}$ depends on $\omega_1 \in \Omega_1$ and $X^{(k,t)}$ depends on $\omega_2 \in \Omega_2$.  By the Fubini-Tonelli theorem, we have
			\[
			\text{for a.e. } \omega_1, \text{ for a.e. } \omega_2, \quad \limsup_{k \to \infty} \norm{g(b^{(k)}(X^{(k,t)})_{t=1}^{t_0})} \leq \norm{g(b,S)}.
			\]
			Moreover, by applying the submultiplicativity of the operator norm and the uniform boundedness of $\eta^{(k)}$, we can see that there are some constants $M$ and $\ell$ such that
			\[
			\norm{g(b^{(k)},(X^{(k,t)})_{t=1}^{t_0})} \leq M \left( 1 + \max_i \norm{X_i^{(k,0)}}^\ell + \max_{i \in I, t =1,\dots,t_0} \norm{X_i^{(k,t)}}^\ell \right).
			\]
			Therefore, for almost every $\omega_1$, the hypotheses are satisfied for applying Lemma \ref{lem: dominated convergence argument} to the random variables $(X^{(k,t)})_{t=1}^{t_0}$ on $\Omega_2$ and the function $\norm{g(b^{(k)},({X}^{(k,t)})_{t=1}^{t_0})}$.  Therefore,
			\[
			\text{for a.e. } \omega_1,\quad \limsup_{k \to \infty} \norm*{f(\eta^{(k)},b^{(k)},X^{(k,0)})} \leq \norm{g(b,S)}.
			\]
			Then by Proposition \ref{prop: approximating covariance polynomial}, equation \eqref{eq: approximation difference},
			\[
			\norm{g(b,S)} \leq \norm{f(\eta,b,X)} + \varepsilon.
			\]
			Since $\varepsilon$ was arbitrary, we have the desired almost sure upper bound for strong convergence.  For the upper bound in expectation, note from the preceding argument that
			\[
			\mathbb{E} \norm{ f(\eta^{(k)},b^{(k)},X^{(k,0)}) } \leq \mathbb{E} \norm{g(b^{(k)},({X}^{(k,t)})_{t=0}^{t_0})}.
			\]
			Moreover, the hypotheses are satisfied for applying Lemma \ref{lem: dominated convergence argument} to the random variables $(X^{(k,t)})_{t=0}^{t_0}$ and the function $\norm{g(b^{(k)},(X^{(k,t)})_{t=0}^{t_0})}$, this time viewed as a random variables on the entire probability space $\Omega_1 \times \Omega_2$ rather than on $\Omega_2$.  Lemma \ref{lem: dominated convergence argument} then yields
			\[
			\limsup_{k \to \infty} \mathbb{E} \norm{ f(\eta^{(k)},b^{(k)},X^{(k,0)}) } \leq \limsup_{k \to \infty} \mathbb{E} \norm{g(b^{(k)}({X}^{(k,t)})_{t=0}^{t_0})} \leq \norm{f(\eta,b,X)} + \varepsilon.
			\]
			Since $\varepsilon$ was arbitrary, the upper bound for strong convergence holds in expectation.
		\end{proof}

		\begin{rem}\label{rem:DG+SKE_equivalence}
			The following fact can be useful for checking the hypotheses of Theorem \ref{introthm: strong convergence}.  Fix operator-valued covariance matrices $\eta^{(k)}$ and $\eta$ over $B^{(k)}$ and $B$ respectively, and tuples $(b_\omega^{(k)})_{\omega \in \Omega}$ and $(b_\omega)_{\omega \in \Omega}$.  Then the following are equivalent:
			\begin{enumerate}[(1)]
				\item $(\eta^{(k)},b^{(k)})$ converges strongly to $(\eta,b)$.
				\item Let $Y^{(k)} = (Y_i^{(k)})_{i \in I}$ be a $(B^{(k)},\eta^{(k)})$-semicircular family and $Y = (Y_i)_{i \in I}$ a $(B,\eta)$-semicircular family.  Let $S^{(k)}$ and $S$ be tuples of freely independent copies of $Y^{(k)}$ and $Y$ respectively.  Then $(b^{(k)},S^{(k)})$ converges strongly to $(b,S)$. 
			\end{enumerate}
			Indeed, (1) $\Rightarrow$ (2) follows from Gao and Kunnawalkam Elayavalli's result Theorem~\ref{thm: GKE Toeplitz theorem} (applying it to a direct sum of copies of $\eta$ as in the proof of Theorem \ref{thm: strong convergence general}).  On the other hand, we can prove (2) $\Rightarrow$ (1) from Proposition \ref{prop: approximating covariance polynomial} as follows.  Since each covariance polynomial only depends on finitely many variables, we can assume without loss of generality that $I$ and $\Omega$ are finite.  The strong convergence assumption of (2) also implies that $\sup_k \norm{b_\omega^{(k)}} < \infty$ and $\sup_k \norm{Y_i^{(k)}} < \infty$, so also $\sup_k \norm{\eta_{i,j}^{(k)}} < \infty$.  Given $\varepsilon > 0$ and a covariance polynomial $f$ in variables corresponding to $(\eta,b)$, Proposition \ref{prop: approximating covariance polynomial} produces a plain polynomial $g$ in variables corresponding to $(b,S)$ such that $\norm{f(\eta^{(k)},b^{(k)}) - g(b^{(k)},S^{(k)})} \leq \varepsilon$ and $\norm{f(\eta,b) - g(b,S)} \leq \varepsilon$.  Since $\norm{g(b^{(k)},S^{(k)}} \to \norm{g(b,S)}$ and $\varepsilon$ was arbitrary, we obtain $\norm{f(\eta^{(k)},b^{(k)})} \to \norm{f(\eta,b)}$.  The convergence of the trace of $f(\eta^{(k)},b)$ to that of $f(\eta,b)$ follows similarly.
		\end{rem}
		
		\section{Applications} \label{sec: weighted GUE}
		
		\subsection{Continuously weighted Gaussian Wigner ensembles}
		In this section we apply Theorems~\ref{introthm: weak convergence} and \ref{introthm: strong convergence} to prove Theorem~\ref{introthm: strong convergence weighted}. As examples, we consider Gaussian band matrices and random matrix ensembles whose asymptotic limits generate an interpolated free group factor. 
		
		\begin{thm}[{Theorem~\ref{introthm: strong convergence weighted}}]\label{thm: strong convergence weighted proof}
			Equip $L^\infty[0,1]$ with the trace $\tau=\int_{[0,1]}$ and let $b=(b_\omega)_{\omega\in \Omega}\subset C[0,1]$ be a generating tuple for $L^\infty[0,1]$. Fix a family $(h_{i,j})_{i,j \in I}\subset C[0,1]^2$ such that for all $(s,t)\in [0,1]^2$:
			\begin{itemize}
				\item $h_{i,j}(s,t) = h_{j,i}(t,s)$;
				
				\item $(h_{i,j}(s,t))_{i,j \in F}\geq 0$ in $\mb{M}_{|F|}$ for all finite $F\subseteq I$.
			\end{itemize}
			Define a $\tau$-symmetric operator-valued covariance matrix $\eta = (\eta_{i,j})_{i,j\in I}$ over $L^\infty[0,1]$  by
			\[
			\eta_{i,j}(f)(s) = \int_{[0,1]} h_{i,j}(s,t)f(t)\,dt,
			\]
			and let $X=(X_i)_{i\in I}$ be a $(L^\infty[0,1],\eta)$-semicircular family. 
			
			For each $n\in \N$, let $E_n\colon \mb{M}_n\to \mb{D}_n$ be the $\tr_n$-preserving conditional expectation onto the subalgebra of diagonal matrices, and let us also identify $\mb{D}_n$ as a subalgebra of $L^\infty[0,1]$ so that
			\[
			E^{(n)}(f):= \diag\left(n \int_0^{1/n} f,\ldots , n\int_{(n-1)/n}^1 f\right) 
			\]
			gives the unique $\tau$-preserving conditional expectation onto $\mb{D}_n$. Define a $\mb{D}_n$-valued covariance matrix over $\mb{M}_n$ by $\eta^{(n)}:=(E^{(n)}\circ \eta_{i,j}\circ E_n)_{i,j\in I}$, let ${X}^{(n)}$ be an $\eta^{(n)}$-Gaussian family, and let $b^{(n)}= (E^{(n)}(b_\omega))_{\omega\in \Omega}$. 
			
			Then $(\eta^{(n)}, b^{(n)}, X^{(n)})$ converges weakly and strongly in covariance law to $(\eta,b,X)$, both almost surely and in expectation.
		\end{thm}
		
		\begin{proof}
			We will reduce the proof to applications of Theorems~\ref{introthm: weak convergence} and \ref{introthm: strong convergence}, so we begin by establishing a bound for the Choi matrices. Fixing a finite $F\subseteq I$, observe that the Choi matrix $T_F^{(n)}$ associated to $\eta_F^{(n)}$ is given by
			\begin{align*}
				T_F^{(n)} &= \sum_{i,j \in F} \sum_{k,\ell=1}^n E_{i,j} \otimes \eta_{i,j}^{(n)} (E_{k,\ell}) \otimes E_{\ell,k}^{\op}\\
				&= \sum_{i,j \in F} \sum_{k=1}^n E_{i,j} \otimes \eta_{i,j}^{(n)} (E_{k,k}) \otimes E_{k,k}^{\op}\\
				&= \sum_{i,j \in F} E_{i,j} \otimes \left[ \sum_{k=1}^n \diag \left( n \int_{0}^{1/n} \int_{(k-1)/n}^{k/n} h_{i,j}, \ldots, n \int_{(n-1)/n}^{1} \int_{(k-1)/n}^{k/n} h_{i,j} \right) \otimes E_{k,k}^{\op} \right].
			\end{align*}
			As the bracketed terms in the last expression are  diagonal matrices for each $i,j\in F$, we obtain
			\begin{align}\label{eqn: weighted GUE covariance norm}
				\|T_F^{(n)} \| \leq \sum_{i,j\in F} \frac{1}{n} \|h_{i,j}\|_\infty.
			\end{align}
			In particular, $(\log{n})^3\| T_F^{(n)}\|\to 0$ as $n\to\infty$ for all finite $F\subseteq I$.\\
			
			\noindent(\textbf{Weak Convergence}): In light of (\ref{eqn: weighted GUE covariance norm}), to apply Theorem~\ref{introthm: weak convergence} it remains to show $(\eta^{(n)},b^{(n)})$ converges in covariance law to $(\eta,b)$. We will establish the following stronger claim: for all covariance monomials $f=m_0(t) \Lambda_{\pi,\vec{\imath}}[m_1(t),\ldots, m_{2k-1}(t)]m_{2k}(t)$ one has
			\[
			\lim_{n\to\infty} \| f(\eta^{(n)},b^{(n)}) - f(\eta,b)\| =0.
			\]
			To prove this claim, we will proceed by induction on $|\pi|=k$. The base case $k=0$ corresponds to when $f=t_{\omega_1}\cdots t_{\omega_d}$ is a non-commutative monomial, and follows since $b_{\omega_j}^{(n)}=E^{(n)}(b_{\omega_j})\to b_{\omega_j}$ in operator norm by the continuity of $b_{\omega_j}$ for each $j=1,\ldots, d$. Suppose now the the claim has been verified for non-crossing pair partitions with at most $k-1$ blocks, and let $f$, $\pi$, and $\vec{\imath}$ be as above. If $\pi$ does not have the outer block $\{1,2k\}$, then $f=f_1 f_2$ where $f_1,f_2$ are covariance monomials corresponding corresponding to partitions $\pi_1,\pi_2$ with most $k-1$ blocks. Hence
			\[
			\| f(\eta^{(n)}, b^{(n)}) - f(\eta, b)\| \leq \| f_1(\eta^{(n)},b^{(n)}) - f_1(\eta, b)\| \|f_2(\eta^{(n)}, b^{(n)})\| + \| f_1(\eta,b)\| \| f_2(\eta^{(n)}, b^{(n)}) - f(\eta, b)\| \to 0
			\]
			as $n\to\infty$ by the induction hypothesis. Otherwise $\pi$ has the outer block $\{1,2k\}$, and so $f = m_0(t) \Lambda_{i,j}(f_0) m_{2k}(t)$ for some $i,j\in I$ and $f_0$ a covariance monomial corresponding to a partition $\pi_0$ with $k-1$ blocks. Hence
			\begin{align*}
				f(\eta^{(n)}, b^{(n)}) &= m_0(b^{(n)}) \eta_{i,j}^{(n)}[f_0(\eta^{(n)}, b^{(n)})] m_{2k}(b^{(n)})\\
				&=  m_0(b^{(n)}) E^{(n)}\circ\eta_{i,j}[f_0(\eta^{(n)}, b^{(n)}))] m_{2k}(b^{(n)}).
			\end{align*}
			The base case ensures the first and last factors converge to $m_0(b)$ and $m_{2k}(b)$, respectively. For the middle factor observe that
			\begin{align*}
				\| E^{(n)}\circ\eta_{i,j}[f_0(\eta^{(n)}, b^{(n)}))]& - \eta_{i,j}(f_0(\eta,b))\| \\
				&\leq \| E^{(n)}\circ \eta_{i,j} \left[f_0(\eta^{(n)}, b^{(n)})) - f_0(\eta,b) \right] \| + \| (E^{(n)} -1)[\eta_{i,j}(f_0(\eta,b))] \|\\
				&\leq \|h_{i,j}\| \|f_0(\eta^{(n)}, b^{(n)})) - f_0(\eta,b)\| + \| (E^{(n)} -1)[\eta_{i,j}(f_0(\eta,b))] \|.
			\end{align*}
			As $n\to\infty$, the first term in the last expression tends to zero by the induction hypothesis, and the second term tends to zero by the continuity of $\eta_{i,j}(f_0(\eta,b))$ as a function on $[0,1]$. Hence $f(\eta^{(n)},b^{(n)})\to f(\eta,b)$ in operator norm, completing the proof of the claim.\\

			\noindent(\textbf{Strong Convergence}): We must show the hypotheses of Theorem~\ref{introthm: strong convergence} are satisfied, but we have already seen $(\log{n})^3 \|T_F^{(n)}\| \to 0$ as $n\to\infty$ for all finite $F\subseteq I$ by (\ref{eqn: weighted GUE covariance norm}). Thus, it remains to check that $(\eta^{(n)},b^{(n)})$ converges strongly to $(\eta,b)$.  Let $\tilde{\eta}^{(n)}$ and $\tilde{\eta}$ be the direct sum of infinitely many copies of $\eta^{(n)}$ and $\eta$ respectively.  Let $S^{(n)}$ be a $(\mathbb{M}_n,\tilde{\eta}^{(n)})$-semicircular family, and let $S$ be an $(L^\infty[0,1],\tilde{\eta})$-semicircular family. Using Proposition \ref{prop: approximating covariance polynomial} to approximate covariance polynomials by ordinary non-commutative polynomials, it suffices to show that $(b^{(n)},S^{(n)})$ converges strongly to $(b,S)$ in order to establish strong convergence of $(\eta^{(n)},b^{(n)})$ to $(\eta,b)$ (see also Remark~\ref{rem:DG+SKE_equivalence}). 
			We further observe that since $b^{(n)}$ is contained in $\mathbb{D}_n$ and $\mathbb{D}_n$ is $\tilde{\eta}^{(n)}$-invariant, it follows from Lemma~\ref{lem: semicircular partition formula} that our $(\mathbb{M}_n,\tilde{\eta}^{(n)})$-semicircular family $S^{(n)}$ is also a $(\mathbb{D}_n,\tilde{\eta}^{(n)})$-semicircular family.  Additionally, since any given polynomial only depends on finitely many indices in $I$, it suffices to consider $(\mb{D}_n, \tilde{\eta}^{(n)}_F)$-semicircular families for a fixed finite $F \subseteq I \times \N$.
			
			Finally, by replacing $\eta^{(n)}$ and $\eta$ with $\tilde{\eta}^{(n)}|_F$ and $\tilde{\eta}|_F$ as above, the proof of strong convergence in Theorem \ref{thm: strong convergence weighted proof} is reduced to the following claim:
			
			\begin{clm} \label{clm: strong convergence of weighted semicirculars}
				Let $(\eta^{(n)}, b^{(n)})$ and $(\eta, b)$ be as in Theorem \ref{thm: strong convergence weighted proof} with the index set $I$ finite.  Let $S^{(n)}$ be an $(\mathbb{D}_n,\eta^{(n)})$-semicircular family and let $S$ be an $(L^\infty[0,1],\eta)$-semicircular family.  Then $(b^{(n)},S^{(n)})$ converges strongly to $(b,S)$.
			\end{clm}
			
			We prove the claim by realizing $(b^{(n)},S^{(n)})$ and $(b,S)$ on the same Hilbert space, such that they are close together in operator norm.  Let $\mathbb{D} = L^\infty[0,1]$, and recall that $(b,S)$ can be realized on an $\mathbb{D}$-valued Fock space as follows (see \cite{shlyakhtenko1999valued} for background).  View $\mathcal{H} = L^2([0,1]^2) \otimes \C^I$ as a bimodule over $\mathbb{D}$ by defining
			\[
			\varphi \cdot (f \otimes e_j) \cdot \psi := \varphi \cdot f \cdot \psi \otimes e_j.
			\]
			for $f\in L^2([0,]^2)$, $j\in I$, and $\varphi$, $\psi \in L^\infty[0,1]$, where 
			\[
			(\varphi \cdot f \cdot \psi)(s,t) = \varphi(s) f(s,t) \psi(t).
			\]
			We can then construct the relative tensor product
			\[
			\mathcal{H}^{\otimes_{\mathbb{D}} k} = \underbrace{\mathcal{H} \otimes_{\mathbb{D}} \mathcal{H} \otimes_{\mathbb{D}} \dots \otimes_{\mathbb{D}} \mathcal{H}}_k \cong L^2([0,1]^{k+1}),
			\]
			and the Fock space
			\[
			\mathcal{F}(\mathcal{H}) = \bigoplus_{k \geq 0} \mathcal{H}^{\otimes_{\mathbb{D}} k},
			\]
			where the $0$th term is $L^2[0,1]$, or equivalently the $L^2$-completion of $\mathbb{D}$ viewed as a $\mathbb{D}$-bimodule by the left and right multiplication actions.  For a vector $f \in \mathcal{H}$, the \emph{left creation operator} is the operator acting on simple tensors as
			\[
			\ell(f) (f_1 \otimes \dots \otimes f_k) = f \otimes f_1 \otimes \dots \otimes f_k;
			\]
			as noted in \cite[Equation (2.8)]{shlyakhtenko1999valued}), this produces a well-defined bounded operator on $\mathcal{F}(\mathcal{H})$ with
			\[
			\norm{\ell(f)} \leq \norm{\ip{f, f}_{\mathbb{D}}}^{1/2},
			\]
			where $\ip{\cdot,\cdot}_{\mathbb{D}}$ is the right $\mathbb{D}$-valued inner product given by
			\[
			\ip{a \otimes e_i, b \otimes e_j}_{\mathbb{D}}(s) = \delta_{i,j} \int_{[0,1]} \overline{a(s,t)} b(s,t)\,dt.
			\]
			
			For each $i,j\in I$ and $n\in \N$ set
			\[
			h_{i,j}^{(n)} := (E^{(n)}\otimes E^{(n)})(h_{i,j}),
			\]
			and let $(g_{i,j})_{i,j\in I},(g_{i,j}^{(n)})_{i,j\in I} \in \mb{M}_{|I|}\otimes C[0,1]^2$ be the respective (positive) square roots of $(h_{i,j})_{i,j\in I}, (h_{i,j}^{(n)})_{i,j\in I}\in \mb{M}_{|I|} \otimes C[0,1]^2$. Then for each $i\in I$ and $n\in \N$ define vectors in $L^2[0,1]^2\otimes \C^I$ by
			\[
			\zeta_i :=\sum_{j\in I} g_{i,j} \otimes e_j \qquad \text{ and } \qquad \zeta_i^{(n)}:= \sum_{j\in I} g_{i,j}^{(n)}\otimes e_j.
			\]
			Consider the operators on $\mathcal{F}(\mathcal{H})$ given by
			\[
			S_i:= \ell(\zeta_i)+ \ell(\zeta_i)^* \qquad \text{ and } \qquad S_i^{(n)}:= \ell(\zeta_i^{(n)}) + \ell(\zeta_i^{(n)})^* \qquad i\in F.
			\]
			Then
			\begin{align*}
				E_{L^\infty[0,1]}(X_i f X_j) &= \sum_{k\in F}  \int_{[0,1]} \overline{g_{i,k}(s,t)} f(s) g_{j,k}(s,t)\ ds\\
				&= \sum_{k\in F} \int_{[0,1]} g_{k,i}(s,t) f(s) g_{j,k}(s,t)\ ds \\
				&= \int_{[0,1]} f(s) h_{j,i}(s,t)\ ds = \int_{[0,1]} h_{i,j}(t,s) f(s) \ ds = \eta_{i,j}(f)(t),
			\end{align*}
			where the second equality used that $(g_{i,j})_{i,j\in F}$ is self-adjoint and the fourth equality used the first assumption on $h_{i,j}$. Similarly for the $S_i^{(n)}$. So, by the definition given in \cite{shlyakhtenko1999valued}, $S = (S_i)_{i \in I}$ is an $(L^\infty[0,1],\eta)$-semicircular family and $S^{(n)}=(S_i^{(n)})_{i\in I}$ is a $(\mb{D}_n, \eta^{(n)})$-semicircular family. Moreover, by \cite[Equation (2.8)]{shlyakhtenko1999valued}, for each $i\in F$ and $n\in \N$ we have 
			\begin{align*}\label{eqn:common rep norm difference}
				\| S_i - S_i^{(n)}\| &\leq 2 \norm*{ \< \zeta_i - \zeta_i^{(n)}, \zeta_i  - \zeta_i^{(n)}\>_{L^\infty([0,1]^2)} }^{\frac{1}{2}} \\
				&= 2\left\| \sum_{j\in F} \int_{[0,1]} |g_{i,j}(s,t) - g_{i,j}^{(n)}(s,t)|^2\ ds \right\|^{\frac12} \\
				&\leq 2\left(\sum_{j\in F} \|g_{i,j} - g_{i,j}^{(n)}\|_\infty^2 \right)^{\frac12} \to 0.
			\end{align*}
			Moreover, since each $b_\omega$ is continuous on $[0,1]$, we have that
			\[
			\lim_{n \to \infty} \norm{b_\omega^{(n)} - b_\omega} = 0.
			\]
			Thus, $(b^{(n)},S^{(n)})$ converges to $(b,S)$ coordinate-wise in operator norm, which of course implies that $\norm{p(b^{(n)},S^{(n)})} \to \norm{p(b,S)}$ for every non-commutative polynomial $p$.  This completes the proof of Claim \ref{clm: strong convergence of weighted semicirculars} and hence Theorem \ref{thm: strong convergence weighted proof}.
		\end{proof}
		
		\begin{rem} \label{rem: weak convergence without continuity}
			For $a\in L^\infty[0,1]$, note that the sequence $(E^{(n)}(a) )_{n\in\N}$ is uniformly bounded in operator norm and converges to $a$ in the strong operator topology. Using this observation, one can extend the weak convergence result in Theorem~\ref{thm: strong convergence weighted proof} to the case when $(h_{i,j})_{i,j\in I} \subset L^\infty[0,1]^2$ by a similar argument. However, to obtain the strong convergence result in this case would require a novel argument, if it is even true.
		\end{rem}

		\begin{ex}[Gaussian band matrices]\label{ex: Guassian_band_matrices}
			One can recover strong convergence for generalized Gaussian band random matrices (implicit in \cite{BBvH2023}, see also \cite{Au2021band}) using Theorem~\ref{introthm: strong convergence weighted}.
			Fix $\epsilon >0$, and consider the case where $I$ is a singleton. Specify a single continuous function $h \in C[0,1]^2$ such that $h(s,t)=h(t,s)\geq 0$ and
			\[
			\operatorname{supp}(h) \subseteq \left\{(s,t)\in [0,1]^2 : |s-t| < \epsilon \right\}.
			\]
			(In particular, one can do this with $h(s,t) = f(s-t)$ for an appropriately chosen $f \in C[0,1]$.)
			The covariance of the associated $\eta^{(n)}$-Gaussian $X^{(n)}$ is given by
			\[
			\Cov(X^{(n)}_{i,j}, X^{(n)}_{k,\ell}) = \mathbf{1}_{(i,j) = (k,\ell)} n \, \int_{(i-1)/n}^{i/n} \int_{(j-1)/n}^{j/n} h(s,t)\ dsdt.
			\]
			Thus $X^{(n)}_{i,j} = 0$ whenever $\left[\frac{i-1}{n}, \frac{i}{n}\right]\times \left[ \frac{j-1}{n}, \frac{j}{n}\right]$ is disjoint from $\operatorname{supp}(h)$. In particular, $X^{(n)}$ has zeroes outside of a band of size $\lceil \epsilon n \rceil$ around the diagonal, and joint covariances of entries otherwise specified by $h$ as above.
		\end{ex}
		
		\begin{ex}[Random matrix models for interpolated free group factors]\label{ex: interpolated_fgf}
			Theorem~\ref{introthm: strong convergence weighted} can be used to produce random matrix ensembles that converge weakly and strongly in covariance law to generators of interpolated free group factors. Let $I:=J_1\sqcup J_2$ be a countable index set and let $\{S_i:i\in \{0\}\sqcup I\}$ be a free semicircular family. Set $B:=W^\ast(S_0)\cong L^\infty[0,1]$ which we equip with the trace $\tau= \int_{[0,1]}$, and for each $i\in I$ let $f_i, g_i\in C[0,1]\subset B$ satisfy either
			\[
			\begin{cases}
				f_i = \overline{g_i}, & \text{if } i\in J_1\\
				\operatorname{supp}(f_i)\cap \operatorname{supp} (g_i)=\emptyset, & \text{if } i\in J_2
			\end{cases}.
			\]
			Assume that there exists at least one $i\in J_2$ with $m(\supp(f_i)\cup \supp(g_i))=1$ and $0 < m(\supp(f_i)) < 1$, where $m$ denotes the Lebesgue measure on $[0,1]$. For each $i,j\in J_1\sqcup J_2$, define $h_{i,j}\in C[0,1]^2$ by
			\[
			h_{i,j}(s,t) = \delta_{i=j} \begin{cases}
				|f_i(s) f_i(t)|^2 & \text{if }i = j \in J_1\\
				|f_i(s) g_i(t)|^2 + |g_i(s) f_i(t)|^2 & \text{if }i = j \in J_2 \\
				0, & \text{if } i \neq j.
			\end{cases},
			\]
			and note that the family $(h_{i,j})_{i,j\in I}$ satisfies the symmetry and positivity conditions in the statement of Theorem~\ref{thm: strong convergence weighted proof}. Using this family of continuous functions, define a $\tau$-symmetric operator valued covariance matrix $\eta=(\eta_{i,j})_{i,j\in I}$ on $B$ as in Theorem~\ref{thm: strong convergence weighted proof}. For each $i\in I$, letting
			\[
			X_i:=\begin{cases} f_i S_i \overline{f_i}, & \text{if }i\in J_1\\
				f_iS_ig_i+\overline{g_i}S_i\overline{f_i}, & \text{if } i\in J_2
			\end{cases},
			\] 
			we claim that $X=(X_i)_{i\in I}$ is a $(B,\eta)$-semicircular family. Let $E_B\colon W^*(S_i,i\in \{0\}\sqcup I)\to B$ be the trace-preserving conditional expectation and let $K_k$ be the associated cumulant maps for $k\in \N$. Any cumulant $K_k[X_{i_1} b_1,\ldots, X_{i_k} b_k]$ can be expanded as a sum of terms of the form $K_k[ c_1 S_{i_1} d_1, \ldots, c_k S_{i_k} d_k]$ for $c_1,d_1,\ldots, c_k, d_k\in B$, and for $k\neq 2$ these terms all vanish since $(S_i)_{i\in I}$ is a $(B,\operatorname{id}_B)$-semicircular family (see, for example, \cite[Proposition 3.8]{shlyakhtenko2000cpentropy}). The second order cumulants are therefore given by the second order moments, and these can be computed by taking their trace against an element in $B$. For $i,j\in I$ and $b_1,b_2\in B$, using the free independence of $S_i, S_j$ from $B$ (and each other if $i\neq j$), it follows that
			\[
			\tau(K_2[X_i, b_1 X_j] b_2) = \tau(E_B(X_i b_1 X_j) b_2)= \int_{[0,1]}\int_{[0,1]} h_{i,j}(s,t) b_1(t) b_2(s)\ dsdt = \tau( \eta_{i,j}(b_1) b_2), 
			\]
			and therefore $K_2[X_i, b_1 X_j] = \eta_{i,j}(b_1)$. Thus $X=(X_i)_{i\in I}$ forms a $(B,\eta)$-semicircular family, as claimed. In particular, Theorem~\ref{thm: strong convergence weighted proof} indicates how to find random matrix ensembles that converge weakly and strongly to $(b,X)$ in covariance law for any tuple $b=(b_\omega)_{\omega\in \Omega} \subset C[0,1]$ generating $B$.
			
			Finally, we claim that $W^*(B\cup \{X_i\colon i\in I\})\cong L(\mathbb{F}_t)$, where
			\[
			t = 1 + \sum_{i\in J_1} m(\supp(f_i))^2 + 2\sum_{i\in J_2}  m(\supp(f_i)) m(\supp(g_i)).
			\]
			Indeed, if $p_i := 1_{\supp(f_i)}$ and $q_i :=1_{\supp(g_i)}$, then by assumption $0<p_i=1-q_i<1$ for at least one $i\in J_2$.  Therefore $ W^*(B\cup \{ p_i S_i q_i\colon i\in I\}) \cong L(\mathbb{F}_t)$ by \cite[Definition 4.1]{Rad94} (and the discussion following it). On the other hand, $p_i,q_i\in B$ for each $i\in I$ implies $W^*(B\cup \{X_i\colon i\in I\}) = W^*( B \cup \{ f_i S_i g_i\colon i\in I\})$, and a routine approximation argument shows the latter algebra equals $W^*(B\cup \{ p_i S_i q_i\colon i\in I\})$.
		\end{ex}
		
		\subsection{Random band matrices with width approaching zero} \label{subsec: shrinking band matrix}
		
		Another related application is to show strong convergence for Gaussian band random matrices with width $k_n = n \varepsilon_n$, where $\varepsilon_n$ satisfies $\frac{(\log n)^3}{n \epsilon_n} \rightarrow 0$.  Periodic-band matrices with a shrinking width relative to $n$ have been studied in many previous works.  This includes a detailed analysis of eigenvector (de)localization and local laws (see \cite{XYYY2024,YauYin2025} and the references therein) as well as the largest eigenvalue (see \cite{Sodin2010} and the references therein).  In particular, \cite[Theorem 1.3]{Sodin2010} shows convergence of the operator norm of $X^{(n)}$ to the operator norm of the corresponding semicircular whenever the width $k_n$ of the band satisfies $k_n / \log n \to \infty$.  The bulk convergence of a single band matrix to the semicircular law when $k_n \to \infty$ was proved in \cite{BMP1991}.  Moreover, for several independent band matrices, we have weak convergence to a free semicircular family when $k_n \to \infty$ and strong convergence when $(\log n)^3 / k_n \to 0$ thanks to \cite[Theorem 2.10]{BBvH2023}.  Hence, the new contribution of Theorem \ref{thm: shrinking band} is the operator-valued strong convergence that includes the deterministic matrices modeling $C(\R / \Z)$.
		
		In our theorem, we use an $L^1$ profile $g_i$ to define the $i$th band matrix rather than simply making each entry have variance zero or a fixed constant.  The profile $g_i$ is rescaled and periodized to a function $g_{i,\varepsilon_n}$.  The setup of the matrix models is similar to Theorem \ref{thm: strong convergence weighted proof} and Example \ref{ex: Guassian_band_matrices}, but with a weight function $h_{i,j}^{(n)}(s,t) = g_{i,\varepsilon_n}(s-t)$ that now varies with $n$.
		
		The classic Gaussian band matrix whose entries have variance zero or one is essentially a special case of this construction.  Indeed, let $k_n$ be a sequence of integers with $k_n \in \{1,\dots,n\}$.  Let $\varepsilon_n = k_n / n$, and take $g_i = \frac{1}{2} \mathbf{1}_{[-1,1]}$ and   Then the $(i,j)$ entry of $X^{(n)}$ will have variance $0$ when $|i - j| > k_n$, variance $1/2k_n$ when $|i - j| < k_n$, and variance $1/4k_n$ when $|i - j| = k_n$.  This case at the edge of the band results from the fact that half of the square $[(i-1)/n,i/n] \times [(j-1)/n,j/n]$ is in the region where $g((s-t) / \varepsilon_n) = 1$.  However, if we want the variance at the edge of the band to be equal to the variance inside the band, this can be arranged by a small modification; instead of using the weight function $h^{(n)}(s,t) = g_\varepsilon(s-t)$, let $\eta^{(n)}$ be given by the weight function $h^{(n)}(s,t)$ which is constant on the union of the squares $[(i-1)/n,i/n] \times [(j-1)/n,j/n]$ where $|i - j| \leq k_n$, and zero elsewhere, and which integrates to $1$.  Our proof (see below) will still work since it only requires that $\eta^{(n)}$ is unital, $\norm{\eta^{(n)}(b^{(n)}) - b^{(n)}} \to 0$, and $\norm{T^{(n)}} (\log n)^3 \to 0$.  Our result then applies whenever $(\log n)^3 / k_n \to 0$ and $k_n / n \to 0$.
		
		The first part of the proof is similar to Theorem \ref{thm: strong convergence weighted proof}.  However, to check the strong convergence hypotheses, rather than using semicirculars as before, we argue directly that $(\eta^{(n)},b^{(n)})$ converges strongly to $(\tilde{\eta},b)$ using the properties of convolution.
		
		\begin{thm} \label{thm: shrinking band}
			Equip $B := L^\infty(\R / \Z) \cong L^\infty(\mathbb{T})$ with the trace $\tau$ given by integration with respect to Haar measure and let $b=(b_\omega)_{\omega\in \Omega}\subset C(\R/\Z)$ be a generating tuple for $B$. Fix a family $(g_i)_{i \in I}\subset L^\infty(\R)$ such that $\norm{g_i}_1 = 1$ for all $i$. Define for any $\varepsilon > 0$ the function $g_{i, \varepsilon}(x) = \sum_{k \in \mathbb{Z}} \varepsilon^{-1} g_i(\frac{x-k}{\varepsilon})$, and fix a sequence $\varepsilon_n$ such that $\frac{(\log n)^3}{n \epsilon_n} \rightarrow 0$. Define an operator-valued covariance matrix $\tilde{\eta}^{(n)} = (\tilde{\eta}_{i,j}^{(n)})_{i, j \in I}$ over $B$ by 
			\[
			\tilde{\eta}^{(n)}_{i,j}(f)(s) = \delta_{i=j} \int_{\R / \mathbb{Z}} f(s-t) g_{i, \varepsilon_n}(t) dt,
			\]
			and set $\eta^{(n)} := (E^{(n)} \circ \tilde{\eta}^{(n)} \circ E_n)$, where $E_n: \mathbb{M}_n \rightarrow \mathbb{D}_n$ is the conditional expectation onto the diagonal matrices, and we identify $\mathbb{D}_n$ as a subalgebra of $B$ with conditional expectation onto $\mathbb{D}_n$ given by 
			\[
			E^{(n)}(f):= \diag\left(n \int_0^{1/n} f,\ldots , n\int_{(n-1)/n}^1 f\right) 
			\]
			Let ${X}^{(n)}$ be an $\eta^{(n)}$-Gaussian family, and let $b^{(n)}= (E^{(n)}(b_\omega))_{\omega\in \Omega}$. 
			
			For the $\tau$-symmetric $B$-valued covariance matrix $\eta:=(\delta_{i=j} \id)_{i,j\in I}$, let $X = (X_i)_{i \in I}$ be an $(B, \eta)$-semicircular family. Then $(\eta^{(n)}, b^{(n)}, X^{(n)})$ converges weakly and strongly in covariance law to $(\eta,b,X)$, both almost surely and in expectation.
		\end{thm}
		
		\begin{proof}[Proof of Theorem \ref{thm: shrinking band}]       % For each $i$ in a fixed index set $I$, fix $g_i \in L^\infty(\mathbb{R})$ satisfying $\norm{g_i}_1 = 1$. Define for any $\varepsilon > 0$ the function $g_{i, \varepsilon}(x) = \sum_{k \in \mathbb{Z}} \varepsilon^{-1} g_i(\frac{x-k}{\varepsilon})$. 
			By small modifications of the weak convergence argument of Theorem \ref{thm: strong convergence weighted proof} along with Remark \ref{rem: weak convergence without continuity} (applied with $L^\infty(\R/\mathbb{Z})$ in place of $L^\infty[0,1]$), we obtain weak convergence in covariance law of $(\eta^{(n)}, b^{(n)}, X^{(n)})$ to $(\eta, b, X)$.
			(Notice that $\eta^{(n)}$ could be presented resembling the framework of Theorem \ref{thm: strong convergence weighted proof} by using $h_{i,j}^{(n)}(s,t) := \delta_{i = j} g_{i, \varepsilon_n} (s-t)$.)
			
			We now check the assumptions of Theorem \ref{introthm: strong convergence}.  First, similarly to (\ref{eqn: weighted GUE covariance norm}), 
			\[
			\norm{T_F^{(n)}} \leq \sum_{i \in F} \frac{1}{n} \norm{g_{i, \varepsilon_n}}_\infty = \sum_{i \in F} \frac{1}{n \varepsilon_n} \norm{g_{i}}_\infty,
			\]
			and since $\varepsilon_n$ satisfies $\frac{(\log n)^3}{n \varepsilon_n} \rightarrow 0$, we have $(\log n)^3 \norm{T_F^{(n)}} \rightarrow 0$ for every finite $F \subseteq I$.

			For the latter hypothesis, we must show that $(\eta^{(n)},b^{(n)})$ converges strongly to $(\eta,b)$.  We again view $\mathbb{D}_n \subseteq L^\infty(\R / \Z)$.  In preparation for an inductive argument on covariance polynomials, we note the following claim.
			
			\begin{clm}
				Let $c \in C(\R / \Z)$ and let $c^{(n)} \in \mathbb{D}^{(n)}$ such that $\norm{c^{(n)} - c}_{L^\infty(\R/\Z)} \to 0$.  Then $\norm{\tilde{\eta}_{j,j}^{(n)}(c^{(n)}) - c}_{L^\infty(\R/\Z)} \to 0$.
			\end{clm}
			
			\begin{proof}[Proof of claim]
				Note that
				\begin{align*}
					\norm{\tilde{\eta}_{j,j}^{(n)}(c^{(n)}) - c}_{L^\infty(\R/\Z)} &\leq \norm{\tilde{\eta}_{j,j}^{(n)}(c^{(n)} - c)}_{L^\infty(\R/\Z)} + \norm{\tilde{\eta}_{j,j}^{(n)}(c) - c}_{L^\infty(\R/\Z)} \\
					&\leq \norm{c^{(n)} - c}_{L^\infty(\R / \Z)} + \norm{g_{j,\varepsilon_n} * c - c}_{L^\infty(\R/\Z)}.
				\end{align*}
				The first term vanishes by assumption, and the last term vanishes because $\norm{g_{j, \varepsilon_n}}_1 = \norm{g_j}_1 = 1$ and $c$ is uniformly continuous (see e.g.\ \cite[Theorem 8.14]{Folland1999}).
			\end{proof}
			
			We claim that for all covariance polynomials $f$, we have
			\begin{equation} \label{eq: convergence of covariance polynomials}
				\lim_{n \to \infty} \norm{f(\eta^{(n)},b^{(n)}) - f(\tilde{\eta},b)}_{L^\infty(\R / \Z)} = 0.
			\end{equation}
			Letting $\mathcal{P}$ be the set of covariance polynomials satisfying this conclusion, it suffices to show that $\mathcal{P}$ contains the degree-$1$ monomials and is closed under linear combinations, products, and applications of the covariance symbols $\Lambda_{i,j}$.  The fact that it contains degree-$1$ monomials is clear by hypothesis, and the closure under linear combinations and products follows easily by the triangle inequality and submultiplicativity of the norm.  Now suppose that $f \in \mathcal{P}$ and we will show that $\Lambda_{i,j}(f) \in \mathcal{P}$.  If $i \neq j$, then $\eta_{i,j}^{(n)}(f(\eta^{(n)},b^{(n)})) = \tilde{\eta}_{i,j}^{(n)}(f(\tilde{\eta}^{(n)},b^{(n)})) = 0$ and $\tilde{\eta}_{i,j}(f(\tilde{\eta},b^{(n)}))  = 0$, so we are done.  If $i = j$, then we apply the preceding claim to $c^{(n)} = f(\tilde{\eta}^{(n)},b^{(n)})$ and $c = f(\tilde{\eta},b)$.  Thus, \eqref{eq: convergence of covariance polynomials} holds, so the strong convergence hypothesis of Theorem \ref{introthm: strong convergence} is satisfied, and the desired conclusions follow from that theorem.
		\end{proof}

		\bibliographystyle{amsalpha}
		\bibliography{main}
		
		\section*{Open Access and Data Statement}
		For the purpose of Open Access, the authors have applied a CC BY public copyright license to any Author Accepted Manuscript (AAM) version arising from this submission.

		Data Access Statement: Data sharing is not applicable to this article as no new data were created or analysed in this work.
		
	\end{document}